\documentclass[10pt]{amsart}
\usepackage{amsmath}
\usepackage{amssymb,amscd}
\usepackage{graphicx}

\usepackage{tikz}
\usepackage{tikz-cd}
\usepackage{tikz-3dplot}
\usetikzlibrary{shapes.geometric, calc}

\usepackage{mathdots}
\usepackage{mathrsfs}       
\usepackage{graphicx}
\usepackage{colonequals} 
\usepackage[mathscr]{euscript}

\usepackage{enumitem}

\usepackage{hyperref} 
\hypersetup{
    colorlinks=true,
    linkcolor=blue,
    urlcolor=red,
    linktoc=all
}

\usepackage{appendix}

\theoremstyle{plain}
\newtheorem{theorem}{Theorem}[section]
\newtheorem{lemma}[theorem]{Lemma}
\newtheorem{proposition}[theorem]{Proposition}
\newtheorem{corollary}[theorem]{Corollary}
\newtheorem{question}[theorem]{Question}

\newtheorem{problem}[theorem]{Problem}

\theoremstyle{definition}
\newtheorem{definition}[theorem]{Definition}
\newtheorem{example}[theorem]{Example}
\newtheorem{remark}[theorem]{Remark}

\numberwithin{equation}{section}

\def \mc{\mathcal}

\def \CC{{\mathbb{C}}}
\def \CP{\mathbb{C}{\rm P}}

\def \QQ{{\mathbb{Q}}}
\def \RR{{\mathbb{R}}}
\def \ZZ{{\mathbb{Z}}}
\def \SS{{\mathfrak{S}}}

\DeclareMathOperator{\Hess}{Hess}

\def\M{\mathfrak{F}}

\def\FA{F(I)}
\def\FB{F(J)}
\def\FAbullet{F(I_{\bullet})}
\def\CAbullet{C(I_{\bullet})}
\def\FlA{F_K(\indI)}
\def\FlB{F_K(\indJ)}

\def\Hk{H(k)}
\def\Hn{H(n)}
\def\H{H}

\def\x{x}

\def\v{\mathsf{v}}

\def\g{\mathfrak{g}}
\def\b{\mathfrak{b}}
\def\h{\mathcal{H}}

\def\d{d}

\def\nbar{\overline{[n]}}

\usetikzlibrary{arrows,snakes,backgrounds,calc}

\def\indI{I}
\def\indJ{J}

\def\An{A_{n-1}}

\tikzset{rotation45/.style={rotate=45}}
\tikzset{rotation315/.style={rotate=-45}}
\tikzset{rotation35/.style={rotate=35}}
\tikzset{rotation325/.style={rotate=325}}
\tikzset{rotation20/.style={rotate=20}}
\tikzset{rotation340/.style={rotate=340}}
\makeatletter
\@namedef{subjclassname@2020}{\textup{2020} Mathematics Subject Classification}
\makeatother

\begin{document}
\title[Toric orbifolds associated with partitioned weight polytopes]{Toric orbifolds associated with partitioned weight polytopes in classical types}

\author[T. Horiguchi]{Tatsuya Horiguchi}
\address{Osaka City University Advanced Mathematical Institute, 
Osaka 558-8585, Japan.
}
\email{tatsuya.horiguchi0103@gmail.com}

\author[M. Masuda]{Mikiya Masuda}
\address{Osaka City University Advanced Mathematical Institute, 
Osaka 558-8585, Japan.}
\email{mikiyamsd@gmail.com}

\author[J. Shareshian]{John Shareshian}
\address{Washington University, St Louis, MO 63130, USA}
\email{jshareshian@wustl.edu}

\author[J. Song]{Jongbaek Song}
\address{School of Mathematics, KIAS, Seoul 02455, Republic of Korea}
\email{jongbaek@kias.re.kr}



\abstract 
Given a root system $\Phi$ of type $A_n$, $B_n$, $C_n$, or $D_n$ in Euclidean space $E$, let $W$ be the associated Weyl group.  For a point $p \in E$ not orthogonal to any of the roots in $\Phi$, we consider the $W$-permutohedron $P_W$, which is the convex hull of the $W$-orbit of $p$.  The representation of $W$ on the rational cohomology ring $H^\ast(X_\Phi)$ of the toric variety $X_\Phi$ associated to (the normal fan to) $P_W$ has been studied by various authors.  
Let $\{s_1,\ldots,s_n\}$ be a complete set of simple reflections in $W$.  For $K \subseteq [n]$, let $W_K$ be the standard parabolic subgroup of $W$ generated by $\{s_k:k \in K\}$.  We show that the fixed subring $H^\ast(X_\Phi)^{W_K}$ is isomorphic to the cohomology ring of 
the toric variety $X_\Phi(K)$ associated to a polytope obtained by intersecting $P_W$ with half-spaces bounded by reflecting hyperplanes for the given generators of $W_K$.
By a result of Balibanu--Crooks, the cohomology rings $H^\ast(X_\Phi(K))$ are isomorphic with cohomology rings of certain regular Hessenberg varieties.
\endabstract

\maketitle

\tableofcontents

\section{Introduction}\label{intro}

Let $\Phi$ be a crystallographic root system in Euclidean space $E$.  The linear hyperplanes orthogonal to the roots in $\Phi$ determine a complete fan $\Delta_\Phi$ in $E$, which determines in turn a toric variety $X_\Phi$.  The action of the associated finite reflection group $W$ on $E$ preserves $\Delta_\Phi$ and so gives rise to an action of $W$ on $X_\Phi$, which induces a (graded) linear representation of $W$ on the rational cohomology ring $H^\ast(X_\Phi;\QQ)$.  This representation has been studied when $\Phi$ is of type $A$ by Procesi in \cite{Pro}, by Stanley in \cite{Stan} and by Stembridge in \cite{Stem1}; and when $\Phi$ is an arbitrary irreducible system by Stembridge in \cite{Stem2} and by Dolgachev and Lunts in \cite{D-L}.  The ring structure of $H^\ast(X_\Phi;\QQ)$ is determined by Danilov's Theorem (see \cite{Dan, Jur}), and was studied further by Klyachko in \cite{Kly} and by Abe in \cite{Abe}.  In \cite{Blu}, Blume provided a result that involves both the $W$-representation and the ring structure.  Namely, he shows that if $\Phi$ is of type $A$, $B$, or $C$, then the quotient space $Y\colonequals X_\Phi/W$ is a toric orbifold.  So, the fixed point subring
\begin{equation} \label{eq_isom_induced_from_Blume}
H^\ast(X_\Phi;\QQ)^W \cong H^\ast(Y;\QQ)
\end{equation}
is determined by the combinatorial structure of some complete fan.  We observe that, in general,  the usual torus action on $X_\Phi$ does not descend to an action on the quotient $X_\Phi/W$, since this action does not commute with that of $W$ on $X_\Phi$. 

Here we generalize (\ref{eq_isom_induced_from_Blume}). We consider the case where $\Phi$ belongs to any of the infinite families of irreducible root systems of types $A,B,C$ or $D$.  Let $\Sigma=\{\alpha_1,\ldots,\alpha_n\}$ be a fixed set of simple roots in $\Phi$.  For each $\alpha \in \Phi$, let $s_\alpha$ be the reflection through the hyperplane orthogonal to $\alpha$.  We write $s_i$ for $s_{\alpha_i}$.  So, $W=\langle s_i\mid i \in [n] \rangle$.  Given $K \subseteq [n]$, we define as usual the parabolic subgroup
\[
W_K\colonequals \langle s_i\mid i \in K \rangle.
\]
We study here the fixed point subring $H^\ast(X_\Phi;\QQ)^{W_K}$.  We will see that the ring is isomorphic with the cohomology ring of the toric variety associated to a fan determined by $\Phi$ and $K$, and describe this fan explicitly.

More precisely, the fan $\Delta_\Phi$ is the normal fan to a simple lattice polytope, the {\it $W$-permutohedron} $P_W$.  One obtains $P_W$ by choosing a point $p$ in the interior of one of the maximal cones of $\Delta_\Phi$ and taking the convex hull of the $W$-orbit $W(p)$.  (We will make particular choices for $p$ that suit our purposes.)  For each $k \in [n]$, we define $H(k)$ to be the hyperplane orthogonal to $\alpha_k$,
and define the closed half-space 
$$H(k)^\leq\colonequals \{x \in E \mid \alpha_k^\v(x) \leq 0\}.$$
(Here $\alpha_k^\v$ is the coroot associated to $\alpha_k$.)

Now, given $K \subseteq [n]$, we define the polytope
$$
P_W(K)\colonequals P_W \cap \bigcap_{k \in K}H(k)^\leq.
$$
We will prove that $P_W(K)$ is a simple rational polytope.  We define $X_\Phi(K)$ to be the projective toric variety associated to (the normal fan of) $P_W(K)$.  Our main theorem is as follows.

\begin{theorem} \label{main}
If $\Phi$ is an irreducible root system of type $A_n$, $B_n$, $C_n$, or $D_n$ and $K \subseteq [n]$,  then the rings $H^\ast(X_\Phi;\QQ)^{W_K}$ and $H^\ast(X_\Phi(K);\QQ)$ are isomorphic.
\end{theorem}

We will prove Theorem \ref{main} when $\Phi$ is of type A in Section \ref{sec_A_proof_of_main_thm}, and explain how to adjust the proof in types B and C and in type D, in Appendices \ref{sec_typeB} and \ref{sec_typeD}, respectively.

The variety $X_\Phi$ is a regular semisimple Hessenberg variety, as defined and studied by De Mari, Procesi and Shayman in \cite{MPS}.  In \cite{Tym}, Tymoczko defines a representation of $W$ on $H^\ast(X_\Phi;\QQ)$, known as the dot representation. 
Moreover, one can see from \cite{Tym, Teff} 
that the dot representation and the representation determined by the action of $W$ on $\Delta_\Phi$ are the same.  It was shown by Balibanu--Crooks in \cite{BalCro} and  later by Vilonen--Xue in \cite{ViXu} that, for an arbitrary regular semisimple Hessenberg variety $X$ in a flag variety ${\mathcal B}$ associated to $\Phi$, and arbitrary $K \subseteq \Sigma$, there is a regular Hessenberg variety $Y_K \subseteq {\mathcal B}$ such that the graded $\QQ$-algebras
$H^\ast(X;\QQ)^{W_K}$ and $H^\ast(Y_K;\QQ)$ are isomorphic. Therefore, in the case at hand, there is some $Y_K$ such that
\begin{equation} \label{torreg}
H^\ast(X_\Phi(K);\QQ) \cong H^\ast(Y_K;\QQ).
\end{equation}
However, it is not necessarily the case that we can choose such a $Y_K$ that is isomorphic with $X_\Phi(K)$, as we will see in Section \ref{sec_hess}.

Work of Danilov in \cite{Dan}  and that of Jurkiewicz in \cite{Jur} tells us that the face numbers of $P_W(K)$ determine the betti numbers of $X_\Phi(K)$, and vice versa.  (See Theorem \ref{theorem_cohom_toric_var} below.) This allows us two approaches to understanding these face numbers.  First, Stembridge gives in \cite{Stem2} essentially combinatorial formulas for the characters of the representations under consideration here.  Combining these formulas with Theorem \ref{main}, we obtain essentially combinatorial formulas for the face numbers.  Second, Precup gives in \cite{Pre} combinatorial formulas for the betti numbers of regular Hessenberg varieties.  Combining Precup's formulas with Theorem \ref{theorem_cohom_toric_var} and (\ref{torreg}), we obtain a priori different formulas.

We do not consider the exceptional crystallographic root systems in this paper.  We do expect that the analogue of Theorem \ref{main} holds for these systems also.  We mention this and other open questions in Section \ref{openquestions}.

\bigskip
\noindent \textbf{Acknowledgements.}  

This research is supported in part by Osaka City University Advanced Mathematical Institute (MEXT Joint Usage/Research Center on Mathematics and Theoretical Physics). 
Horiguchi is supported in part by JSPS Grant-in-Aid for Young Scientists: 19K14508.  Masada was supported in part by JSPS Grant-in-Aid for Scientific Research 19K03472 and a HSE University Basic Research Program. Shareshian was partially supported by NSF grant DMS-1518389.
Song is supported by Basic Science Research Program through the National Research Foundation of Korea (NRF) funded by the Ministry of Education (NRF-2018R1D1A1B07048480) and a KIAS Individual Grant (MG076101) at Korea Institute for Advanced Study. 

This project was initiated while Horiguchi, Masuda and Song were visiting Fields institute to attend the Thematic program on Toric Topology and Polyhedral Products. We are grateful for the support of Fields Institute and the organizers of the program.  We are grateful for the hospitality of  Washington University in St. Louis and Osaka City University Advanced Mathematical Institute (MEXT Joint Usage/Research Center on Mathematics and Theoretical Physics), where parts of this research were conducted.

\section{Toric orbifolds} \label{sec_toric_orb}
We recall here the presentation given  by Danilov  and by Jurkiewicz for the cohomology of a toric variety associated to the normal fan of a rational, simple polytope.  (See for example \cite[Chapter VII.3]{Ew} for a discussion of the relation between fans and toric varieties.)  We discuss in particular the case where the fan in question arises from a root system.
We refer the reader to \cite{Hum1} and \cite{Hum2} for relevant Lie-theoretic terminology.

\subsection{Cohomology rings of toric orbifolds} \label{subsec_cohom_toric_orb}
A rational polytope $P$ (with respect to a lattice $M$) is \emph{simple} if, for every positive integer $k$, every codimension-$k$ face of $P$ is the intersection of $k$ facets.  The normal fan $\Delta(P)$ to a simple polytope is called a \emph{simplicial fan}, and the corresponding toric variety $X(P)$ is a \emph{toric orbifold}.

Next we record a version that suits our purposes of the results of Danilov and of Jurkiewicz mentioned above.  We denote by $N$ the dual lattice of $M$ and denote by $\left<~ ,~\right>$ the standard pairing between $M$ and $N$.

\begin{theorem}\cite[Theorem~10.8 and Remark~10.9]{Dan} \label{theorem_cohom_toric_var}
Let $P$ be a simple polytope, rational with respect to the lattice $M$.  
Let $X(P)$ be the toric orbifold associated to $P$ and let $\mc{F}(P)$ be the set of facets of $P$. For each $F \in \mc{F}(P)$, let $\nu(F) \in N$ be an (inward) normal vector to $F$.  
Then
$$H^\ast(X(P)) \cong \QQ[\tau_F \mid F\in \mc{F}(P)]/\mc{K},$$
where $\deg\tau_F=2$ and $\mc{K}$ is the ideal generated by 
\begin{enumerate}
\item $\sum_{F\in \mc{F}(P)} \left<u, \nu(F)\right>\tau_F$ for all $u\in M$; and 
\item $\prod_{k=1}^r \tau_{F_{i_k}}$ for all $F_{i_1}, F_{i_2}, \dots, F_{i_r}\in \mathcal{F}(P)$ with $F_{i_1} \cap F_{i_2}\cap \dots \cap F_{i_r} =\emptyset$. 
\end{enumerate}
\end{theorem}

\subsection{Toric varieties associated to Weyl chambers} \label{subsect:toricWeyCambers}
Let $E$ be an $n$-dimensional Euclidean space with an inner product $( \ , \ )$.
Let $\Phi$ be a root system in $E$.
Let $L(\Phi) \subset E$ be the root lattice of $\Phi$ and $\hat{L}(\Phi^{\v}) \subset E^*$ its dual lattice which is called the coweight lattice of $\Phi$. 
Fix a set of simple roots $\Sigma\colonequals \{\alpha_1, \ldots, \alpha_n \} \subset E$ and its dual basis $\Sigma^\ast=\{\omega_1, \ldots, \omega_n\} \subset E^*$. Then $L(\Phi)$ and $\hat{L}(\Phi^{\v})$ are the lattices generated by $\Sigma$ and $\Sigma^\ast$, respectively. 

Let $W \leq GL(E)$ be the Weyl group, generated by reflections through hyperplanes orthogonal to simple roots. We embed $W$ in $GL(E^\ast)$, by the usual dual representation of $W$ on $E^*$.
We now consider the fan $\Delta_{\Phi}$ in $E^\ast$ with maximal cones 
\begin{equation*} \label{eq:maximalcone}
\sigma_u\colonequals \text{cone}(u^{-1}\omega_1, \ldots, u^{-1}\omega_n) \ \ \ \textrm{for} \ u\in W.
\end{equation*}
The set of primitive generators for the $1$-dimensional cones in $\Delta_\Phi$ is 
\begin{equation*} \label{eq_Phistar}
\Phi^*\colonequals \displaystyle\bigcup_{u \in W} \{u^{-1}\omega_1,u^{-1}\omega_2,\ldots,u^{-1}\omega_n\}.
\end{equation*}
Since the set $\{\omega_1, \ldots, \omega_n\}$ forms a $\ZZ$-basis of $\hat{L}(\Phi^{\v})$, the fan $\Delta_\Phi$ is nonsingular. Hence the associated toric variety $X_\Phi$ is smooth. 
For simplicity, we will write $\Delta$ and $X$ for $\Delta_\Phi$ and $X_\Phi$, respectively, from now on.

\subsection{The $W$-permutohedron}
The fan $\Delta$ is the normal fan to a simple, lattice polytope $P_W$, know as a \emph{weight polytope} or \emph{W-permutohedron}.   Fix a point ${\mathbf a} \in E$ such that $(\alpha,{\mathbf a}) \neq 0$ for all $\alpha \in \Phi$, and let $P_W$ be the convex hull of the $W$-orbit of~${\mathbf a}$,
\begin{equation} \label{eq_weight_poly}
P_W = \text{conv}\{w(\mathbf{a})\mid w\in W\} \subset E.
\end{equation}
(A short proof that $\Delta$ is the normal fan to $P_W$ appears in \cite[Appendix A]{Kam}.)

\section{The permutohedron and permutohedral variety of type $\An$}

\label{sec_perm_and_permutohedral_var}
In this section, we discuss in detail the $W$-permutohedron and its normal fan when $\Phi$ is of type $\An$.

\subsection{Permutohedron}\label{subsec_permutohedron}
Let $E$ be the subspace of $\RR^n$ defined by  
\begin{equation*}
E=\{(x_1,\ldots,x_{n}) \in \RR^{n} \mid x_1+\cdots+x_{n}=0 \}.
\end{equation*}
Let $\Phi_{\An}$ be the root system of type $\An$ in $E$ generated by simple roots 
\begin{equation}\label{eq_A_simple_roots}
\alpha_i=t_i-t_{i+1} ~ \text{ for } 1 \leq i \leq n-1,
\end{equation}
where $\{t_1,t_2,\ldots,t_{n}\}$ is the standard basis of $\RR^{n}$. The Weyl group $W_{\An}$ is the symmetric group $\SS_{n}$ and acts on $\RR^{n}$ by permuting coordinates.

Take an element $(a_1, \dots, a_n)\in E$ with $a_1<\cdots<a_n$ and define the permutohedron $P_{A_{n-1}}$ by 
\begin{equation*} \label{eq_A_permutohedron}
P_{\An} \colonequals \text{\rm conv}\{(a_{u(1)},\dots,a_{u(n)})\in E \mid u\in \mathfrak{S}_n\}. 
\end{equation*}
The polytope $P_{\An}$ is $(n-1)$-dimensional and sits in the $(n-1)$-dimensional vector space $E$.  
Here, the vertex $(a_{u(1)},\dots,a_{u(n)})$ is labeled by the one-line notation of $u\in \mathfrak{S}_n$. For instance, the vertex $(a_2,a_3,a_1)$ is labeled by $231$. (See Figure~\ref{fig_labeling_on_PAn}.)

To describe the face structure of $P_{\An}$, we introduce the following notation. 

\begin{definition}
We denote by $\M_{\An}$ the set of all nonempty proper subsets of $[n]$. 
\end{definition}

Every facet of $P_{\An}$ is of the form 
\begin{equation}\label{eq_A_facet_convex_hull}
F(\indI)\colonequals\left\{(x_1, \dots, x_{n}) \in P_{\An} ~\Big|~  \sum_{i\in I}x_i =\sum_{i\in [|I|]}a_i\right\}
\end{equation}
for $\indI \in \M_{\An}$, and every such $F(\indI)$ is a facet.

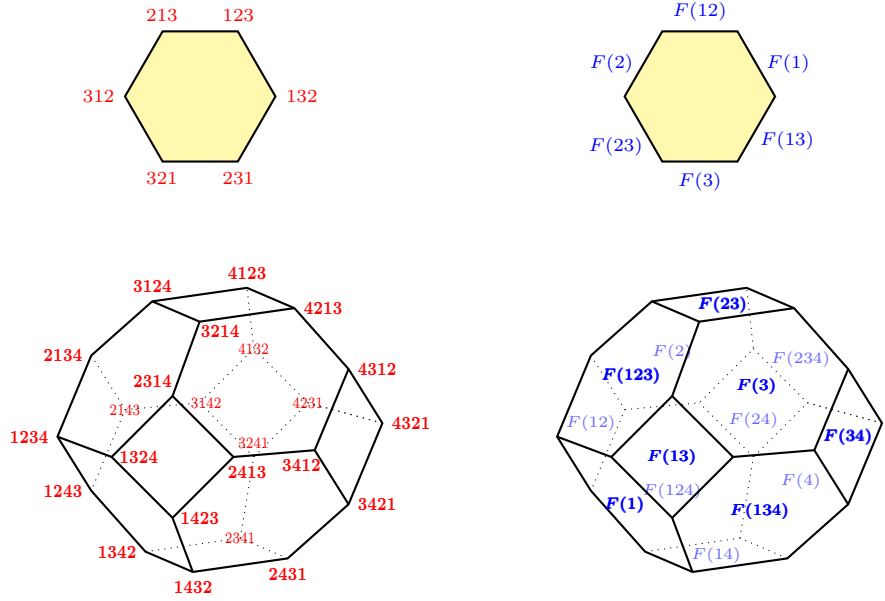
\begin{figure}
\begin{tikzpicture}[scale=0.9]
\begin{scope}[xshift=10]
\node[thick, fill=yellow!40, regular polygon, regular polygon sides=6, draw, minimum size=2cm](m) at (0,0) {};
\node[red, above] at (m.corner 1){\scriptsize$123$};
\node[red,above] at (m.corner 2){\scriptsize$213$};
\node[red,left] at (m.corner 3){\scriptsize$312$};
\node[red,below] at (m.corner 4){\scriptsize$321$};
\node[red,below] at (m.corner 5){\scriptsize$231$};
\node[red,right] at (m.corner 6){\scriptsize$132$};
\end{scope}

\begin{scope}[xshift=220]
\node[thick, fill=yellow!40, regular polygon, regular polygon sides=6, draw, minimum size=2cm](m) at (0,0) {};
\coordinate (123) at (m.corner 1); 
\coordinate (213) at (m.corner 2); 
\coordinate (312) at (m.corner 3); 
\coordinate (321) at (m.corner 4); 
\coordinate (231) at (m.corner 5); 
\coordinate (132) at (m.corner 6); 

\node[blue, above] at ($1/2*(123)+1/2*(213)$) {\scriptsize$F(12)$}; 
\node[blue, left] at ($1/2*(213)+1/2*(312)$) {\scriptsize$F(2)$}; 
\node[blue, left] at ($1/4*(312)+3/4*(321)$) {\scriptsize$F(23)$}; 
\node[blue, below] at ($1/2*(321)+1/2*(231)$) {\scriptsize$F(3)$}; 
\node[blue, right] at ($2/3*(231)+1/3*(132)$) {\scriptsize$F(13)$}; 
\node[blue, right] at ($1/2*(123)+1/2*(132)$) {\scriptsize$F(1)$}; 
\end{scope}

\begin{scope}[yshift=-200]

\begin{scope}[scale=0.2, xshift=-250]
\coordinate (1234) at (0,10); 
\coordinate (1243) at (2.5,6); 
\coordinate (1342) at (6.5,1.5); 
\coordinate (1432) at (10,0);
\coordinate (1324) at (4,8.5);
\coordinate (1423) at (8.5,4);

\coordinate (2143) at (5,12);
\coordinate (4123) at (14,21);
\coordinate (3124) at (7,20);
\coordinate (2134) at (2.5,16);
\coordinate (3142) at (10.5,12.5);
\coordinate (4132) at (14.5,16.5);

\coordinate (4312) at (21.5,15);
\coordinate (4213) at (17.5,19.5);
\coordinate (2413) at (13,8.5);
\coordinate (2314) at (8.5,13);
\coordinate (3214) at (10.5,18.5);
\coordinate (3412) at (19,9);

\coordinate (2431) at (17,1);
\coordinate (3421) at (21.5,5);
\coordinate (4321) at (24,11);
\coordinate (2341) at (13.5,2.5);
\coordinate (3241) at (14.5,8.5);
\coordinate (4231) at (18.5,12.5);

\node[red] at (5,12) {\scalebox{.8}[1.0]{\tiny{2143}}};
\node[red] at (13.5,2.5) {\scalebox{.8}[1.0]{\tiny{2341}}};
\node[red] at (11,12.5) {\scalebox{.8}[1.0]{\tiny{3142}}};
\node[red] at (14.5,16.5) {\scalebox{.8}[1.0]{\tiny{4132}}};
\node[above, red] at (14.5,8.5) {\scalebox{.8}[1.0]{\tiny{3241}}};
\node[red] at (18.5,12.5) {\scalebox{.8}[1.0]{\tiny{4231}}};

\draw[thick] (1234)--(1243)--(1342)--(1432)--(2431)--(3421)--(4321)--(4312)--(4213)--(4123)--(3124)--(2134)--cycle;
\draw[thick] (1234)--(1324)--(1423)--(1432);
\draw[thick] (1324)--(2314)--(3214)--(3124);
\draw[thick] (3214)--(4213);
\draw[thick] (2314)--(2413)--(1423);
\draw[thick] (2413)--(3412)--(3421);
\draw[thick] (4312)--(3412);

\draw[dotted] (2134)--(2143)--(1243);
\draw[dotted] (1342)--(2341)--(2431);
\draw[dotted] (2143)--(3142)--(4132)--(4123);
\draw[dotted] (4132)--(4231)--(4321);
\draw[dotted] (2341)--(3241)--(3142);
\draw[dotted] (4321)--(4231)--(3241);

\node[red, left] at (0,10) {\scalebox{.8}[1.0]{\scriptsize{\textbf{1234}}}};
\node[red, left] at (2.5,6) {\scalebox{.8}[1.0]{\scriptsize{\textbf{1243}}}};
\node[red, left] at (6.5,1.5) {\scalebox{.8}[1.0]{\scriptsize{\textbf{1342}}}};
\node[red, below] at (10,0) {\scalebox{.8}[1.0]{\scriptsize{\textbf{1432}}}};
\node[red, below] at (17,1) {\scalebox{.8}[1.0]{\scriptsize{\textbf{2431}}}};
\node[red, right] at (21.5,5) {\scalebox{.8}[1.0]{\scriptsize{\textbf{3421}}}};
\node[red, right] at (24,11) {\scalebox{.8}[1.0]{\scriptsize{\textbf{4321}}}};
\node[red, right] at (21.5,15) {\scalebox{.8}[1.0]{\scriptsize{\textbf{4312}}}};
\node[red, right] at (17.5,19.5) {\scalebox{.8}[1.0]{\scriptsize{\textbf{4213}}}};
\node[red, above] at (14,21) {\scalebox{.8}[1.0]{\scriptsize{\textbf{4123}}}};
\node[red, above] at (7,20) {\scalebox{.8}[1.0]{\scriptsize{\textbf{3124}}}};
\node[red, left] at (2.5,16) {\scalebox{.8}[1.0]{\scriptsize{\textbf{2134}}}};
\node[red] at (6,8.5) {\scalebox{.8}[1.0]{\scriptsize{\textbf{1324}}}};
\node[red] at (10.5,4) {\scalebox{.8}[1.0]{\scriptsize{\textbf{1423}}}};
\node[red, below] at (14,8.5) {\scalebox{.8}[1.0]{\scriptsize{\textbf{2413}}}};
\node[red, above] at (7,13) {\scalebox{.8}[1.0]{\scriptsize{\textbf{2314}}}};
\node[red] at (12,17.8) {\scalebox{.8}[1.0]{\scriptsize{\textbf{3214}}}};
\node[red, below] at (18,9) {\scalebox{.8}[1.0]{\scriptsize{\textbf{3412}}}};
\end{scope}

\begin{scope}[scale=0.2, xshift=800]

\coordinate (1234) at (0,10); 
\coordinate (1243) at (2.5,6); 
\coordinate (1342) at (6.5,1.5); 
\coordinate (1432) at (10,0);
\coordinate (1324) at (4,8.5);
\coordinate (1423) at (8.5,4);

\coordinate (2143) at (5,12);
\coordinate (4123) at (14,21);
\coordinate (3124) at (7,20);
\coordinate (2134) at (2.5,16);
\coordinate (3142) at (10.5,12.5);
\coordinate (4132) at (14.5,16.5);

\coordinate (4312) at (21.5,15);
\coordinate (4213) at (17.5,19.5);
\coordinate (2413) at (13,8.5);
\coordinate (2314) at (8.5,13);
\coordinate (3214) at (10.5,18.5);
\coordinate (3412) at (19,9);

\coordinate (2431) at (17,1);
\coordinate (3421) at (21.5,5);
\coordinate (4321) at (24,11);
\coordinate (2341) at (13.5,2.5);
\coordinate (3241) at (14.5,8.5);
\coordinate (4231) at (18.5,12.5);

\node[blue!60] at ($0.5*(1234)+0.5*(2143)$) {\tiny$F(12)$};
\node[blue!60] at ($0.5*(4132)+0.5*(4312)$) {\tiny$F(234)$};
\node[blue!60] at ($0.5*(1342)+0.5*(2431)$) {\tiny$F(14)$};
\node[blue!60] at ($0.5*(3241)+0.5*(3421)$) {\tiny$F(4)$};
\node[blue!60] at ($0.5*(2134)+0.5*(4132)$) {\tiny$F(2)$};
\node[blue!60, below] at ($0.5*(1243)+0.5*(3241)$) {\tiny$F(124)$};
\node[blue!60, below] at ($0.5*(4132)+0.5*(3241)$) {\tiny$F(24)$};

\draw[thick] (1234)--(1243)--(1342)--(1432)--(2431)--(3421)--(4321)--(4312)--(4213)--(4123)--(3124)--(2134)--cycle;
\draw[thick] (1234)--(1324)--(1423)--(1432);
\draw[thick] (1324)--(2314)--(3214)--(3124);
\draw[thick] (3214)--(4213);
\draw[thick] (2314)--(2413)--(1423);
\draw[thick] (2413)--(3412)--(3421);
\draw[thick] (4312)--(3412);

\draw[dotted] (2134)--(2143)--(1243);
\draw[dotted] (1342)--(2341)--(2431);
\draw[dotted] (2143)--(3142)--(4132)--(4123);
\draw[dotted] (4132)--(4231)--(4321);
\draw[dotted] (2341)--(3241)--(3142);
\draw[dotted] (4321)--(4231)--(3241);

\node[blue] at ($0.5*(2134)+0.5*(2314)$) {\tiny$\pmb{F(123)}$};
\node[blue] at ($0.5*(1234)+0.5*(1432)$) {\tiny$\pmb{F(1)}$};
\node[blue] at ($0.5*(1324)+0.5*(2413)$) {\tiny$\pmb{F(13)}$};
\node[blue] at ($0.5*(1423)+0.5*(3421)$) {\tiny$\pmb{F(134)}$};
\node[blue] at ($0.5*(3214)+0.5*(3412)$) {\tiny$\pmb{F(3)}$};
\node[blue] at ($0.5*(3412)+0.5*(4321)$) {\tiny$\pmb{F(34)}$};
\node[blue] at ($0.5*(3124)+0.5*(4213)$) {\tiny$\pmb{F(23)}$};

\end{scope}

\end{scope}

\end{tikzpicture}
\caption{$P_{A_2}\subset \RR^3$ and $P_{A_3}\subset \RR^4$, and their labeling on facets.}
\label{fig_labeling_on_PAn}
\end{figure}

\begin{lemma}\label{lem_A_facets_intersecting}
For $\indI, \indJ \in \M_{\An}$, the facets $F(\indI)$ and $F(\indJ)$ of $P_{\An}$ intersect nontrivially if and only if $\indI \subset \indJ$ or $\indJ \subset \indI$. 
\end{lemma}
\begin{proof}
Assume that $\indI \subset \indJ$. Then, 
\[
\{(x_1, \dots, x_{n})\in F(\indJ) \mid \sum_{i\in I}x_i=\sum_{i\in [|\indI|]}a_i\} \subset F(\indI)\cap F(\indJ).
\]
Conversely, suppose that $\indI \not\subset \indJ$ and $\indJ \not\subset \indI$. Without loss of generality,  we may assume that $|\indI| \leq |\indJ|$. Now, we choose an element $i\in \indI \setminus \indJ$. Then, the $i$-th coordinate of any vertex of $F(\indI)$ is an element of $\{a_k\mid k\in [|\indI|]\}$, while the $i$-th coordinate of any vertex of $F(\indJ)$ is an element of $\{a_k\mid k\in [n]\setminus [|\indJ|]\}$. Since $|\indI| \leq |\indJ|$, two sets $[|\indI|]$ and $[n]\setminus [|\indJ|]$ have empty intersection, which implies that $F(\indI)$ and $F(\indJ)$ have no common vertex. Hence, the lemma follows. 
\end{proof}

Let $\{\epsilon_1, \dots, \epsilon_{n}\}\subset (\RR^{n})^\ast$ be the basis dual to the standard basis $\{t_1, \dots, t_{n}\}$ of $\RR^{n}$. Then $E^\ast=(\RR^n)^\ast/\RR\left<\epsilon_1+\ldots + \epsilon_n\right>$. We denote by $e_i\in E^\ast$ the image of $\epsilon_i$ under the projection $(\RR^{n})^* \to E^*$. The coweights, namely the duals of simple roots \eqref{eq_A_simple_roots}, are 
\[
\omega_i=e_1+\cdots+e_i ~\text{ for } 1 \leq i \leq n-1
\]
The coweight lattice $\hat{L}(\Phi_{\An}^{\v})$ is generated by $\omega_1,\dots,\omega_{n-1}$.  Moreover, an (inward) normal vector to $F(I)$, regarded as an element of $\hat{L}(\Phi_{\An}^{\v})$, is given by  
\begin{equation} \label{eq:A_facet_normal}
e_I\colonequals\sum_{i\in I}e_i.
\end{equation}

Since $P_{\An}$ is a simple polytope, any codimension-$\d$ face $F$ of $P_{\An}$ is the intersection of $\d$ facets of $P_{\An}$, say $F(I_1)\cap  \cdots \cap  F(I_\d)$. Hence, Lemma~\ref{lem_A_facets_intersecting} implies that every face $F$ of $P_{\An}$ is determined by a nested sequence 
\begin{equation}\label{eq_Anested_seq}
I_\bullet \colonequals (I_1 \subsetneq I_2\subsetneq \cdots  \subsetneq I_\d)
\end{equation}
from $\M_{\An}$.
Such a face $F$ will be denoted by $F(I_\bullet)\colonequals F(I_1) \cap \dots \cap F(I_\d)$.

We write $X_{\An}$ for the toric variety associated with the permutohedron $P_{\An}$. This variety is often called the \emph{permutohedral variety}. We refer to \cite{MPS, Pro, Huh}. Since $P_{\An}$ is a simple lattice polytope, we may apply Theorem~\ref{theorem_cohom_toric_var} to $X_{\An}$, which gives the following presentation. 

\begin{proposition} \label{prop:A_cohom_of_perm}
The cohomology ring of $X_{\An}$ is given by 
\begin{equation}\label{eq_A_cohom_of_perm}
H^\ast(X_{\An}) \cong \QQ\big[\tau_\indI \mid \indI \in \M_{\An} \big]/\left(\mc{I}+\mc{J}\right)
\end{equation}
with $\deg\tau_\indI=2$, where
\begin{enumerate}
\item $\mc{I}=\big(\sum_{\indI \in \M_{\An} } \langle u, e_{\indI} \rangle \tau_{\indI}\mid u\in L(\Phi_{\An}) \big)$ and
\item $\mathcal{J}=\big( \tau_I \tau_J \mid I, J \in\M_{\An} \text{ with } I\not\subset J \text{ and } J \not\subset I\big)$.
\end{enumerate}
Here, $e_I$ is the element of the coweight lattice $\hat{L}(\Phi_{\An}^\v)$ defined in \eqref{eq:A_facet_normal} and 
$\left< \ , \ \right>: L(\Phi_{\An}) \times \hat{L}(\Phi_{\An}^{\v}) \to \RR$ denotes the natural paring between the root lattice and the coweight lattice. 
\end{proposition}

\begin{remark}
As $P_{\An}$ is a flag polytope, the generators of $\mathcal{J}$ can be taken to be quadratic forms as in (2) above.
\end{remark}

\subsection{The ring of invariants.}\label{subsec_ring_of_inv_A}

The action of $\SS_{n}$ on the permutohedron $P_{\An}$ induces an action of $\SS_{n}$ on $H^\ast(X_{\An})$.  We now investigate the action of parabolic subgroups of $\SS_{n}$ on $H^\ast(X_{\An})$. 

Let $K$ be a subset of $[n-1]$.  We will write $\SS_K$, rather than $W_K$ as above, for the subgroup of $\SS_n$ generated by $\{s_k:k \in K\}$.
Given $\indI \in \M_{\An}$, we write $\SS_K^\indI$ for the setwise stabilizer of $\indI$ in $\SS_K$. For a non-negative integer $m$, we consider the sum of the $\SS_K$-orbit of $\tau_\indI^m\in H^*(X_{\An})$: 
\begin{equation*} \label{eq_tau_orbit}
[\tau_\indI^m]\colonequals \sum_{w\in \SS_K/\SS_K^\indI}\tau_{w(\indI)}^m\in H^\ast(X_{\An})^{\SS_K}.
\end{equation*}
As $\tau_{u(I)}\tau_{v(I)} \in \mathcal{J}$ whenever $u \neq v$ in $\mathfrak{S}_K/\mathfrak{S}_K^I$, we see that
\begin{equation} \label{eq_tau_Am}
[\tau_\indI^m]=\sum_{w\in \SS_K/\SS_K^\indI}\tau_{w(\indI)}^m=\left( \sum_{w\in \SS_K/\SS_K^\indI}\tau_{w(\indI)}\right)^m =[\tau_\indI]^m,
\end{equation}
where $[\tau_\indI]\colonequals[\tau_\indI^1]$. 

The action of $\SS_{n}$ on $[n]$ determines an action on the set of nested sequences of subsets of $[n]$.  The correspondence between codimension-$d$ faces of $P_{\An}$ and nested sequences from \eqref{eq_Anested_seq} is equivariant with respect to the respective $\SS_{n}$ actions.  
We write $\SS_K^{\indI_\bullet}$ for the subgroup of $\SS_K$ which fixes a nested sequence $\indI_\bullet$. For a non-negative integral vector $m=(m_1,\dots,m_\d)$, we denote the product $\prod_{i=1}^\d\tau_{\indI_i}^{m_i}$ by $\tau_{\indI_\bullet}^m$ and consider the sum of its $\SS_K$-orbit:
\begin{equation*}
[\tau_{\indI_\bullet}^m]\colonequals \sum_{w\in \SS_K/\SS_K^{\indI_\bullet}}\tau_{w(\indI_\bullet)}^m\in H^\ast(X_{\An})^{\SS_K}.
\end{equation*} 

\begin{lemma} \label{lem_invariants}
Let $\indI_\bullet$ be a nested sequence as in \eqref{eq_Anested_seq} and $m=(m_1,\dots,m_\d)$ a vector of non-negative integers. Then $[\tau_{\indI_\bullet}^m]=\prod_{i=1}^\d [\tau_{\indI_i}]^{m_i}.$
\end{lemma}
\begin{proof}
It suffices to prove 
\[
[\tau_{\indI_\bullet}^m]=\prod_{i=1}^\d [\tau_{\indI_i}^{m_i}]
\]
by \eqref{eq_tau_Am}. The coefficients of monomials in $\tau_\indJ$'s with $J\in \M_{\An}$ are all one on both sides above and any monomial on the left hand side appears on the right hand side. So, it suffices to show that any monomial on the right hand side appears on the left hand side. A monomial on the right hand side is of the form 
\[
\prod_{i=1}^\d\tau_{w_i(\indI_i)}^{m_i} 
\]
with $w_i \in \SS_K$ for $i=1,\dots,\d$. Therefore, it suffices to show the existence of $w\in \SS_K$ such that $w\indI_i=w_i\indI_i$ for all $i$. 

For simplicity, we consider first the special case where $\SS_K=\SS_{n}$. We have 
\[
\begin{split}
&\indI_1\subsetneq \indI_2\subsetneq \dots\subsetneq \indI_\d,\\
&w_1\indI_1\subsetneq w_2\indI_2\subsetneq \dots\subsetneq w_\d\indI_\d. 
\end{split}
\]
Then it is clear that there exists $w\in \SS_{n}$ such that 
\[
w(\indI_i\backslash \indI_{i-1})=w_i\indI_i\backslash w_{i-1}\indI_{i-1} \quad (i=1,\dots,\d)
\]
where $\indI_0=\emptyset$, but this means that $w\indI_i=w_i\indI_i$ for all $i$. 

The argument for the general case is essentially same as that for the special case above. 
We consider the $\SS_K$-orbit decomposition of $[n]$,
\begin{equation*}\label{eq_A_SK_orbit_decomp}
[n]=N_1 \sqcup N_2 \sqcup \cdots \sqcup N_s.
\end{equation*}
This induces a decomposition of $\indI_i$ 
\[
\indI_i=\indI_i^{(1)} \sqcup \indI_i^{(2)} \sqcup \dots \sqcup \indI_i^{(s)}, 
\]
where $\indI_i^{(j)}\colonequals\indI_i\cap N_j$. Then, we can write $w_i=(w_i^{(1)},\dots,w_i^{(s)})$ where $w_i^{(j)} \in \SS_{n}$ fixes $[n] \setminus N_j$ pointwise.
For $j=1,\dots,s$, we have 
\[
w_i\indI_i=w_i^{(1)}\indI_i^{(1)} \sqcup w_i^{(2)}\indI_i^{(2)} \sqcup \dots \sqcup w_i^{(s)}\indI_i^{(s)}
\]
and 
\begin{align*}
\indI_i\backslash \indI_{i-1}&=( \indI_i^{(1)}  \backslash \indI_{i-1}^{(1)} ) \sqcup \dots\sqcup ( \indI_i^{(s)} \backslash \indI_{i-1}^{(s)} ),\\
w_i\indI_i\backslash w_{i-1}\indI_{i-1}&=(w_i^{(1)}\indI_i^{(1)} \backslash w_{i-1}^{(1)}\indI_{i-1}^{(1)} )\sqcup \dots\sqcup 
(w_i^{(s)}\indI_i^{(s)} \backslash w_{i-1}^{(s)}\indI_{i-1}^{(s)}).
\end{align*}
For each $j=1,\dots,s$, there exists a permutation $w^{(j)}$ on $N_j$ such that 
\[
w^{(j)}( \indI_i^{(j)}  \backslash \indI_{i-1}^{(j)} )=w_i^{(j)}\indI_i^{(j)} \backslash w_{i-1}^{(j)}\indI_{i-1}^{(j)} \quad (i=1,\dots,\d).
\]
Therefore, $w=(w^{(1)},\dots,w^{(s)})\in \SS_K$ is the desired element. 
\end{proof}

\begin{remark}
For any subgroup $H$ of $\SS_{n}$, one can consider the sum of the $H$-orbit of $\tau_{\indI_{\bullet}}^m$ but Lemma~\ref{lem_invariants} does not hold in general. 
For instance, if we take $H$ to be the alternating subgroup of $\SS_3$, then the sum of the $H$-orbit of $\tau_1\tau_{12}$ is 
\begin{equation} \label{eq_27-2}
\tau_1\tau_{12}+\tau_2\tau_{23}+\tau_3\tau_{13}.
\end{equation}
On the other hand, the sum of the $H$-orbit of $\tau_1$ is $\tau_1+\tau_2+\tau_3$ and that of $\tau_{12}$ is $\tau_{12}+\tau_{23}+\tau_{13}$ and their product is 
\[
\tau_1\tau_{12}+\tau_1\tau_{13}+\tau_2\tau_{12}+\tau_2\tau_{23}+\tau_3\tau_{23}+\tau_3\tau_{13}, 
\]
which does not coincide with \eqref{eq_27-2}. 
\end{remark}

As a consequence of Lemma~\ref{lem_invariants}, we have the following result. 
\begin{proposition} \label{prop_A_generated_by_deg2}
The ring of invariants $H^\ast(X_{\An})^{\SS_K}$ is generated by $\{[\tau_{\indI}] \mid \indI \in \M_{\An}\}$. 
\end{proposition}
\begin{proof}
Let $T$ be the dense torus in $X_{\An}$ corresponding to the origin in the normal fan to $P_{\An}$. By \cite[Theorem 8]{BDCP}, the $T$-equivariant cohomology ring of $X_{\An}$ is presented as 
\[
H_T^\ast(X_{\An}) =\QQ\big[\tau_\indI\mid \indI\in \M_{\An} \big]/(\tau_\indI\tau_\indJ\mid \indI\not\subset \indJ \text{ and }  \indJ\not\subset \indI).
\]
Here, we regard $\tau_\indI$ as an element of $H^2_T(X_{\An})$. 
From this explicit presentation, one can see that the monomials $\tau_{\indI_\bullet}^m$ form an additive basis of $H_T^\ast(X_{\An})$. 

Let $x$ be an arbitrary element of $H_T^\ast(X_{\An})^{\SS_K}$. We represent $x$ as a polynomial in $\tau_\indI$'s, say $p$, and let $\tau_{\indI_\bullet}^m$ be a monomial which appears in the polynomial $p$ with a nonzero coefficient, say $c$. 
Since $x$ is $\SS_K$-invariant, any monomial in the $\SS_K$-orbit of $\tau_{\indI_\bullet}^m$ appears in the polynomial $p$. Therefore, the number of monomials which appear in $p-c[\tau_{\indI_\bullet}^m]$ is strictly smaller than that in $p$. Since $p-c[\tau_{\indI_\bullet}^m]$ is again $\SS_K$-invariant, we can apply the same argument to $p-c[\tau_{\indI_\bullet}^m]$ unless it is zero. Repeating this procedure, we finally reach zero and this shows that $x$ is a linear combination of $[\tau_{\indI_\bullet}^m]$'s and hence a polynomial in $[\tau_\indI]$'s by Lemma~\ref{lem_invariants}. Therefore $H_T^\ast(X_{\An})^{\SS_K}$ is generated by $[\tau_\indI]$'s for $\indI \in \M_{\An}$. 
The natural homomorphism $H_T^\ast(X_{\An})\to H^\ast(X_{\An})$ is surjective and $\SS_K$-equivariant. Therefore its restriction $H_T^\ast(X_{\An})^{\SS_K}\to H^\ast(X_{\An})^{\SS_K}$ is also surjective. This fact and the above observation prove the proposition. 
\end{proof}

\section{Toric varieties associated with partitioned permutohedra}\label{sec_part_perm}
\subsection{Partitioned permutohedra}

If $\Phi$ is of type $A_{n-1}$, then for every $k \in [n-1]$, the half-space $\Hk^\leq$ satisfies
\begin{align*} \label{eq_Hk}
H(k)^\leq= \{ x \in E \mid x_k \leq x_{k+1}\}
\end{align*}
and the partitioned permutohedron $P_{\An}(K)$ associated with a subset $K\subset[n-1]$ is defined by
\[
P_{\An}(K)\colonequals P_{\An}\cap \bigcap_{k\in K}H(k)^\leq.
\] 
See Figure \ref{fig_part_perm_A} for several examples of partitioned permutohedra of types $A_2$ and $A_3$. 

\begin{figure}
\begin{tikzpicture}[scale=0.3]

\begin{scope}
\node[dotted, thick, regular polygon, regular polygon sides=6, draw, minimum size=1.8cm](m) at (0,0) {};
\coordinate (123) at (m.corner 1); 
\coordinate (213) at (m.corner 2); 
\coordinate (312) at (m.corner 3); 
\coordinate (321) at (m.corner 4); 
\coordinate (231) at (m.corner 5); 
\coordinate (132) at (m.corner 6); 

\draw[thick, fill=yellow!40] (123)--(132)--(231)--($1/2*(321)+1/2*(231)$)--($1/2*(123)+1/2*(213)$)--cycle;

\draw[dashed, red, thick] ($4/5*(123)+4/5*(213)$)--($4/5*(321)+4/5*(231)$);
\node[below] at ($4/5*(321)+4/5*(231)$) {\small$P_{A_2}(\{1\})$};
\end{scope}

\begin{scope}[xshift=370]
\node[dotted, thick, regular polygon, regular polygon sides=6, draw, minimum size=1.8cm](m) at (0,0) {};
\coordinate (123) at (m.corner 1); 
\coordinate (213) at (m.corner 2); 
\coordinate (312) at (m.corner 3); 
\coordinate (321) at (m.corner 4); 
\coordinate (231) at (m.corner 5); 
\coordinate (132) at (m.corner 6); 

\draw[thick, fill=yellow!40] (123)--(213)--(312)--($1/2*(312)+1/2*(321)$)--($1/2*(123)+1/2*(132)$)--cycle;

\draw[dashed, red, thick] ($4/5*(123)+4/5*(132)$)--($4/5*(312)+4/5*(321)$);
\node[below] at ($4/5*(321)+4/5*(231)$) {\small$P_{A_2}(\{2\})$};
\end{scope}

\begin{scope}[xshift=740]
\node[dotted, thick, regular polygon, regular polygon sides=6, draw, minimum size=1.8cm](m) at (0,0) {};
\coordinate (123) at (m.corner 1); 
\coordinate (213) at (m.corner 2); 
\coordinate (312) at (m.corner 3); 
\coordinate (321) at (m.corner 4); 
\coordinate (231) at (m.corner 5); 
\coordinate (132) at (m.corner 6); 

\draw[thick, fill=yellow!40] (123)--($1/2*(123)+1/2*(132)$)--($1/2*(123)+1/2*(321)$)--($1/2*(213)+1/2*(123)$)--cycle;

\draw[dashed, red, thick] ($4/5*(123)+4/5*(132)$)--($4/5*(312)+4/5*(321)$);
\draw[dashed, red, thick] ($4/5*(123)+4/5*(213)$)--($4/5*(321)+4/5*(231)$);

\node[below] at ($4/5*(321)+4/5*(231)$) {\small$P_{A_2}(\{1,2\})$};
\end{scope}

\begin{scope}[scale=0.45, yshift=-1050, xshift=-300]

\begin{scope}
\coordinate (1234) at (0,10); 
\coordinate (1243) at (2.5,6); 
\coordinate (1342) at (6.5,1.5); 
\coordinate (1432) at (10,0);
\coordinate (1324) at (4,8.5);
\coordinate (1423) at (8.5,4);

\coordinate (2143) at (5,12);
\coordinate (4123) at (14,21);
\coordinate (3124) at (7,20);
\coordinate (2134) at (2.5,16);
\coordinate (3142) at (10.5,12.5);
\coordinate (4132) at (14.5,16.5);

\coordinate (4312) at (21.5,15);
\coordinate (4213) at (17.5,19.5);
\coordinate (2413) at (13,8.5);
\coordinate (2314) at (8.5,13);
\coordinate (3214) at (10.5,18.5);
\coordinate (3412) at (19,9);

\coordinate (2431) at (17,1);
\coordinate (3421) at (21.5,5);
\coordinate (4321) at (24,11);
\coordinate (2341) at (13.5,2.5);
\coordinate (3241) at (14.5,8.5);
\coordinate (4231) at (18.5,12.5);

\draw (1234)--(1243)--(1342)--(1432)--(2431)--(3421)--(4321)--(4312)--(4213)--(4123)--(3124)--(2134)--cycle;
\draw (1234)--(1324)--(1423)--(1432);
\draw (1324)--(2314)--(3214)--(3124);
\draw (3214)--(4213);
\draw (2314)--(2413)--(1423);
\draw (2413)--(3412)--(3421);
\draw (4312)--(3412);

\draw[dotted] (2134)--(2143)--(1243);
\draw[dotted] (1342)--(2341)--(2431);
\draw[dotted] (2143)--(3142)--(4132)--(4123);
\draw[dotted] (4132)--(4231)--(4321);
\draw[dotted] (2341)--(3241)--(3142);
\draw[dotted] (4321)--(4231)--(3241);

\draw[fill=yellow, opacity=0.5] ($0.5*(1234)+0.5*(2134)$)--($0.5*(2314)+0.5*(3214)$)--($0.5*(3412)+0.5*(4312)$)--($0.5*(3421)+0.5*(4321)$)--(3421)--(2431)--(1432)--(1342)--(1243)--(1234)--cycle;
\draw[thick, blue] ($0.5*(1234)+0.5*(2134)$)--($0.5*(2314)+0.5*(3214)$)--($0.5*(3412)+0.5*(4312)$)--($0.5*(3421)+0.5*(4321)$)--(3421)--(2431)--(1432)--(1342)--(1243)--(1234)--cycle;

\draw[thick, blue, dashed] ($0.5*(1234)+0.5*(2134)$)--($0.5*(1243)+0.5*(2143)$)--($0.5*(3241)+0.5*(2341)$)-- ($0.5*(4321)+0.5*(3421)$);
\draw[thick, blue, dashed] ($0.5*(1243)+0.5*(2143)$)--(1243);
\draw[thick, blue, dashed] ($0.5*(3241)+0.5*(2341)$)--(2341)--(1342);
\draw[thick, blue, dashed] ($0.5*(3241)+0.5*(2341)$)--(2341)--(2431);

\draw[thick, blue] ($0.5*(2314)+0.5*(3214)$)--(2314);
\draw[thick, blue] ($0.5*(3412)+0.5*(4312)$)--(3412);
\draw[thick, blue] (1234)--(1324);
\draw[thick, blue] (1324)--(2314)--(2413)--(1423)--cycle;
\draw[thick, blue] (2413)--(3412)--(3421)--(2431)--(1432)--(1423)--cycle;

\node[below] at ($(1432) - (0,1)$) {\small$P_{A_3}(\{1\})$};
\end{scope}

\begin{scope}[xshift=850]
\coordinate (1234) at (0,10); 
\coordinate (1243) at (2.5,6); 
\coordinate (1342) at (6.5,1.5); 
\coordinate (1432) at (10,0);
\coordinate (1324) at (4,8.5);
\coordinate (1423) at (8.5,4);

\coordinate (2143) at (5,12);
\coordinate (4123) at (14,21);
\coordinate (3124) at (7,20);
\coordinate (2134) at (2.5,16);
\coordinate (3142) at (10.5,12.5);
\coordinate (4132) at (14.5,16.5);

\coordinate (4312) at (21.5,15);
\coordinate (4213) at (17.5,19.5);
\coordinate (2413) at (13,8.5);
\coordinate (2314) at (8.5,13);
\coordinate (3214) at (10.5,18.5);
\coordinate (3412) at (19,9);

\coordinate (2431) at (17,1);
\coordinate (3421) at (21.5,5);
\coordinate (4321) at (24,11);
\coordinate (2341) at (13.5,2.5);
\coordinate (3241) at (14.5,8.5);
\coordinate (4231) at (18.5,12.5);

\draw[dotted] (2134)--(2143)--(1243);
\draw[dotted] (1342)--(2341)--(2431);
\draw[dotted] (2143)--(3142)--(4132)--(4123);
\draw[dotted] (4132)--(4231)--(4321);
\draw[dotted] (2341)--(3241)--(3142);
\draw[dotted] (4321)--(4231)--(3241);

\draw[fill=yellow, opacity=0.5] (16.5/2,1.5/2)--(30.5/2,3.5/2)--(18,13.5/2)--(5,14)-- (2.5/2, 13)--(1234)--(1243)--(1342)--cycle;
\draw[thick, blue] (16.5/2,1.5/2)--(30.5/2,3.5/2)--(18,13.5/2)--(5,14)-- (2.5/2, 13)--(1234)--(1243)--(1342)--cycle;

\draw[thick, blue] (5,14)--(2,18.5/2)--(1234);
\draw[thick, blue] (2,18.5/2)--(16.5/2,1.5/2);
\draw[thick, blue, dashed] (2.5/2, 13)--(7.5/2,9)--(1243);
\draw[thick, blue, dashed] (7.5/2,9)--(14, 11/2)--(18,13.5/2);
\draw[thick, blue, dashed] (14, 11/2)--(2341)--(1342);
\draw[thick, blue, dashed] (30.5/2,3.5/2)--(2341);

\draw (1234)--(1243)--(1342)--(1432)--(2431)--(3421)--(4321)--(4312)--(4213)--(4123)--(3124)--(2134)--cycle;
\draw (1234)--(1324)--(1423)--(1432);
\draw (1324)--(2314)--(3214)--(3124);
\draw (3214)--(4213);
\draw (2314)--(2413)--(1423);
\draw (2413)--(3412)--(3421);
\draw (4312)--(3412);

\node[below] at ($(1432) - (0,1)$) {\small$P_{A_3}(\{1,2\})$};

\end{scope}

\begin{scope}[xshift=1700]
\coordinate (1234) at (0,10); 
\coordinate (1243) at (2.5,6); 
\coordinate (1342) at (6.5,1.5); 
\coordinate (1432) at (10,0);
\coordinate (1324) at (4,8.5);
\coordinate (1423) at (8.5,4);

\coordinate (2143) at (5,12);
\coordinate (4123) at (14,21);
\coordinate (3124) at (7,20);
\coordinate (2134) at (2.5,16);
\coordinate (3142) at (10.5,12.5);
\coordinate (4132) at (14.5,16.5);

\coordinate (4312) at (21.5,15);
\coordinate (4213) at (17.5,19.5);
\coordinate (2413) at (13,8.5);
\coordinate (2314) at (8.5,13);
\coordinate (3214) at (10.5,18.5);
\coordinate (3412) at (19,9);

\coordinate (2431) at (17,1);
\coordinate (3421) at (21.5,5);
\coordinate (4321) at (24,11);
\coordinate (2341) at (13.5,2.5);
\coordinate (3241) at (14.5,8.5);
\coordinate (4231) at (18.5,12.5);

\draw (1234)--(1243)--(1342)--(1432)--(2431)--(3421)--(4321)--(4312)--(4213)--(4123)--(3124)--(2134)--cycle;
\draw (1234)--(1324)--(1423)--(1432);
\draw (1324)--(2314)--(3214)--(3124);
\draw (3214)--(4213);
\draw (2314)--(2413)--(1423);
\draw (2413)--(3412)--(3421);
\draw (4312)--(3412);

\draw[dotted] (2134)--(2143)--(1243);
\draw[dotted] (1342)--(2341)--(2431);
\draw[dotted] (2143)--(3142)--(4132)--(4123);
\draw[dotted] (4132)--(4231)--(4321);
\draw[dotted] (2341)--(3241)--(3142);
\draw[dotted] (4321)--(4231)--(3241);

\draw[fill=yellow, opacity=0.5] (1234)--($0.5*(1234)+0.5*(2134)$)--($0.5*(1234)+0.5*(3214)$)--($0.5*(1234)+0.5*(4321)$)--($0.5*(1234)+0.5*(1432)$)--($0.5*(1234)+0.5*(1243)$)--cycle;
\draw[thick, blue] (1234)--($0.5*(1234)+0.5*(2134)$)--($0.5*(1234)+0.5*(3214)$)--($0.5*(1234)+0.5*(4321)$)--($0.5*(1234)+0.5*(1432)$)--($0.5*(1234)+0.5*(1243)$)--cycle;

\draw[thick, blue, dashed] ($0.5*(1234)+0.5*(2134)$)--($0.5*(1234)+0.5*(2143)$)--($0.5*(1234)+0.5*(4321)$);
\draw[thick, blue, dashed] ($0.5*(1234)+0.5*(2143)$)--($0.5*(1234)+0.5*(1243)$);

\draw[thick, blue] (1234)--($0.5*(1234)+0.5*(1324)$)--($0.5*(1234)+0.5*(3214)$);
\draw[thick, blue] ($0.5*(1234)+0.5*(1324)$)--($0.5*(1234)+0.5*(1432)$);

\draw (1234)--(1243)--(1342)--(1432)--(2431)--(3421)--(4321)--(4312)--(4213)--(4123)--(3124)--(2134)--cycle;
\draw (1234)--(1324)--(1423)--(1432);
\draw (1324)--(2314)--(3214)--(3124);
\draw (3214)--(4213);
\draw (2314)--(2413)--(1423);
\draw (2413)--(3412)--(3421);
\draw (4312)--(3412);

\node[below] at ($(1432) - (0,1)$) {\small$P_{A_3}(\{1,2,3\})$};
\end{scope}

\end{scope}

\end{tikzpicture}
\caption{Examples of partitioned permutohedra in type $A_2$ and $A_3$.}
\label{fig_part_perm_A}
\end{figure}
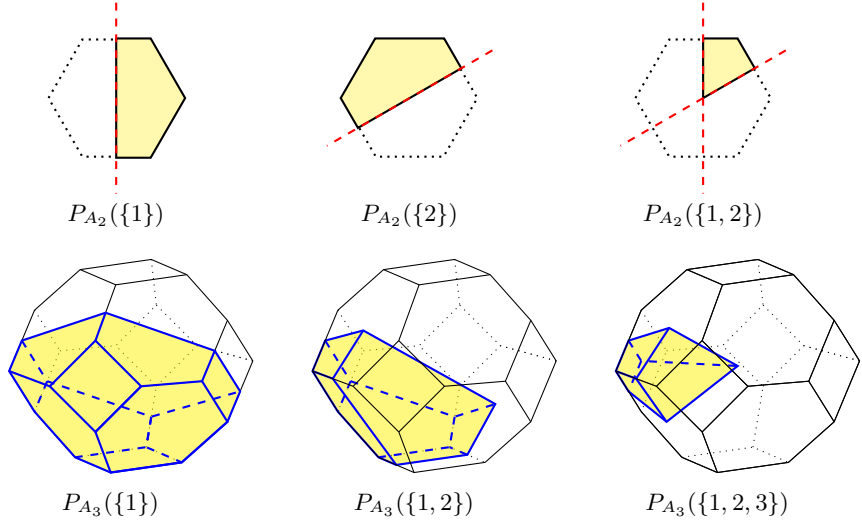

Now we study the face structure of $P_{\An}(K)$. Let $\Hk$ be the boundary of $\Hk^\leq$, which is the hyperplane in $E$ defined by $x_k=x_{k+1}$, and let $\Hk^<$ be the interior of $\Hk^\leq$ in $E$.  We note that facets of $P_{\An}(K)$ are of the form
\begin{equation*}\label{eq_F_K(I)_H_K(k)}
F_K(I) \colonequals F(\indI)\cap P_{\An}(K) ~ \text{or}~ H_K(k)\colonequals \Hk\cap P_{\An}(K).
\end{equation*}

\begin{proposition}\label{prop_A_position_F(I)}
For $k\in [n-1]$ and $I\in \M_{\An}$, the following claims hold.
\begin{enumerate}
\item We have $F(\indI)\cap \Hk\not=\emptyset$ if and only if $\indI$ is $s_k$-invariant, i.e. $\{k,k+1\}$ is contained in either $I$ or $[n]\backslash I$. In this case, $F(\indI)$ is bisected by the hyperplane $\Hk$.
\item We have $F(\indI)\subset \Hk^<$ if and only if $k\in I$ but $k+1\notin I$.  
\end{enumerate}
Therefore, $F(I)\cap H(k)^{\le}\not=\emptyset$ if and only if $k\in I$ whenever $k+1\in I$.  
\end{proposition}
\begin{proof}
(1) We observe first that the facet $F(\indI)$ is either bisected by $\Hk$ or has empty intersection with $\Hk$. Indeed, since the permutohedron $P_{\An}$ is $s_k$-invariant, the reflection through the hyperplane $\Hk$ either  preserves $F(\indI)$ or sends $F(\indI)$ to another facet $F(s_k(I))$. In the latter case, $F(\indI)\cap F(s_k(I))=\emptyset$, since it follows from  the definition of $F(\indI)$ in \eqref{eq_A_facet_convex_hull} that the two facets have no common vertex. As $\indI$ is $s_k$-invariant if and only if $F(\indI)$ is invariant under the reflection through the hyperplane $\Hk$ if and only if $F(\indI)$ is bisected by $\Hk$, statement (1) holds. 

(2) If $k\in I$ but $k+1\not\in I$, then it follows from the definition of $F(\indI)$ in \eqref{eq_A_facet_convex_hull} that any vertex $(x_1, \dots, x_{n})$ of $F(\indI)$ satisfies $x_k \leq a_{|\indI|}<x_{k+1}$, which implies that $F(\indI)$ is contained in $\Hk^<$.  Conversely, if $F(\indI)$ is contained in $\Hk^<$, then exactly one of $k$ and $k+1$ is contained in $\indI$ by (1).  If $k \notin \indI$ but $k+1 \in \indI$, then the same argument as above shows that  $x_{k+1}\leq a_{|I|} < x_k$ for every vertex $(x_1, \dots, x_{n})$ of $F(\indI)$, which implies that $F(\indI)$ is contained in the open half-space opposite to $\Hk^<$.  Therefore, $k\in I$ but $k+1\notin I$. 
\end{proof}

Motivated by Proposition~\ref{prop_A_position_F(I)}, we make the following definition. 

\begin{definition} \label{def_A_K-lower}
Let $K$ be a subset of $[n-1]$ and $[n] = N_1 \sqcup \cdots \sqcup N_s$ the $\SS_K$-orbit decomposition of $[n]$.  Given $i \in [s]$, we call a subset $M$ of $N_i$ a {\it lower subset} if $k \in M$ whenever $k \in N_i$ and $k+1 \in M$.  We define
\[
\M_{\An}(K)\colonequals \{I\in \M_{\An}\mid \text{$I \cap N_i$ is a lower subset of $N_i$ for all $1 \leq i \leq s$}\}.
\]
\end{definition}

For instance, consider the root system of type $A_2$, take $K=\{1\}$, and consider the corresponding parabolic subgroup $\SS_{\{1\}}$ of $\SS_3$. Then, $\SS_{\{1\}}$-orbit decomposition of $[3]$ is given by $\{1,2\} \sqcup \{3\}$ and $\M_{A_2}(K)$ consists of $\{1\}, \{3\},\{1,2\}$ and $\{1,3\}$.

\begin{corollary}\label{cor_A_facet_characterization_of_Part_perm}
For $\indI\in \M_{\An}$ and $k\in [n-1]$, we have 
\begin{enumerate}
\item $H_K(k)$ is a facet of $P_{\An}(K)$ if and only if $k\in K$, 
\item $F_K(\indI)$  is a facet of $P_{\An}(K)$ if and only if $\indI\in \M_{\An}(K)$.
\end{enumerate}
Therefore, the set of facets of $P_{\An}(K)$ corresponds bijectively to $\M_{\An}(K)\cup K$. 
\end{corollary}
\begin{proof}
Statement (1) follows quickly from the definition of $P_{\An}(K)$ 
and statement (2) follows from Proposition~\ref{prop_A_position_F(I)}. 
\end{proof}

\subsection{Flagness of $P_{\An}(K)$}\label{subsec_flagness}
Recall that a codimension-$\d$ face of $P_{\An}$ is represented by a nested sequence 
\[
I_\bullet=(I_1\subsetneq I_2\subsetneq \cdots\subsetneq I_d)
\]
of elements $I_1,\dots,I_d$ in $\M_{\An}$ as in \eqref{eq_Anested_seq}. We denote by $F(I_\bullet)$ and $C(I_\bullet)$ the face corresponding to $I_\bullet$ and its barycenter respectively. 
\begin{lemma}\label{lem_A_barycenter}
\begin{enumerate}
\item \label{lem_barycenter_1}$C(I_\bullet) \in  \Hk$ if $I_\bullet$ is $s_k$-invariant.
\item \label{lem_barycenter_2}$C(I_\bullet) \in P_{\An}(K)$ if $I_j\in \M_{\An}(K)$ for each $1\leq j \leq \d$. 
\end{enumerate}
\end{lemma}
\begin{proof}
(1) If $I_\bullet$ is $s_k$-invariant, then each facet $F(I_j)$ of $P_{\An}$ is invariant under the reflection through the hyperplane $\Hk$. This implies statement (1). 

(2) Suppose that $I_j\in \M_{\An}(K)$ for each $1\le j\le d$.  Then $F_K(I_j)$ is a facet of $P_{\An}(K)$ by Corollary~\ref{cor_A_facet_characterization_of_Part_perm}-(2). 
Let $k\in K$.  If $I_\bullet$ is $s_k$-invariant, then 
\begin{equation*} 
C(I_\bullet) \in \Hk
\end{equation*}
by (1) above.  If $I_\bullet$ is not $s_k$-invariant, then some $I_j$ is not $s_k$-invariant. Hence $k\in I_j$ but $k+1\notin I_j$ since $I_j\in \M_{\An}(K)$, which implies $F(I_j)\subset H(k)^{<}$ by Proposition~\ref{prop_A_position_F(I)}-(2). It follows that 
\begin{equation*} \label{eq:C(Idot)_2}
 C(I_\bullet) \in F(I_\bullet)=\bigcap_{j=1}^d F(\indI_j) \subset \Hk^<. 
\end{equation*}
In any case $C(I_\bullet)\in H(k)^\leq$ for $k\in K$. Hence we get
\[
C(I_\bullet) \in F(I_\bullet) \cap \bigcap_{k\in K}\Hk^\leq  \subset P_{\An}\cap \bigcap_{k\in K}\Hk^\leq=P_{\An}(K),
\]
proving the statement (2). 
\end{proof}

\begin{proposition} \label{prop_A_flagness_part_w_poly}
For any subset $K\subset [n-1]$, the associated partitioned permutohedron $P_{\An}(K)$ is a flag polytope. 
\end{proposition}
\begin{proof}
We must show that, assuming each pair from a set 
\begin{equation*}\label{eq_A_pairwise_intersection}
\{F_K(I_1), \dots, F_K(I_p), H_K(k_1), \dots, H_K(k_q)\}
\end{equation*}
of facets of $P_{\An}(K)$ has nonempty intersection, we have 
\begin{equation}\label{eq_A_claim_for_flagness}
F_K(I_1)\cap  \cdots\cap  F_K(I_p)\cap  H_K(k_1)\cap  \cdots\cap  H_K(k_q)\neq \emptyset.
\end{equation}
This is obvious when $p=0$ because every $H_K(k)$ contains the barycenter of $P_{\An}$.  We will prove \eqref{eq_A_claim_for_flagness} when $p\ge 1$ as follows. 

Since $F_K(I_i) \cap F_K(I_j) \neq \emptyset$ for all $i,j \in [p]$ and $F_K(I)=F(I) \cap P_{\An}(K)$ for all $I\in \M_{\An}$, we see that $F(I_i) \cap F(I_j) \neq \emptyset$ for all $i,j \in [p]$.  As $P_{\An}$ is flag, we may permute indices to obtain a nested sequence
\[
I_\bullet \colonequals \left(I_1\subsetneq I_2 \subsetneq\cdots \subsetneq I_p\right)
\]
by Lemma~\ref{lem_A_facets_intersecting}. We show that $C(I_\bullet)$ lies in the intersection \eqref{eq_A_claim_for_flagness}. 

First, we claim that $C(I_\bullet) \in F_K(I_1)\cap  \cdots\cap  F_K(I_p)$. Indeed, since $F_K(I_j)$ is a facet of $P_{\An}(K)$ for $1\le j\le p$, we have $I_j\in \M_{\An}(K)$ by Corollary~ \ref{cor_A_facet_characterization_of_Part_perm}-(2) and hence $C(I_\bullet)\in P_{\An}(K)$ by Lemma \ref{lem_A_barycenter}-\eqref{lem_barycenter_2}.  It follows that 
\begin{align*}
C(I_\bullet)\in F(I_\bullet)\cap P_{\An}(K)=\bigcap_{i=1}^p F(I_i) \cap P_{\An}(K)
=  \bigcap_{i=1}^p F_K(I_i).
\end{align*}

Next, we claim that $C(I_\bullet) \in H_K(k_1)\cap  \cdots\cap  H_K(k_q).$ By assumption we have 
$$\emptyset \neq F_K(I_i) \cap H_K(k_j) \subset F(I_i) \cap H(k_j)$$
for any $1\leq i\leq p$ and $1\leq j \leq q$. 
Therefore, each $I_i$ is $s_{k_j}$-invariant for any $1\le j\le q$ by Proposition~\ref{prop_A_position_F(I)}-(1) and hence $I_\bullet$ is $s_{k_j}$-invariant whenever $1\le j\le q$.  Since $C(I_{\bullet})\in P_{\An}(K)$ as observed above, it follows from Lemma~\ref{lem_A_barycenter}-\eqref{lem_barycenter_1} that  
$$C(I_\bullet)\in  P_{\An}(K) \cap \bigcap_{j=1}^q \H(k_j)=\bigcap_{j=1}^q H_K(k_j).$$
Two claims above prove the desired \eqref{eq_A_claim_for_flagness}.  
\end{proof}

There are two types of facets, $F_K(I)$ and $H_K(k)$, in $P_{\An}(K)$.  We know that  $H_K(k)\cap H_K(k')\neq \emptyset $ for arbitrary $k, k'\in K$ since the barycenter of $P_{\An}$ is contained in $H_K(k)$ for all $k\in K$. The next proposition tells us which pairs of facets of $P_{\An}(K)$ have empty intersection. 

\begin{proposition} \label{prop_A_intersection_facets_PK}
Let $\FlA, \FlB$ and $H_K(k)$ be facets of $P_{\An}(K)$, so that $\indI,\indJ\in\M_{\An}(K)$ and $k\in K$.  Then the following statements hold:
\begin{enumerate}
\item $\FlA\cap \FlB=\emptyset$ if and only if $I\not\subset J$ and $J\not\subset I$.
\item $\FlA\cap H_K(k)=\emptyset$ if and only if $\indI$ is not $s_k$-invariant.
\end{enumerate} 
\end{proposition} 

\begin{proof}
(1) Since $\FlA\subset \FA$ and $\FlB\subset \FB$, the \lq\lq if\rq\rq\ part is obvious by Lemma \ref{lem_A_facets_intersecting}. In order to prove the ``only if'' part, we prove 
\[
\FlA\cap \FlB\not=\emptyset\ \text{ if $I\subset J$ or $J\subset I$}.\tag{\text{$*$}}
\]  
If $I=J$, then $(*)$ above trivially holds and we may assume $I\subsetneq J$ without loss of generality.  Then, we have a flag $\indI_\bullet=(\indI\subsetneq \indJ)$ and the barycenter $\CAbullet$ of $\FAbullet$ is in $\FAbullet\cap P_{\An}(K)$ by Lemma \ref{lem_A_barycenter}-(2).  This proves $(*)$ above since $\FAbullet\cap P_{\An}(K)=\FlA\cap \FlB$, thereby proving the ``only if'' part.

(2) Statement (2) is equivalent to the following:
\[
\text{$\FlA\cap H_K(k)\not=\emptyset$ if and only if $\indI$ is $s_k$-invariant.}
\]
Furthermore, $I$ is $s_k$-invariant if and only if $F(I)\cap H(k)\not=\emptyset$ by Proposition~\ref{prop_A_position_F(I)}-(1).  Therefore, it suffices to prove the following:
\[
\text{$\FlA\cap H_K(k)\not=\emptyset$ if and only if $F(I)\cap H(k)\not=\emptyset$.}\tag{\text{$**$}}
\]
The \lq\lq only if\rq\rq\ part in $(**)$ above is trivial because 
\begin{equation} \label{eq:**}
F_K(I)\cap H_K(k)=F(I)\cap H(k)\cap P_{\An}(K).
\end{equation}
The \lq\lq if\rq\rq\ part in $(**)$ above can be proved as follows.  If $F(I)\cap H(k)\not=\emptyset$, then the barycenter $C(I)$ of $F(I)$ is in $F(I)\cap H(k)$.  Since $H_K(k)$ is a facet of $P_{\An}(K)$, we have $k\in K$ by Corollary~\ref{cor_A_facet_characterization_of_Part_perm}-(1) and hence $C(I)$ is also in $P_{\An}(K)$ by Lemma~\ref{lem_A_barycenter}-(1). This together with \eqref{eq:**} shows $C(I)\in F_K(I)\cap H_K(k)$, proving the \lq\lq if\rq\rq\ part in $(**)$ above.   
\end{proof}

\subsection{Cohomology of the toric variety associated with $P_{\An}(K)$} \label{subsec_A_toric_var_part_w_poly}

The primitive inward normal vector of the facet $\FlA$ agrees with that of $\FA$. On the other hand, an (inward) normal vector of $H_K(k)$ is given by 
$$-\alpha^{\v}_k=-e_k+e_{k+1}$$ 
because $H_K(k)=P_{\An}(K)\cap \{x\in E \mid \alpha^{\v}_k(x)=0\}$ and $P_{\An}(K)$ is contained in the half-space $H(k)^{\le}=\{x\in E\mid \alpha_k^{\v}(x)\le 0\}$. 

Let $X_{\An}(K)$ be the toric variety associated with a partitioned permutohedron $P_{\An}(K)$. Since $P_{\An}$ is a simple polytope, so is $P_{\An}(K)$. Hence, $X_{\An}(K)$ is a toric orbifold as described in Section \ref{sec_toric_orb}. Now, applying Theorem~\ref{theorem_cohom_toric_var} together with Propositions~\ref{prop_A_flagness_part_w_poly} and~\ref{prop_A_intersection_facets_PK} to $X_{\An}(K)$, we obtain the following.
 
\begin{proposition} \label{prop_A_cohomologyXK}
Let $K$ be a subset of $[n-1]$. 
The cohomology ring of the toric orbifold $X_{\An}(K)$ associated with a partitioned permutohedron $P_{\An}(K)$ is given by
\begin{equation}\label{eq_A_cohomologyXK}
H^\ast(X_{\An}(K))\cong \QQ[\tau_\indI, \tau_{s_k} \mid \indI \in \M_{\An}(K), \ k \in K]/\left(\mathcal{I}(K) + \mathcal{J}(K)\right),
\end{equation}
where  $\deg\tau_\indI=\deg\tau_{s_k}=2$; $\mc{I}(K)$ is the ideal generated by 
\begin{enumerate}
\item $\sum_{k\in K} \left< u, -\alpha^{\v}_k \right> \tau_{s_k} + \sum_{\indI\in\M_{\An}(K)} \left< u, e_\indI \right> \tau_\indI$ for $u\in L(\Phi_{\An})$;
\end{enumerate}
and $\mc{J}(K)$ is the ideal generated by
\begin{enumerate}
\item[(2)] $\tau_\indI \tau_\indJ$ for $\indI, \indJ \in\M_{\An}(K)~ \text{such that}~\indI \not\subset \indJ \text{ and } \indJ \not\subset \indI$, and
\item[(3)] $ \tau_\indI \tau_{s_k}$ for $\indI \in\M_{\An}(K),~ k\in K~\text{such that}~\textrm{$\indI$ is not $s_k$-invariant}$.
\end{enumerate}
Here, $\left< \ , \ \right>$ denotes the natural pairing between $L(\Phi_{\An})$ and $\hat{L}(\Phi_{\An}^{\v})$. 
\end{proposition}

\section{Proof of Theorem \ref{main} for type $A$} \label{sec_A_proof_of_main_thm}

In this section we prove Theorem \ref{main} when $\Phi$ is of type $\An$.  Let ${\mathbb K}$ be any field.  A  finite dimensional graded ${\mathbb K}$-algebra $R=\bigoplus_{k=1}^d R_i$ is called a {\it Poincar\'{e} duality algebra} if the bilinear map from $R_i \times R_{d-i}$ to $R_d$ determined by multiplication in $R$ is nondegenerate for $0 \leq i \leq d$.  We observe that if $X$ is a compact, orientable orbifold, then $H^\ast(X;\QQ)$ is a Poincar\'{e} duality $\QQ$-algebra.

\begin{lemma}{\cite[Lemma~10.5]{AHHM}} \label{lem_PDA-implies_isom}
Let $\phi \colon \mathcal{A}=\bigoplus_{i=0}^m A_i \to \mathcal{B}=\bigoplus_{i=0}^m B_i$ be a surjective homomorphism between two graded algebras. If the domain $\mathcal{A}$ is a Poincar\'e duality algebra and $\phi |_{A_m} \colon A_m \to B_m$ is an isomorphism, then $\phi$ is an isomorphism. 
\end{lemma}

We aim to apply Lemma \ref{lem_PDA-implies_isom} with ${\mathcal A}=H^\ast(X_{\An}(K))$ and ${\mathcal B}=H^\ast(X_{\An})^{\SS_K}$, where $K$ is an arbitrary subset of $[n-1]$.  To this end, we will define a homomorphism
$$
\varphi:\QQ[\tau_\indI, \tau_{s_k} \mid \indI \in \M_{\An}(K), \ k \in K] \rightarrow H^\ast(X_{\An})
$$
and show that
\begin{itemize}
\item[(A)] the kernel of $\varphi$ contains the ideals $\mathcal{I}(K)$ and $\mathcal{J}(K)$ from (\ref{eq_A_cohomologyXK}), and
\item[(B)] the image of $\varphi$ is $H^\ast(X_{\An})^{\SS_K}$.
\end{itemize}

We consider the $\SS_K$-orbit decomposition 
\[
[n]=N_1 \sqcup \cdots \sqcup N_s
\]
of $[n]$. It is straightforward to see that for $v \in \SS_n$ and $I \in \M_{\An}(K)$, 
\begin{equation}\label{eq_A_coefficientsbeta}
e_\indI - e_{v(\indI)}=\sum_{j=1}^s (e_{\indI \cap N_j}-e_{v(\indI) \cap N_j})=\sum_{k \in K}  c^{\indI,v}_k \alpha^{\v}_k
\end{equation}
for some non-negative integers $c^{\indI,v}_k$. 

\begin{example}
Consider type $A_3$ and $K=\{1,2,3 \}$. \\
(1) If $\indI=\{1\}$ and $v=s_1$, then we have
$$e_\indI-e_{v(\indI)} = e_{\{1\}}-e_{\{2\}}=\alpha^{\v}_1.$$ 
(2) If $\indI=\{1,2\}$ and $v=s_2s_1s_3s_2$, then we have
$$e_\indI-e_{v(\indI)} = e_{\{1,2\}}-e_{\{3,4\}}=(e_1-e_3)+(e_2-e_4)=\alpha^{\v}_1+2\alpha^{\v}_2+\alpha^{\v}_3.$$ 
\end{example}

\begin{lemma} \label{lem_A_coeff_nonzero}
Assume that $k \in N_i$ for some $k \in K$ and some $i \in [s]$.  The coefficient $c^{\indI, v}_k$ in \eqref{eq_A_coefficientsbeta} is zero if
\begin{equation} \label{eq:Lemma4.4-1TypeA}
v(\indI) \cap N_i \subset [k] \cap N_i \ {\rm or} \ v(\indI) \cap N_i \supset [k] \cap N_i
\end{equation}
holds.
\end{lemma}

\begin{proof}
The first equality of \eqref{eq_A_coefficientsbeta} implies that  $\alpha^{\v}_k=e_k-e_{k+1}$ does not appear in $e_{\indI \cap N_j}-e_{v(\indI) \cap N_j}$ unless $j=i$. Therefore, it is enough to prove that $\alpha_k^{\v}$ does not appear in $e_{\indI\cap N_i}-e_{v(\indI)\cap N_i}$ if $v(\indI)$ satisfies the condition \eqref{eq:Lemma4.4-1TypeA}.

Since $\indI\in \M_{\An}(K)$, we have 
\begin{align*}
\indI \cap N_i \subset [k] \cap N_i \ {\rm or} \ \indI \cap N_i \supset [k] \cap N_i. 
\end{align*}
We note that $|v(I)\cap N_i|=|I\cap N_i|$ since $v$ preserves $N_i$.  Therefore, under the condition \eqref{eq:Lemma4.4-1TypeA}, we have either
\begin{enumerate}
\item[$\bullet$] $v(\indI) \cap N_i \subset [k] \cap N_i$ and $\indI \cap N_i \subset [k] \cap N_i$ \ {\rm or} 
\item[$\bullet$] $v(\indI) \cap N_i \supset [k] \cap N_i$ and $ \indI \cap N_i \supset [k] \cap N_i$.
\end{enumerate}
In the former case, both $\indI\cap N_i$ and $v(\indI)\cap N_i$ are contained in $[k]$ while in the latter case, both $\indI\cap N_i$ and $v(\indI)\cap N_i$ contain $[k]$. Therefore, $\alpha^{\v}_k$ does not appear in $e_{\indI \cap N_i}-e_{v(\indI) \cap N_i}$ in either case. 
\end{proof}

We define a graded ring homomorphism 
\begin{equation*}
\varphi \colon \QQ[\tau_\indI, \tau_{s_k} \mid \indI \in \M_{\An}(K), \ k \in K] \to H^\ast(X_{\An})
\end{equation*}
by setting 
\begin{align} \label{eq_def_varphi}
\begin{split}
\varphi (\tau_\indI) &= [\tau_\indI]=\sum_{v\in \SS_K/\SS_K^\indI} \tau_{v(\indI)}; \\
\varphi(\tau_{s_k})&=\sum_{\indI\in\M_{\An}(K)} \sum_{v\in \SS_K/\SS_K^\indI} c_k^{\indI,v} \tau_{v(\indI)},
\end{split}
\end{align}
where the $c_k^{\indI,v}$ are defined in \eqref{eq_A_coefficientsbeta}. Note that if $v(\indI)=v'(\indI)$ for $v,v'\in \SS_K$, then $e_{v(\indI)}=e_{v'(\indI)}$, which means $c_k^{\indI,v}=c_k^{\indI,v'}$. Hence, $\varphi(\tau_{s_k})$ in \eqref{eq_def_varphi} makes sense. 

\begin{proposition} \label{prop_A_well_definedness}
The kernel of the map $\varphi$ contains both ideals $\mathcal{I}(K)$ and $\mathcal{J}(K)$ from (\ref{eq_A_cohomologyXK}).
\end{proposition}

\begin{proof}
First, we will show that $\varphi(\mc{I}(K))=0$, that is, 
\begin{equation} \label{eq_A_varphi(I_K)}
\varphi \left( \sum_{k\in K} \left< u, -\alpha^{\v}_k \right> \tau_{s_k} + \sum_{\indI\in\M_{\An}(K)} \left< u, e_\indI \right> \tau_\indI \right)=0 
\end{equation}
for any $u \in L(\Phi_{\An})$. We compute the left-hand side of \eqref{eq_A_varphi(I_K)} as follows:
\begin{align*}
& \sum_{k\in K} \left< u, -\alpha^{\v}_k \right> \varphi(\tau_{s_k}) + \sum_{\indI \in \M_{\An}(K)} \left< u, e_\indI \right> \left( \sum_{v\in \SS_K/\SS_K^{\indI}} \tau_{v(\indI)} \right) \\
=& \sum_{k\in K} \left< u, -\alpha^{\v}_k \right> \varphi(\tau_{s_k}) + \sum_{\indI \in \M_{\An}(K)} \sum_{v\in \SS_K/\SS_K^{\indI}} \left( \left< u, e_\indI-e_{v(\indI)} \right>  \tau_{v(\indI)} + \left< u, e_{v(\indI)} \right>  \tau_{v(\indI)}  \right) \\
=& \sum_{k\in K} \left< u, -\alpha^{\v}_k \right> \varphi(\tau_{s_k}) + \sum_{\indI \in \M_{\An}(K)} \sum_{v\in \SS_K/\SS_K^{\indI}} \left< u, e_\indI-e_{v(\indI)} \right>  \tau_{v(\indI)} + \sum_{\indI \in \M_{\An}} \left< u, e_\indI \right>  \tau_{\indI}  \\
=& \sum_{k\in K} \left< u, -\alpha^{\v}_k \right> \varphi(\tau_{s_k}) + \sum_{\indI \in \M_{\An}(K)} \sum_{v\in \SS_K/\SS_K^{\indI}} \left< u, e_\indI-e_{v(\indI)} \right>  \tau_{v(\indI)}\\
=& \sum_{k\in K} \left< u, -\alpha^{\v}_k \right> \left( \sum_{\indI\in\M_{\An}(K)} \sum_{v\in \SS_K/\SS_K^{\indI}} c_k^{\indI,v} \tau_{v(\indI)} \right)\\
& \qquad \qquad \qquad \qquad \qquad \qquad + \sum_{\indI\in\M_{\An}(K)} \sum_{v\in \SS_K/\SS_K^{\indI}} \left< u, \sum_{k \in K} c_k^{\indI,v} \alpha^{\v}_k \right>  \tau_{v(\indI)} \\
=&0.
\end{align*}
(The third equality follows from the fact that  $\sum_{\indI \in \M_{\An}} \left< u, e_\indI \right>  \tau_{\indI}$ is trivial in $H^\ast(X_{\An})$, as it is an element in $\mathcal{I}$ from \eqref{eq_A_cohom_of_perm}.)

Next, we prove that $\varphi(\mc{J}(K))=0$. Consider first a product $\tau_\indI\tau_\indJ$ in $\mc{J}(K)$, where $\indI, \indJ \in \M_{\An}(K)$ are as in Definition \ref{def_A_K-lower}, with $\indI\not\subset \indJ$ and $\indJ\not\subset \indI$. It follows from \eqref{eq_def_varphi} that 
\[
\varphi(\tau_\indI\tau_\indJ)=\varphi(\tau_\indI)\varphi(\tau_\indJ)=\left( \sum_{w\in \SS_K/\SS_K^\indI} \tau_{w(\indI)}\right) \left(\sum_{v\in \SS_K/\SS_K^\indJ} \tau_{v(\indJ)}\right).
\]
The last product above vanishes because otherwise $w(\indI)\subset v(\indJ)$ or $v(\indJ)\subset w(\indI)$ for some $v,w\in \SS_K$.  Indeed, by the lower condition for $\indI$ and $\indJ$ we have  $\indI\subset \indJ$ or $\indJ\subset \indI$ which contradicts the assumption that $\indI\not\subset \indJ$ and $\indJ\not\subset \indI$. 

Consider now a product $\tau_\indI\tau_{s_k}$ in $\mc{J}(K)$, where $\indI \in \M_{\An}(K), k \in K$ and $\indI$ is not $s_k$-invariant.  
Recall the $\SS_K$-orbit decomposition of $[n]$ and suppose that $k\in K\cap N_i$ for some $i$. 
Then it follows from \eqref{eq_def_varphi} and Lemma~\ref{lem_A_coeff_nonzero} that 
\begin{align*}
\varphi(\tau_\indI\tau_{s_k})&=\left(\sum_{v\in \SS_K/\SS_K^\indI}\tau_{v(\indI)}\right)\left( \sum_{L\in\M_{\An}(K)} \sum_{w\in \SS_K/\SS_K^L} c_k^{L,w}\tau_{w(L)} \right) \\
&= \left(\sum_{v\in \SS_K/\SS_K^\indI}\tau_{v(\indI)}\right) \left(\sum_{\indJ} d_\indJ\tau_\indJ\right)
\end{align*}
for some integers $d_\indJ$, where $\indJ$ runs over nonempty proper subsets of $[n]$ satisfying the condition
\begin{align}
&\indJ\cap N_i \not\subset [k]\cap N_i \ {\rm and} \ \indJ \cap N_i \not\supset [k]\cap N_i.\label{eq:conditionforTypeA} 
\end{align}
We claim that for each $v\in \SS_K$ 
$$\tau_{v(\indI)} \sum_{\indJ} d_\indJ\tau_\indJ=0.$$
First we show this for $v=id$.
It is enough to prove that $\tau_{\indI}\tau_\indJ=0$.
Since $\indI$ belongs to $\M_{\An}(K)$ and $\indI$ is not $s_k$-invariant, we have $\indI\cap N_i=[k]\cap N_i$. 
If $\indI \subset \indJ$, then $[k]\cap N_i = \indI \cap N_i \subset \indJ \cap N_i$ which contradicts \eqref{eq:conditionforTypeA}.
By a similar argument, if $\indI \supset \indJ$, which is also impossible.
Hence, we obtain $\indI \not\subset \indJ$ and $\indI \not\supset \indJ$, which implies 
\[
\tau_{\indI}\varphi(\tau_{s_k})=\tau_{\indI} \sum_{\indJ} d_\indJ\tau_\indJ=0.
\]

Considering a general $v \in \SS_K$, we first notice that  $\varphi(\tau_{s_k})$ is $\SS_K$-invariant. Indeed, plugging the $k$-th fundamental weight $\varpi_k=(t_1+\cdots+t_k)-\frac{k}{n}(t_1+\cdots+t_n)$ for $u$ in \eqref{eq_A_varphi(I_K)}, we have
\begin{align*}
\varphi(\tau_{s_k}) = \sum_{I \in \M_{\An}(K)} \langle \varpi_k, e_I \rangle \varphi(\tau_I) \in H^*(X_{\An})^{\mathfrak{S}_K}.
\end{align*}
Finally, we have
\[\tau_{v(\indI)} \sum_{\indJ} d_\indJ\tau_\indJ=\tau_{v(\indI)} \varphi(\tau_{s_k})=v(\tau_{\indI}) v(\varphi(\tau_{s_k}))=v(\tau_{\indI}\varphi(\tau_{s_k}))=0.\]
Thus, we conclude that $\varphi(\tau_\indI\tau_{s_k})=0$.
\end{proof}

\begin{lemma} \label{lemma:imagephi}
The image of $\varphi$ is $H^\ast(X_{\An})^{\SS_K}$.
\end{lemma}

\begin{proof}
The lemma follows from Proposition~\ref{prop_A_generated_by_deg2}. 
In fact, $H^\ast(X_{\An})^{\SS_K}$ is generated by the $[\tau_\indI]$'s for $\indI \in \M_{\An}$ and the $\SS_K$-orbit of $\indI$ contains an element $\indI'\in \M_{\An}(K)$ so that $[\tau_\indI]=\varphi(\tau_{\indI'})$. 
\end{proof}

We claim now that the restriction of $\varphi$ to $H^{2n-2}(X_{\An}(K))$ is an isomorphism onto $H^{2n-2}(X_{\An})^{\SS_K}$.  Once we prove this claim, Theorem \ref{main} for type $\An$ will follow from Lemmas \ref{lem_PDA-implies_isom}, \ref{prop_A_well_definedness}, and
\ref{lemma:imagephi}.
We observe that both $H^{2n-2}(X_{\An}(K))$ and $H^{2n-2}(X_{\An})^{\SS_K}$ are $1$-dimensional over $\QQ$.  Indeed, $H^{2n-2}(X_{\An})^{\SS_K}=H^{2n-2}(X_{\An})$, since the action of $\SS_K$ on $X_{\An}$ is orientation preserving.  This proves the claim.

\section{Hessenberg varieties}\label{sec_hess}

The toric variety $X_\Phi$ can be realized as a subvariety of a flag variety, as shown by De Mari, Procesi and Shayman in \cite{MPS}.  Let $G$ be a connected, simply-connected, semisimple algebraic group over $\CC$ and $B$ a Borel subgroup and $T \leq B$ a maximal torus. We denote by $\g$, $\b$, and ${\mathfrak h}$ their respective Lie algebras.  There is a Cartan decomposition $\g={\mathfrak h} \oplus \bigoplus_{\alpha \in \Phi}\g_\alpha$, with $\Phi$ the associated root system.  Assume that $G$ has rank $n$ and, as above, that the set of simple roots in $\Phi$ is $\Sigma=\{\alpha_1,\ldots,\alpha_n\}$.

A $\b$-submodule $\mathcal{H}$ of $\g$ is called a \emph{Hessenberg space} if it  contains $\b$.  In fact, a subspace $\h$ of $\g$ is a Hessenberg space if and only if $\h=\b \oplus \bigoplus_{\gamma \in \Psi} \g_\gamma$, where $\Psi$ is a set of negative roots such that $\gamma +\alpha \in \Psi$ whenever $\gamma \in \Psi$, $\alpha$ is a simple root, and $\gamma +\alpha$ is a negative root.  In particular,
$$
\h_1\colonequals\b \oplus \bigoplus_{\alpha \in \Sigma} \g_{-\alpha}
$$
is a Hessenberg space.

The \emph{Hessenberg variety} $\Hess(\x,\h)$ associated with an element $\x \in \g$ and a Hessenberg space $\h$ is
\[
\Hess(x,\h)\colonequals \{gB \in G/B \mid \mbox{Ad}(g^{-1})(\x) \in \h \}.
\]
An element $r \in \g$ is \emph{regular} if its centralizer in $\g$ has smallest possible dimension, and a \emph{regular Hessenberg variety} is any $\Hess(r,\h)$ such that $r$ is regular.  Similarly, a \emph{regular semisimple Hessenberg variety} is any $\Hess(s,H)$ such that $s$ is regular and semisimple.  

Regular semisimple Hessenberg varieties (for arbitrary $G$ as above) were defined and studied in \cite{MPS}, after work in type A by De Mari and Shayman in \cite{MS}.  The type A case has received considerable attention, as it involves connections between combinatorics, geometry and representation theory, see for example \cite{Tym,SW,BroCho,AHMMS}.  

Given a regular semisimple Hessenberg variety $\Hess(s,\h)$, we may assume that $T$ is the centralizer of $s$ in the adjoint action of $G$.  Then $T$ acts on $\Hess(s,\h)$ by left translation.  This action is equivariantly formal, as defined by Goresky, Kottwitz and MacPherson in \cite{GKM}.  It follows that the $T$-equivariant cohomology ring of $\Hess(s,\h)$ is isomorphic with a certain ring of functions from the vertex set of the associated moment graph to a polynomial ring.  As observed by Tymoczko in \cite{Tym}, the Weyl group $W$ of $G$ acts on this moment graph, and this gives rise to representations of $W$ on both the $T$-equivariant and ordinary cohomology rings of $\Hess(s,\h)$.  The representation of $W$ on $H^\ast(\Hess(s,\h))$ is called the \emph{dot representation}.  The key points for our purposes are the following ones:

\begin{itemize}
\item It is shown in \cite{MPS} that $X_{\Phi}$ is isomorphic with $\Hess(s,\h_1)$ for every regular semisimple $s \in \g$.
\item The (graded) representation of $W$ on $H^\ast(X_{\Phi})$ determined by the action of $W$ on the weight polytope $P_W$, defined in \eqref{eq_weight_poly}, is isomorphic to the dot representation of $W$ on $H^\ast(\Hess(s,\h_1))$, see \cite{Tym,Teff}.
\end{itemize}

Moreover, for regular semisimple $s \in \g$, arbitrary Hessenberg space $\h$, and $K \subseteq [n]$, the fixed point subring $H^\ast(\Hess(s,\h))^{W_K}$ (in the dot representation) is isomorphic to the cohomology ring of a regular Hessenberg variety determined by $K$.  For each $\alpha \in \Sigma$, fix $e_\alpha \in \g_\alpha \setminus \{0\}$. Given $K \subseteq [n]$, set $n_K\colonequals\sum_{i \in K}e_{\alpha_i}$. There exist infinitely many $s_K \in {\mathfrak h}$ such that $[s_K,n_K]=0$ and $s_K+n_K$ is regular.  Choose any such $s_K$ and set $r_K\colonequals s_K+n_K$. 

The following result was proved by Balibanu and Crooks in \cite{BalCro}.  

\begin{theorem} \cite[Proposition 4.7]{BalCro} \label{theorem:BalCro}
Let $s \in \g$ be regular semisimple and let $\h \subseteq \g$ be a Hessenberg space.  For each $K \subseteq [n]$, the graded rings $H^\ast(\Hess(s,\h))^{W_K}$ and $H^\ast(r_K,\h))$ are isomorphic.
\end{theorem}

\begin{remark}
Before the work of Balibanu--Crooks, there were various partial results.   Brosnan and Chow proved in \cite{BroCho}  that when $G=SL_{n+1}(\CC)$, the graded vector spaces $H^\ast(\Hess(s,\h))^{W_K}$ and $H^\ast(\Hess(r_K,\h))$ are isomorphic.  Under the assumption that $K=[n]$, Theorem \ref{theorem:BalCro} was proved by the first two named authors of this paper along with H. Abe and Harada for $G=SL_{n+1}(\CC)$ in \cite{AHHM}, and later by the same two named authors along with T. Abe, Murai and Sato for arbitrary $G$ in \cite{AHMMS}.
Another proof of Theorem \ref{theorem:BalCro} was given by Vilonen and Xue in \cite{ViXu}.
\end{remark}

Combining Theorem \ref{theorem:BalCro} in the case $\h=\h_1$ with Theorem \ref{main}, we obtain the following result.

\begin{corollary} \label{corollary:main}
Let $G$ be a connected, simply-connected, simple complex algebraic group of classical type, with rank $n$.
For arbitrary $K \subseteq [n]$, the graded rational cohomology rings
$H^\ast(\Hess(r_K,\h_1))$ and $H^\ast(X(K))$ are isomorphic.
\end{corollary}

\begin{remark}
$\Hess(r_K,\h_1)$ and $X(K)$ are not isomorphic in general.
In fact, Peterson variety $\Hess(r_{[n]},\h_1)$ is not normal (\cite[Theorem~14]{Kos}, \cite[Corollary~7]{InsYon}), while the corresponding toric variety $X([n])$ is normal.
\end{remark}

\section{$h$-polynomials of partitioned weight polytopes}\label{sect_h-poly}

Given a simple polytope $P$, we write $f_i(P)$ for the number of $i$-dimensional faces of $P$.  the {\it $f$-polynomial} of $P$ is
$$
f_P(t)\colonequals\sum_{i \geq 0}f_i(P)t^i,
$$
and the {\it $h$-polynomial} of $X$ is
$$
h_P(t)\colonequals f_P(t-1).
$$
So, knowing $h_P(t)$ is the same as knowing each $f_i(P)$.  The key point for us is that if $X_P$ is the toric variety associated to rational simple polytope $P$, then
\begin{equation} \label{h.poly}
\sum_{i \geq 0}\dim_\CC H^{2i}(X_P;\CC) t^i=h_P(t),
\end{equation}
we refer to \cite[Section 5.2]{Ful} for instance. 

We will use (\ref{h.poly}) to compute $h$-polynomials of partitioned weight polytopes $P_W(K)$ in two ways.  First, we define $\chi_\Phi$ to be the $t$-graded character for the representation of $W$ on $H^\ast(X_\Phi;\CC)$, so
$$
\chi_\Phi(w)\colonequals\sum_{i \geq 0} \mathsf{Trace}(w,H^{2i}(X_\Phi;\CC)) t^i,
$$
and write $\langle \cdot,\cdot \rangle_G$ for the usual inner product of characters of a finite group $G$.  It follows from (\ref{h.poly}) and Theorem \ref{main} that
\begin{equation} \label{method.1}
h_{P_W(K)}(t)=\langle \chi_\Phi,1 \rangle_{W_K}=\frac{1}{|W_K|}\sum_{w \in W_K}\chi_{\Phi}(w).
\end{equation}

Stembridge provides in \cite[Corollary 1.6]{Stem1} a general formula for $\chi_\Phi(w)$, and gives specific formulas for types $A_n, B_n, C_n$, and $D_n$ in \cite[Corollary 6.1, Corollary 7.2, Theorem 9.2]{Stem1}, respectively.  These can be combined with (\ref{method.1}) to give formulas for $h_{P_W(K)}(t)$.  We write a specific formula in the case $\Phi$ is of type $A_{n-1}$, where \cite[Corollary 6.1]{Stem1} says that if $w \in \SS_n$ has cycle type $\lambda=\lambda(w)=(\lambda_1,\ldots,\lambda_\ell)$, then
\begin{equation} \label{stem.a}
\chi_\Phi(w)=E_{\ell}(t)\prod_{j=1}^\ell [\lambda_j]_t.
\end{equation}
(Here $E_m(t)=\sum_{w \in S_m}t^{\sf{des}(w)}$ is the Eulerian polynomial, that is, the generating function for descents, and $[m]_t\colonequals\frac{t^m-1}{t-1}$.)

Combining (\ref{method.1}) and (\ref{stem.a}), we obtain the following result.

\begin{proposition} \label{type.a.method.1}
Assume that $\Phi$ is of type $A_{n-1}$ and $K \subseteq [n-1]$.  Then
$$
h_{P_W(K)}(t)=\frac{1}{|W_K|}\sum_{w \in W_K}E_{\ell(\lambda(w))}(t)\prod_{j=1}^{\ell(\lambda(w))}[\lambda(w)_j]_t.
$$
\end{proposition}

\begin{example} \label{type.a.ex.method.1}
Say $n=5$ and $K=\{1,2,4\}$.  Then $W_K \cong \SS_3 \times \SS_2$ has one element of cycle type $(1,1,1,1,1)$, four of type $(2,1,1,1)$, three of type $(2,2,1)$, two of type $(3,1,1)$ and two of type $(3,2)$.  According to Proposition \ref{type.a.method.1},
\begin{eqnarray*}
h_{P_W(K)}(t) & = & \frac{1}{12} ( E_5(t)+4E_4(t)(1+t)+3E_3(t)(1+t)^2+2E_3(t)(1+t+t^2) \\ & + & 2E_2(t)(1+t+t^2)(1+t)) \\ & = & 1+9t+17t^2+9t^3+t^4.
\end{eqnarray*}
\end{example}

Next we compute $h_{P_W(K)}(t)$ using (\ref{h.poly}) and Corollary \ref{corollary:main}, which together imply that
\begin{equation} \label{method.2}
h_{P_W(K)}(t)=\sum_i \dim_\QQ H^{2i}(\Hess(r_K,\h_1))t^i.
\end{equation}

In \cite{Pre}, Precup shows that the nonempty intersections of a regular Hessenberg variety $\Hess(r_K,\h)$ with Schubert cells in the associated flag variety determine an affine paving.  She gives a criterion for determining which such intersections are nonempty and a formula for the dimensions of the nonempty intersections.  Applying Precup's results in the case $\h=\h_1$, we obtain the following result.  We maintain the notation from Section \ref{sec_hess}.

\begin{proposition}[See Lemma 2 of \cite{Pre}]
Let $K \subseteq [n]$ and fix $r_K \in \g$.  Set $\Sigma_K\colonequals\{\alpha_i:i \in K\}$ and define $$W(K)\colonequals\{w \in W:w^{-1}(\Sigma_K) \subseteq \Phi^+ \cup -\Sigma\}.$$ For $w \in W(K)$, set $$d(w)\colonequals\left|\{\alpha \in -\Sigma:w(\alpha) \in \Phi^+\}\right|.$$ Then
\begin{equation} \label{precup.method.2}
\sum_i \dim_\QQ H^{2i}(\Hess(r_K,\h_1))t^i=\sum_{w \in W(K)}t^{d(w)}.
\end{equation}
\end{proposition}

Together, (\ref{method.2}) and (\ref{precup.method.2}) determine $h_{P_W(K)}(t)$.  We provide again a more explicit formula and an example in type $A$.  We observe that for $w \in S_n$ and $\alpha=e_i-e_{i+1} \in \Sigma$, we have $w^{-1}(\alpha) \in \Phi^+ \cup -\Sigma$ if and only if $w^{-1}(i)-w^{-1}(i+1) \leq 1$.  Moreover, $d(w)$ is the descent number $\mathsf{des}(w)$.  We obtain the following result.

\begin{proposition} \label{type.a.method.2}
Assume that $\Phi$ is of type $A_{n-1}$ and $K \subseteq [n-1]$.  Then $W(K)$ consists of those $w \in W$ such that $w^{-1}(i)-w^{-1}(i+1) \leq 1$ for all $i \in K$, and
$$
h_{P_W(K)}(t)=\sum_{w \in W(K)}t^{\mathsf{des}(w)}.
$$
\end{proposition}

\begin{example} \label{type.a.ex.method.2}
Say $n=4$ and $K=\{1,3\}$.  Then $W(K)$ consists of those $w \in S_4$ such that $w^{-1}(1)-w^{-1}(2) \leq 1$ and $w^{-1}(3)-w^{-1}(4) \leq 1$.  These are the permutations (in one-line notation) $1234$, $1243$, $1324$, $1342$, $1432$, $2134$, $2143$, $3124$, $3142$, $3214$, $3412$, $3421$, $4312$, and $4321$.  We conclude that $$h_{P_W(K)}(t)=1+6t+6t^2+t^3.$$
\end{example}

\section{Open questions} \label{openquestions}

We discuss here some questions and conjectures related to our work.

\subsection{Quotient polytopes and quotient varieties} Say $P \subseteq {\mathbb R}^n$ is a rational simple polytope and $G \leq GL_n({\mathbb R})$ preserves $P$ setwise.  The action of $G$ on $P$ determines an action on the toric variety $X(P)$ and we may form the quotient, $X(P)/G$.  Let us assume that $G$ is generated by reflections through hyperplanes intersecting transversally with the edges of $P$.  Then there is a ``nice" fundamental region $R$ for this action, that is, a set of orbit representatives that is convex and whose closure is a polytope with every facet being supported by either a hyperplane supporting a facet of $P$ or a hyperplane associated to a reflection in $G$.  We write $P/G$ for the polytope $\overline{R}$.

\begin{question} \label{quotient.conjecture}
If $G$ is generated by reflections in $GL_n(\RR)$ and acts on $P$ as in the paragraph directly above, must $X(P)/G$ be isomorphic with $X(P/G)$?
\end{question}

As mentioned in the introduction, Blume showed in \cite{Blu} that Question \ref{quotient.conjecture} has  positive answers when $P$ is a weight polytope of type $A_n$, $B_n$ or $C_n$, and $G$ is the associated Weyl group with the natural action on the root space. Below, we introduce a couple of more known examples supporting Question~ \ref{quotient.conjecture}. 

Consider the symmetric product $\mathsf{SP}^n(\CP^1)=(\CP^1)^n/\SS_n$, where $\SS_n$ permutes the factors of $(\CP^1)^n$. We regard $(\CP^1)^n$ as the toric variety corresponding to the $n$-dimensional cube, say $[0,1]^n$. 
The fundamental region  for the $\SS_n$-action on $[0,1]^n$ given by permuting factors is 
\begin{equation}\label{eq_fund_region_sym_prod}
[0,1]^n/\SS_n = \{(x_1, \dots, x_n) \in [0,1]^n \mid 0\leq x_1 \leq x_2 \leq \cdots \leq x_n\leq 1\},
\end{equation}
which is an $n$-simplex.  Observe that the toric variety $X([0,1]^n/\SS_n)$ is $\CP^n$, which is isomorphic to $\mathsf{SP}^n(\CP^1)$, see \cite[Corollary 2.6]{Maa} for instance. 

Another example is the weighted projective space $\CP^{n-1}_{(1,2, \dots, n)}=S^{2n-1}/S^1$, where $S^1$-action on $S^{2n-1}\subset \CC^n$ is given by 
\[
t \cdot (z_1, \dots, z_n) = (t^{1}z_1, \dots, t^{n}z_n),\quad \text{for } t\in S^1,~ (z_1, \dots, z_n) \in S^{2n-1}.
\]
This can be constructed as the quotient of $\CP^{n-1}$ by the $\SS_n$-action permuting coordinates as well. Indeed, the quotient morphism $\CP^{n-1} \to \CP^{n-1}_{(1,2,\dots , n)}$ defined by 
\[
[z_1;\cdots ; z_n] \mapsto [e_1(z_1, \dots, z_n) ; \cdots; e_n(z_1, \dots, z_n)],
\]
where $e_i(z_1, \dots, z_n)$ is the elementary symmetric function, induces an isomorphism between $\CP^{n-1}/\SS_n$ and $\CP^{n-1}_{(1,2 \dots, n)}$. We regard $\CP^{n-1}$ as the toric variety corresponding to the standard simplex 
\[
\Delta^{n-1}= \left\{ (x_1, \dots, x_n)\in \RR^n \mid  x_1\geq 0, \dots, x_n\geq 0 \text{ and } x_1+\cdots + x_n=1\right\}. 
\]
The canonical $\SS_n$-action on $\Delta^{n-1}$ given by permuting coordinates defines  
\[
\Delta^{n-1}/\SS_n = \{ (x_1, \dots, x_n)\in \Delta^{n-1} \mid x_1\leq x_2\leq \cdots \leq  x_n\},
\]
whose associated toric variety is $\CP^{n-1}_{(1,2, \dots, n)}$.

\begin{remark}
\begin{itemize}
\item[(a)] Question \ref{quotient.conjecture} has a negative answer if we allow arbitrary finite group actions.  Given $n \geq 2$, consider the toric variety $X\colonequals{\mathbb P}^n \times {\mathbb P}^n$ and let $G \cong {\mathbb Z}_2$ act on $X$, with generator sending $(x,y) \in X$ to $(y,x)$.  The betti numbers of $X/G$ satisfy $\beta_2(X/G)=1$ and $\beta_4(X/G)=2$.  Thus the rational cohomology ring of $X/G$ is not generated in degree two, and $X/G$ is not a toric variety by Theorem \ref{theorem_cohom_toric_var}. 
\item[(b)] Let $T$ be a dense torus in $X(P)$ with free action on $T$ extending to action on $X(P)$.  This action on $X(P)$ does not in general descend to an action on $X(P)/G$, since the actions of $T$ and $G$ on $X(P)$ need not commute.
\end{itemize}
\end{remark}

\subsection{Other Coxeter groups} We believe that the following question has a positive answer.

\begin{question} \label{arbitrary.crystallographic.root}
Let $\Phi$ be an arbitrary crystallographic root system, with associated Weyl group $W$.  Does the conclusion of Theorem \ref{main} hold for $X_\Phi$? 
\end{question}

The crystallographic condition allows the construction of the variety $X_\Phi$ from the fan determined by $\Phi$.  However, we can construct a ring with presentation as given in Theorem \ref{theorem_cohom_toric_var} from an arbitrary root system.  This ring, which we call $R_\Phi$, admits an action of the associated finite reflection group $W$.

\begin{question} \label{arbitrary.root}
Let $\Phi$ be an arbitrary root system. Does the conclusion of Theorem \ref{main} hold for the action of $W$ on $R_\Phi$?
\end{question}

\subsection{Regular semisimple Hessenberg varieties and actions on moment graphs} As discussed in Section \ref{sec_hess}, the representation of $W$ on $H^\ast(X_\Phi)$ is the same as the dot representation on $H^\ast(\Hess(s,\h_1))$, with $s$ regular semisimple. The ring $H^\ast(\Hess(s,\h))$ admits a dot representation for every regular semisimple $s$ and every Hessenberg space $\h$.  It is reasonable to hope that some analogue of Theorem \ref{main} holds in this more general setting.

\begin{problem}
\begin{enumerate}
\item Find an appropriate definition of a quotient moment graph, so that $H^\ast(\Hess(s,\h))^{W_K}$ is isomorphic to a ring of functions from the quotient moment graph to a polynomial ring.
\item Assuming (1) is solved, find a ``nice" variety with equivariantly formal torus action giving rise to a given quotient moment graph.
\end{enumerate}
\end{problem}

\appendix
\section{Types $B$ and $C$}\label{sec_typeB} 

There is no need to distinguish between type $B_n$ and type $C_n$ for Theorem~\ref{main} because the two root systems have the same Weyl group and their simple roots differ by a scalar multiplication.  Moreover, the proof of Theorem~\ref{main} for type $A_{n-1}$ works for these types with suitable modifications.  In this section, we treat type $B_n$ and point out necessary modifications.   
In what follows, we use the notation
\begin{equation*}\label{notation_nbar}
\bar i \colonequals 2n+1-i \quad \text{for $1\le i\le 2n$}\quad\text{and}\quad \overline{[n]}\colonequals \{ \bar1, \dots, \bar n\}.
\end{equation*}

\subsection{Weight polytope of type $B_n$}\label{subsec_B_weight_poly}

Let $E$ be the $n$-dimensional subspace of $\RR^{2n}$ defined by 
\begin{equation}\label{eq_B_vector_sp}
E=\{(x_1,\ldots,x_n,x_{\bar n},\ldots,x_{\bar 1}) \in \RR^{2n} \mid x_i+x_{\bar i}=0 \ \textrm{for all} \ i \in [n] \}.
\end{equation}
Let $\Phi_{B_n}$ be the root system of type $B_n$ in $E$ generated by simple roots  
\begin{equation*}
\alpha_i=\begin{cases} t_i-t_{i+1} & \text{for } 1\leq i \leq n-1;\\
t_n & \text{for } i=n,
\end{cases}
\end{equation*}
where $t_i \in E$ has $1$ as its $i$-th coordinate, $-1$ as its $\bar i$-th coordinate, and all other coordinates $0$. 
The Weyl group $W_{B_n}$ of type $B_n$ is the group of \emph{signed permutations} on the set $[2n]=[n]\sqcup \nbar$, that is 
\begin{align*}
W_{B_n}=\{u \in \SS_{2n} \mid u(\bar{i})=\overline{u(i)} \textrm{ for all }  i \in [n] \}.
\end{align*}
This group acts on $\RR^{2n}$ by permuting coordinates and is generated by $s_1, \ldots, s_n$ where 
\begin{align*}
s_i&=(i, i+1)(\overline{i}, \overline{i+1}) \text{ for } 1 \leq i \leq n-1;\\
s_n&=(n, \overline{n}).
\end{align*}

Take a point $(a_1, \dots, a_n, a_{\bar n}, \dots, a_{\bar1})\in E$ with $a_1<\cdots<a_n<0$ and define the weight polytope of type $B_n$ by
\begin{equation*}
P_{B_n} \colonequals \text{\rm conv}\{ (a_{u(1)},\dots,a_{u(n)}, a_{u(\bar{n})},\cdots,a_{u(\bar 1)})\in E \mid u\in W_{B_n} \}.
\end{equation*}
This polytope is $n$-dimensional and sits in the $n$-dimensional vector space $E$. We describe 
$P_{B_2}$ and $P_{B_3}$ in Figure \ref{fig_weight_Poly_dim2_3TypeB}, where
the vertex $(a_{u(1)},\dots,a_{u(n)}, a_{u(\bar{n})},\dots,a_{u(\bar1)})$ is labeled by the first half of the one-line notation of $u\in W_{B_n}$. For instance, the vertex $(a_{\bar2},a_{\bar3},a_1,a_{\bar1},a_3,a_2)$ is labeled by 
$\bar2\bar3 1$.

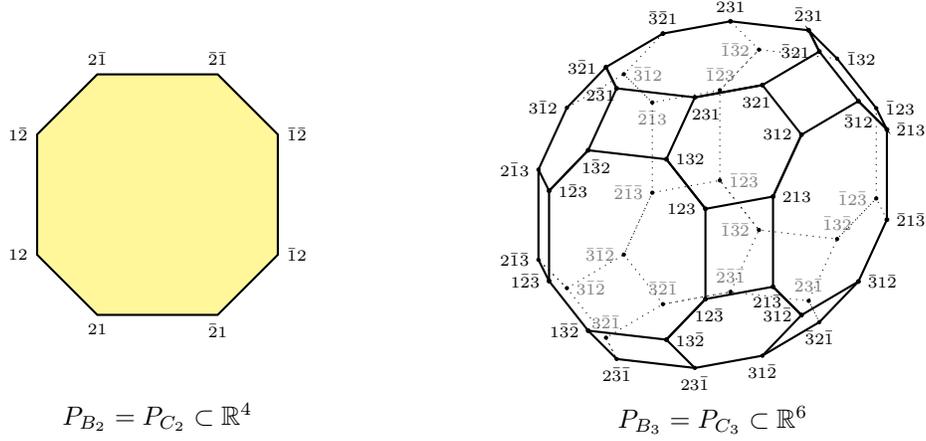
\begin{figure}
\begin{tikzpicture}
\begin{scope}[scale=0.8]
\draw[thick, fill=yellow!50] (-1,-2)--(1,-2)--(2,-1)--(2,1)--(1,2)--(-1,2)--(-2,1)--(-2,-1)--cycle;

\node[below] at (-1,-2) {\tiny$21$};
\node[left] at (-2,-1) {\tiny$12$};
\node[left] at (-2,1) {\tiny$1\bar2$};
\node[above] at (-1,2) {\tiny$2\bar1$};
\node[above] at (1,2) {\tiny$\bar2\bar1$};
\node[right] at (2,1) {\tiny$\bar1\bar2$};
\node[right] at (2,-1) {\tiny$\bar1 2$};
\node[below] at (1,-2) {\tiny$\bar21$};

\node at (0,-3.7) {$P_{B_2}=P_{C_2}\subset \RR^4$};
\end{scope}

\begin{scope}[scale=0.6, rotate around x=-90, rotate around z=-105, xshift=350]
\node at (0,0,-5) {$P_{B_3}=P_{C_3}\subset \RR^6$};


\draw (2,-3,1)--(3,-2,1)--(3,-2,-1)--(2,-3,-1)--cycle;

\draw[fill] (3,1,2) circle (1.2pt); 	\node[right] at (3,1,2) {\tiny$132$};
\draw[fill] (3,-1,2) circle (1.2pt); 	\node[below] at (3,-1+0.2,2) {\tiny$1\bar3 2$};
\draw[fill] (3,-2,1) circle (1.2pt); 	\node[right] at (3,-2,1) {\tiny$1\bar2 3$};
\draw[fill] (3,-2,-1) circle (1.2pt); \node[left] at (3,-2,-1) {\tiny$1 \bar2 \bar3$};
\draw[fill] (3,-1,-2) circle (1.2pt); \node[left] at (3,-1,-2) {\tiny$1 \bar3 \bar2$};
\draw[fill] (3,1,-2) circle (1.2pt); 	\node[right] at (3,1,-2) {\tiny$13\bar2$};
\draw[fill] (3,2,-1) circle (1.2pt); 	\node[below] at (3,2+0.1,-1) {\tiny$12\bar3$};
\draw[fill] (3,2,1) circle (1.2pt); 	\node[left] at (3,2,1) {\tiny$123$};

\draw[thick] (3,1,2)--(3,-1,2)--(3,-2,1)--(3,-2,-1)--(3,-1,-2)--(3,1,-2)--(3,2,-1)--(3,2,1)--cycle;

\draw[fill] (-3,1,2) circle (1pt); 	\node[right] at (-3,1,2) {\tiny$\bar1 32$};
\draw[fill] (-3,-1,2) circle (1pt); 	\node[left, gray] at (-3,-1,2) {\tiny$\bar1 \bar3 2$};
\draw[fill] (-3,-2,1) circle (1pt); 	\node[above, gray] at (-3,-2,1) {\tiny$\bar1 \bar2 3$};
\draw[fill] (-3,-2,-1) circle (1pt); 	\node[right,gray] at (-3,-2,-1) {\tiny$\bar1 \bar2 \bar3$};
\draw[fill] (-3,-1,-2) circle (1pt); 	\node[left, gray] at (-3,-1,-2) {\tiny$\bar1 \bar3 \bar2$};
\draw[fill] (-3,1,-2) circle (1pt); 	\node[above, gray] at (-3,1,-2) {\tiny$\bar1 3 \bar2$};
\draw[fill] (-3,2,-1) circle (1pt); 	\node[left, gray] at (-3,2,-1) {\tiny$\bar1 2 \bar3$};
\draw[fill] (-3,2,1) circle (1pt); 	\node[right] at (-3,2,1) {\tiny$\bar1 23$};

\draw[dotted] (-3,1,2)--(-3,-1,2)--(-3,-2,1)--(-3,-2,-1)--(-3,-1,-2)--(-3,1,-2)--(-3,2,-1)--(-3,2,1)--cycle;

\draw[fill] (1,3,2) circle (1.2pt);	\node[left] at (1,3,2) {\tiny$312$};
\draw[fill] (-1,3,2) circle (1.2pt); 	\node[below] at (-1,3,2-0.1) {\tiny$\bar3 12$};
\draw[fill] (-2,3,1) circle (1.2pt); 	\node[right] at (-2,3,1) {\tiny$\bar2 13$};
\draw[fill] (-2,3,-1) circle (1.2pt); \node[right] at (-2,3,-1) {\tiny$\bar2 1 \bar3 $};
\draw[fill] (-1,3,-2) circle (1.2pt);	\node[right] at (-1,3,-2) {\tiny$\bar3 1 \bar2$};
\draw[fill] (1,3,-2) circle (1.2pt);	\node[left] at (1,3,-2) {\tiny$31\bar2$}; 
\draw[fill] (2,3,-1) circle (1.2pt); 	\node[below] at (2+0.2,3,-1+0.1) {\tiny$21\bar3$};
\draw[fill] (2,3,1) circle (1.2pt); 	\node[right] at (2,3,1) {\tiny$213$};

\draw[thick] (1,3,2)--(-1,3,2)--(-2,3,1)--(-2,3,-1)--(-1,3,-2)--(1,3,-2)--(2,3,-1)--(2,3,1)--cycle;

\draw[fill] (1,-3,2) circle (1pt); 	\node[left] at (1,-3,2) {\tiny$3\bar1 2$};
\draw[fill] (-1,-3,2) circle (1pt); 	\node[right, gray] at (-1,-3,2) {\tiny$\bar3 \bar1 2$};
\draw[fill] (-2,-3,1) circle (1pt); 	\node[below, gray] at (-2,-3,1) {\tiny$\bar2 \bar1 3$};
\draw[fill] (-2,-3,-1) circle (1pt); 	\node[left, gray] at (-2,-3,-1) {\tiny$\bar2 \bar1 \bar3$};
\draw[fill] (-1,-3,-2) circle (1pt); 	\node[left, gray] at (-1,-3,-2) {\tiny$\bar3 \bar1 \bar2$};
\draw[fill] (1,-3,-2) circle (1pt); 	\node[right, gray] at (1,-3,-2) {\tiny$3\bar1 \bar2$};
\draw[fill] (2,-3,-1) circle (1pt); 	\node[left] at (2,-3,-1) {\tiny$2 \bar1 \bar3$};
\draw[fill] (2,-3,1) circle (1pt); 	\node[left] at (2,-3,1) {\tiny$2 \bar1 3$};

\draw[dotted] (1,-3,2)--(-1,-3,2)--(-2,-3,1)--(-2,-3,-1)--(-1,-3,-2)--(1,-3,-2)--(2,-3,-1)--(2,-3,1)--cycle;

\draw[fill] (1,2,3) circle (1.2pt);	\node[below] at (1.2,2,3) {\tiny$321$};
\draw[fill] (1,-2,3) circle (1.2pt); 	\node[left] at (1,-2,3) {\tiny$3\bar 2 1$};
\draw[fill] (-1,2,3) circle (1.2pt); 	\node[left] at (-1,2,3) {\tiny$\bar3 2 1$};
\draw[fill] (-1,-2,3) circle (1.2pt); \node[above] at (-1,-2,3) {\tiny$\bar3 \bar2 1$};
\draw[fill] (2,1,3) circle (1.2pt); 	\node[below] at (2-0.4,1,3-0.2) {\tiny$231$};
\draw[fill] (2,-1,3) circle (1.2pt); 	\node[left] at (2,-1+0.2,3-0.12) {\tiny$2 \bar3 1$};
\draw[fill] (-2,1,3) circle (1.2pt); 	\node[above] at (-2,1,3) {\tiny$\bar2 31$};
\draw[fill] (-2,-1,3) circle (1.2pt);	\node[above] at (-2,-1,3) {\tiny$\bar2 \bar3 1$};

\draw[thick] (1,2,3)--(-1,2,3)--(-2,1,3)--(-2,-1,3)--(-1,-2,3)--(1,-2,3)--(2,-1,3)--(2,1,3)--cycle;

\draw[fill] (1,2,-3) circle (1pt); 	\node[below] at (1,2,-3) {\tiny$31\bar2$};
\draw[fill] (1,-2,-3) circle (1pt); 	\node[above,gray] at (1,-2,-3) {\tiny$3\bar2\bar1$};
\draw[fill] (-1,2,-3) circle (1pt); 	\node[below] at (-1,2,-3) {\tiny$\bar3 2\bar1$};
\draw[fill] (-1,-2,-3) circle (1pt); 	\node[above, gray] at (-1,-2,-3) {\tiny$\bar3 \bar 2\bar1$};
\draw[fill] (2,1,-3) circle (1pt);  	\node[below] at (2,1,-3) {\tiny$23 \bar1$};
\draw[fill] (2,-1,-3) circle (1pt); 	\node[below] at (2,-1,-3) {\tiny$2 \bar3 \bar1$};
\draw[fill] (-2,1,-3) circle (1pt); 	\node[above, gray] at (-2,1,-3) {\tiny$\bar2 3 \bar1$};
\draw[fill] (-2,-1,-3) circle (1pt);  	\node[above, gray] at (-2,-1,-3) {\tiny$\bar 2 \bar3 \bar1$};

\draw[dotted] (1,2,-3)--(-1,2,-3)--(-2,1,-3)--(-2,-1,-3)--(-1,-2,-3)--(1,-2,-3)--(2,-1,-3)--(2,1,-3)--cycle;

\draw[thick] (1,2,3)--(2,1,3)--(3,1,2)--(3,2,1)--(2,3,1)--(1,3,2)--cycle;

\draw[thick] (-1,2,3)--(-2,1,3)--(-3,1,2)--(-3,2,1)--(-2,3,1)--(-1,3,2)--cycle;

\draw[dotted] (-1,2,-3)--(-2,1,-3)--(-3,1,-2)--(-3,2,-1)--(-2,3,-1)--(-1,3,-2)--cycle;

\draw[thick] (1,-2,3)--(2,-1,3)--(3,-1,2)--(3,-2,1)--(2,-3,1)--(1,-3,2)--cycle;

\draw[dotted] (1,-2,-3)--(2,-1,-3)--(3,-1,-2)--(3,-2,-1)--(2,-3,-1)--(1,-3,-2)--cycle;

\draw[dotted] (-1,-2,-3)--(-2,-1,-3)--(-3,-1,-2)--(-3,-2,-1)--(-2,-3,-1)--(-1,-3,-2)--cycle;

\draw[thick] (1,2,-3)--(2,1,-3)--(3,1,-2)--(3,2,-1)--(2,3,-1)--(1,3,-2)--cycle;

\draw[dotted] (-1,-2,3)--(-2,-1,3)--(-3,-1,2)--(-3,-2,1)--(-2,-3,1)--(-1,-3,2)--cycle;

\draw[thick] (2,-3,1)--(2,-3,-1)--(3,-2,-1);
\draw[thick] (3,-1,-2)--(2,-1,-3)--(2,1,-3);
\draw[thick] (1,2,-3)--(-1,2,-3)--(-1,3,-2);
\end{scope}

\end{tikzpicture}
\caption{Weight polytopes $P_{B_n}$ of dimension $2$ and $3$.}
\label{fig_weight_Poly_dim2_3TypeB}
\end{figure}

\begin{definition}\label{definition:signedsubset}
A nonempty subset $\indI$ of $[n]\sqcup \overline{[n]}$ is said to be \textit{signed} if 
$| I \cap \{i, \bar{i}\} |\leq 1$ for all $i\in [n]$. We denote by $\M_{B_n}$ the set of all signed subsets of $[n]\sqcup \overline{[n]}$. 
\end{definition}

Every facet of $P_{B_n}$ is of the form 
\begin{equation}\label{eq_B_facet_convex_hull}
F(\indI)\colonequals \left\{(x_1,\dots,x_n,x_{\bar n},\dots,x_{\bar 1})\in P_{B_n} \mid \sum_{i\in I} \x_i=\sum_{i\in [|I|]} a_i\right\}  
\end{equation}
with $\indI\in\M_{B_n}$, and every such $F(\indI)$ is a facet. An (inward) normal vector to $F(I)$ in $E$, regarded as an element of the space $E^*$ dual to $E$, is given by $e_I\colonequals\sum_{i\in I}e_i$, where $\{e_1,\dots,e_n\}$ is the basis of $E^*$ dual to the basis $\{t_1,\dots,t_n\}$ of $E$ and $e_{\bar{i}}$ for $i\in [n]$ is defined to be $-e_i$.  
With this understood, Proposition~\ref{prop:A_cohom_of_perm} holds for type $B_n$. 

\begin{remark}
The weight polytope $P_{B_n}$ is sometimes called a \emph{truncated cuboctahedron}. It is combinatorially equivalent to the polytope obtained by truncating all proper faces of the $n$-cube $I^n$. The correspondence between facets of $P_{B_n}$ and $\M_{B_n}$ given in \eqref{eq_B_facet_convex_hull} agrees with the labels on faces of $I^n$ given by $i$ and $\bar i$ on each pair of opposite facets of $I^n$ and their natural extension to faces according to the intersection relation. See Figure \ref{fig_labeling_on_PBn} for the case when $n=2,3$. 
\end{remark}

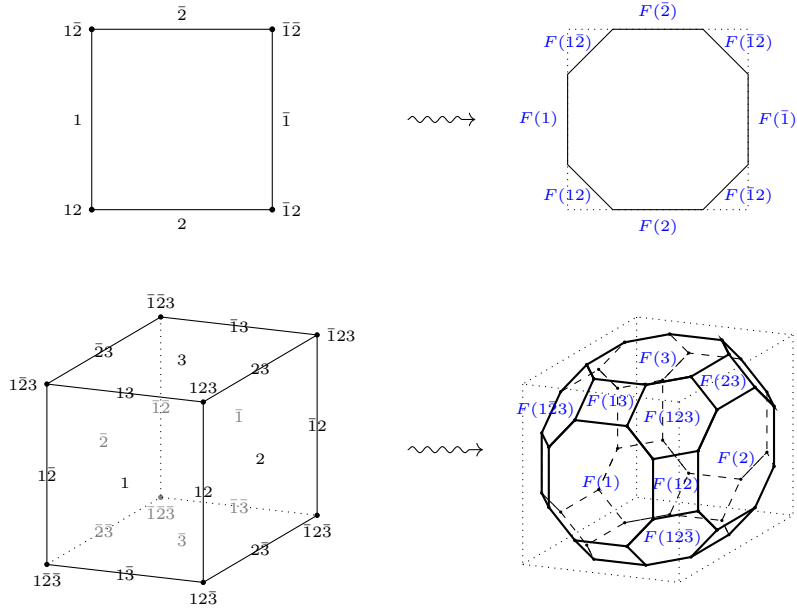
\begin{figure}
\begin{tikzpicture}
\begin{scope}[yshift=125, scale=0.6]
\draw(-2,-2)--(-2,2)--(2,2)--(2,-2)--cycle;

\draw[fill] (-2,-2) circle (1.5pt);
\draw[fill] (2,-2) circle (1.5pt);
\draw[fill] (-2,2) circle (1.5pt);
\draw[fill] (2,2) circle (1.5pt);

\node[left] at (-2,-2) {\tiny$12$};
\node[left] at (-2,2) {\tiny$1\bar2$};
\node[right] at (2,2) {\tiny$\bar1\bar2$};
\node[right] at (2,-2) {\tiny$\bar12$};

\node[left] at (-2,0) {\tiny$1$};
\node[right] at (2,0) {\tiny$\bar1$};
\node[below] at (0,-2) {\tiny$2$};
\node[above] at (0,2) {\tiny$\bar2$};

\draw[->,snake=snake, segment amplitude=.4mm, segment length=2mm, line after snake=1mm] (5,0)--(6.5,0);

\begin{scope}[xshift=300]
\draw[dotted] (-2,-2)--(-2,2)--(2,2)--(2,-2)--cycle;
\draw (-1,-2)--(1,-2)--(2,-1)--(2,1)--(1,2)--(-1,2)--(-2,1)--(-2,-1)--cycle;

\node[blue] at (-2,-1.7) {\tiny$F(12)$};
\node[blue] at (-2,1.7) {\tiny$F(1\bar2)$};
\node[blue] at (2,1.7) {\tiny$F(\bar1\bar2)$};
\node[blue] at (2,-1.7) {\tiny$F(\bar12)$};

\node[blue, left] at (-2,0) {\tiny$F(1)$};
\node[blue, right] at (2,0) {\tiny$F(\bar1)$};
\node[blue, below] at (0,-2) {\tiny$F(2)$};
\node[blue, above] at (0,2) {\tiny$F(\bar2)$};
\end{scope}

\end{scope}

\begin{scope}
\draw[->,snake=snake, segment amplitude=.4mm, segment length=2mm, line after snake=1mm] (3,0)--(4,0);

\begin{scope}[scale=0.4, rotate around x=-90, rotate around z=-105]
\draw[fill] (3,3,3) circle (2pt); 	\node[above] at (3,3,3) {\tiny$123$};
\draw[fill] (3,-3,3) circle (2pt); 	\node[left] at (3,-3,3) {\tiny$1\bar2 3$};
\draw[fill] (-3,-3,3) circle (2pt); 	\node[above] at (-3,-3,3) {\tiny$\bar1 \bar2 3$};
\draw[fill] (-3,3,3) circle (2pt); 	\node[right] at (-3,3,3) {\tiny$\bar1 2 3$};
\draw[fill] (3,3,-3) circle (2pt); 	\node[below] at (3,3,-3) {\tiny$12\bar3$};
\draw[fill] (-3,3,-3) circle (2pt); 	\node[below] at (-3,3,-3) {\tiny$\bar12\bar3$};
\draw[fill] (3,-3,-3) circle (2pt); 	\node[below] at (3,-3,-3) {\tiny$1\bar2\bar3$};
\draw[fill=gray,gray] (-3,-3,-3) circle (2pt); 	\node[below, gray] at (-3,-3,-3) {\tiny$\bar1\bar2\bar3$};

\draw (3,3,3)--(-3,3,3)--(-3,-3,3)--(3,-3,3)--cycle;
\draw (3,-3,-3)--(3,3,-3)--(-3,3,-3);
\draw[dotted] (3,-3,-3)--(-3,-3,-3)--(-3,3,-3);

\draw (3,3,3)--(3,3,-3);
\draw (-3,3,3)--(-3,3,-3);
\draw (3,-3,3)--(3,-3,-3);
\draw[dotted] (-3,-3,3)--(-3,-3,-3);

\node at (3,3,0) {\tiny$12$};
\node at (3,0,3) {\tiny$13$};
\node at (0,3,3) {\tiny$23$};
\node at (-3,3,0) {\tiny$\bar1 2$};
\node at (-3,0,3) {\tiny$\bar1 3$};
\node at (0,-3,3) {\tiny$\bar2 3$};
\node at (3,-3,0) {\tiny$1\bar2$};
\node at (3,0,-3) {\tiny$1\bar3$};
\node at (0,3,-3) {\tiny$2\bar3$};
\node[gray] at (-3,-3,0) {\tiny$\bar1\bar2$};
\node[gray] at (-3,0,-3) {\tiny$\bar1\bar3$};
\node[gray] at (0,-3,-3) {\tiny$\bar2\bar3$};

\node at (3,0,0) {\tiny$1$};
\node at (0,3,0) {\tiny$2$};
\node at (0,0,3) {\tiny$3$};
\node[gray] at (-3,0,0) {\tiny$\bar1$};
\node[gray] at (0,-3,0) {\tiny$\bar2$};
\node[gray] at (0,0,-3) {\tiny$\bar3$};


\begin{scope}[xshift=450]

\draw[dotted] (3,3,3)--(-3,3,3)--(-3,-3,3)--(3,-3,3)--cycle;
\draw[dotted] (3,-3,-3)--(3,3,-3)--(-3,3,-3);
\draw[dotted] (3,-3,-3)--(-3,-3,-3)--(-3,3,-3);

\draw[dotted] (3,3,3)--(3,3,-3);
\draw[dotted] (-3,3,3)--(-3,3,-3);
\draw[dotted] (3,-3,3)--(3,-3,-3);
\draw[dotted] (-3,-3,3)--(-3,-3,-3);

\node[blue] at (2,2,2) {\tiny$F(123)$};
\node[blue] at (2.5,2.5,0) {\tiny$F(12)$};
\node[blue] at (2,2,-2) {\tiny$F(12\bar3)$};
\node[blue] at (2.5,0,2.5) {\tiny$F(13)$};
\node[blue] at (0,2.5,2.5) {\tiny$F(23)$};

\draw[fill] (1,2,-3) circle (1pt); 	
\draw[fill] (1,-2,-3) circle (1pt); 	
\draw[fill] (-1,2,-3) circle (1pt);	
\draw[fill] (-1,-2,-3) circle (1pt);	
\draw[fill] (2,1,-3) circle (1pt);  	
\draw[fill] (2,-1,-3) circle (1pt); 	
\draw[fill] (-2,1,-3) circle (1pt); 	
\draw[fill] (-2,-1,-3) circle (1pt);  

\draw[dashed] (1,2,-3)--(-1,2,-3)--(-2,1,-3)--(-2,-1,-3)--(-1,-2,-3)--(1,-2,-3)--(2,-1,-3)--(2,1,-3)--cycle;

\draw[fill] (1,2,3) circle (1.2pt);	
\draw[fill] (1,-2,3) circle (1.2pt); 	
\draw[fill] (-1,2,3) circle (1.2pt); 	
\draw[fill] (-1,-2,3) circle (1.2pt); 
\draw[fill] (2,1,3) circle (1.2pt); 	
\draw[fill] (2,-1,3) circle (1.2pt); 	
\draw[fill] (-2,1,3) circle (1.2pt); 	
\draw[fill] (-2,-1,3) circle (1.2pt);	

\draw[thick] (1,2,3)--(-1,2,3)--(-2,1,3)--(-2,-1,3)--(-1,-2,3)--(1,-2,3)--(2,-1,3)--(2,1,3)--cycle;
\node[blue] at (0,0,3) {\tiny$F(3)$};

\draw[fill] (-3,1,2) circle (1pt); 	
\draw[fill] (-3,-1,2) circle (1pt); 	
\draw[fill] (-3,-2,1) circle (1pt); 	
\draw[fill] (-3,-2,-1) circle (1pt); 	
\draw[fill] (-3,-1,-2) circle (1pt); 	
\draw[fill] (-3,1,-2) circle (1pt); 	
\draw[fill] (-3,2,-1) circle (1pt); 	
\draw[fill] (-3,2,1) circle (1pt); 	

\draw[dashed] (-3,1,2)--(-3,-1,2)--(-3,-2,1)--(-3,-2,-1)--(-3,-1,-2)--(-3,1,-2)--(-3,2,-1)--(-3,2,1)--cycle;

\draw[fill] (3,1,2) circle (1.2pt); 	
\draw[fill] (3,-1,2) circle (1.2pt); 	
\draw[fill] (3,-2,1) circle (1.2pt); 	
\draw[fill] (3,-2,-1) circle (1.2pt); 
\draw[fill] (3,-1,-2) circle (1.2pt); 
\draw[fill] (3,1,-2) circle (1.2pt); 	
\draw[fill] (3,2,-1) circle (1.2pt); 	
\draw[fill] (3,2,1) circle (1.2pt); 	

\draw[thick] (3,1,2)--(3,-1,2)--(3,-2,1)--(3,-2,-1)--(3,-1,-2)--(3,1,-2)--(3,2,-1)--(3,2,1)--cycle;
\node[blue] at (3,0,0) {\tiny$F(1)$};

\draw[fill] (1,-3,2) circle (1pt); 	
\draw[fill] (-1,-3,2) circle (1pt); 	
\draw[fill] (-2,-3,1) circle (1pt); 	
\draw[fill] (-2,-3,-1) circle (1pt); 	
\draw[fill] (-1,-3,-2) circle (1pt); 	
\draw[fill] (1,-3,-2) circle (1pt); 	
\draw[fill] (2,-3,-1) circle (1pt); 	
\draw[fill] (2,-3,1) circle (1pt); 	

\draw[dashed] (1,-3,2)--(-1,-3,2)--(-2,-3,1)--(-2,-3,-1)--(-1,-3,-2)--(1,-3,-2)--(2,-3,-1)--(2,-3,1)--cycle;

\draw[fill] (1,3,2) circle (1.2pt);	
\draw[fill] (-1,3,2) circle (1.2pt); 	
\draw[fill] (-2,3,1) circle (1.2pt); 	
\draw[fill] (-2,3,-1) circle (1.2pt); 
\draw[fill] (-1,3,-2) circle (1.2pt);	
\draw[fill] (1,3,-2) circle (1.2pt);	
\draw[fill] (2,3,-1) circle (1.2pt); 	
\draw[fill] (2,3,1) circle (1.2pt); 	

\draw[thick] (1,3,2)--(-1,3,2)--(-2,3,1)--(-2,3,-1)--(-1,3,-2)--(1,3,-2)--(2,3,-1)--(2,3,1)--cycle;
\node[blue] at (0,3,0) {\tiny$F(2)$};

\draw[thick] (1,2,3)--(2,1,3)--(3,1,2)--(3,2,1)--(2,3,1)--(1,3,2)--cycle;
\draw[thick] (-1,2,3)--(-2,1,3)--(-3,1,2)--(-3,2,1)--(-2,3,1)--(-1,3,2)--cycle;
\draw[dashed] (-1,2,-3)--(-2,1,-3)--(-3,1,-2)--(-3,2,-1)--(-2,3,-1)--(-1,3,-2)--cycle;
\draw[thick] (1,-2,3)--(2,-1,3)--(3,-1,2)--(3,-2,1)--(2,-3,1)--(1,-3,2)--cycle;
\node[blue] at (2.5,-2.5,2 ){\tiny$F(1\bar23)$};
\draw[dashed] (1,-2,-3)--(2,-1,-3)--(3,-1,-2)--(3,-2,-1)--(2,-3,-1)--(1,-3,-2)--cycle;
\draw[dashed] (-1,-2,-3)--(-2,-1,-3)--(-3,-1,-2)--(-3,-2,-1)--(-2,-3,-1)--(-1,-3,-2)--cycle;
\draw[thick] (1,2,-3)--(2,1,-3)--(3,1,-2)--(3,2,-1)--(2,3,-1)--(1,3,-2)--cycle;
\draw[dashed] (-1,-2,3)--(-2,-1,3)--(-3,-1,2)--(-3,-2,1)--(-2,-3,1)--(-1,-3,2)--cycle;

\draw[thick] (2,-3,1)--(2,-3,-1)--(3,-2,-1);
\draw[thick] (3,-1,-2)--(2,-1,-3)--(2,1,-3);
\draw[thick] (1,2,-3)--(-1,2,-3)--(-1,3,-2);
\end{scope}

\end{scope}
\end{scope}

\end{tikzpicture}
\caption{Labeling on facets of $P_{B_n}$ for $n=2,3$ and their relationship with labeling on faces of $I^n$ }
\label{fig_labeling_on_PBn}
\end{figure}

Let $K$ be a subset of $[n]$ and $W_K$ the subgroup of $W_{B_n}$ generated by $\{s_k:k\in K\}$.  Then, the $W_K$-orbit decomposition of $[n]\sqcup\nbar$ is of the following form:
\begin{equation} \label{eq:B_orbit_decomposition}
[n]\sqcup \nbar= \bigsqcup_{i=1}^{m}(N_i \sqcup \overline{N_i}) \sqcup N',
\end{equation} 
where $m$ is some non-negative integer, $N_i \subset [n]$ for each $i$, $N'=\emptyset$ if $n\notin K$, and $\{n, \bar n \} \subset N' = \overline{N'}$ if $n\in K$. Here 
\begin{equation*}\label{notation_overline[N]}
\overline{N}\colonequals \{\bar i \mid i\in N\}
\end{equation*}
for $N\subset [n]\sqcup \nbar$. 

\begin{example}\label{ex_B_3_orbit_decomp}
Let $n=3$.  
\begin{enumerate}
\item If $K=\{1\}$, then $[3]\sqcup \overline{[3]}=\{1,2\} \sqcup \{\bar2, \bar1\}\sqcup  \{3\}\sqcup \{\bar 3\}.$
\item If $K=\{1,2\}$, then $[3]\sqcup \overline{[3]}=\{1,2,3\} \sqcup \{\bar3, \bar2, \bar1\}.$
\item If $K=\{2,3\}$, then $[3]\sqcup \overline{[3]}=\{1\}\sqcup \{\bar 1\} \sqcup \{2,3,\bar3, \bar2\}$. 
\end{enumerate}
\end{example}

With this understood, it is straightforward to prove Lemma \ref{lem_invariants} for type $B_n$ and Proposition \ref{prop_A_generated_by_deg2} holds with $\An$, $\SS_K$ and $\mathfrak{F}_{A_{n-1}}$ replaced by $B_n$, $W_K$ and $\mathfrak{F}_{B_n}$ respectively.

\subsection{Partitioned weight polytope of type $B_n$}\label{subsec_B_part_weight_poly}

If $\Phi$ is of type $B_n$, then, for every $k \in [n]$, the half-space $\Hk^\leq$ satisfies   
\begin{align*} \label{eq_Hk}
\Hk^{\leq}& = \left\{ x \in E \mid x_k \leq x_{k+1}\right\}, \quad \text{for }  1 \leq k \leq n-1;\\
\Hn^{\leq} &= \left\{ x \in E ~\Big|~ x_n \leq 0 \right\}, 
\end{align*}
and the partitioned permutohedron $P_{B_n}(K)$ associated with a subset $K\subset[n]$ is defined by  
\[
P_{B_n}(K)\colonequals P_{B_n}\cap \bigcap_{k\in K}H(k)^\leq.
\] 
See Figure \ref{fig_part_weight_poly_B} for several examples in dimension 3.

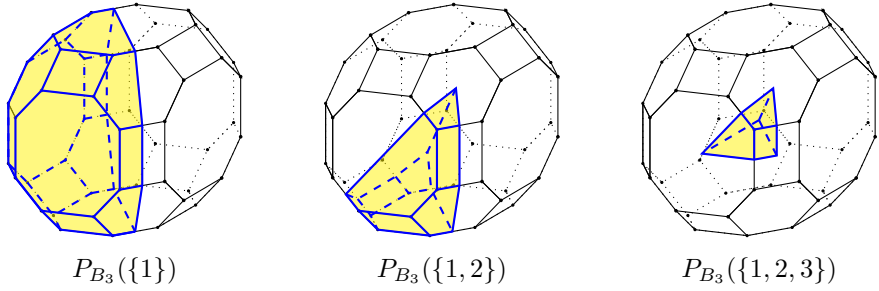
\begin{figure}
\begin{tikzpicture}

\begin{scope}[scale=0.4, rotate around x=-90, rotate around z=-105]

\coordinate (132) at (3,1,2) ;
\coordinate (1b32) at (3,-1,2) ;
\coordinate (1b23) at (3,-2,1) ;
\coordinate (1b2b3) at (3,-2,-1) ;
\coordinate (1b3b2) at (3,-1,-2) ;
\coordinate (13b2) at (3,1,-2) ;
\coordinate (12b3) at (3,2,-1) ;
\coordinate (123) at (3,2,1) ;

\coordinate (b132) at (-3,1,2) ;
\coordinate (b1b32) at (-3,-1,2) ;
\coordinate (b1b23) at (-3,-2,1) ;
\coordinate (b1b2b3) at (-3,-2,-1) ;
\coordinate (b1b1b3) at (-3,-1,-2) ;
\coordinate (b13b2) at (-3,1,-2) ;
\coordinate (b12b3) at (-3,2,-1) ;
\coordinate (b123) at (-3,2,1) ;

\coordinate (312) at (1,3,2) ;
\coordinate (b312) at (-1,3,2) ;
\coordinate (b213) at (-2,3,1) ;
\coordinate (b21b3) at (-2,3,-1) ;
\coordinate (b31b2) at (-1,3,-2) ;
\coordinate (31b2) at (1,3,-2) ;
\coordinate (21b3) at (2,3,-1) ;
\coordinate (213) at (2,3,1) ;

\coordinate (3b12) at (1,-3,2) ;
\coordinate (b3b12) at (-1,-3,2) ;
\coordinate (b2b13) at (-2,-3,1) ;
\coordinate (b2b1b3) at (-2,-3,-1) ;
\coordinate (b3b1b2) at (-1,-3,-2) ;
\coordinate (3b1b2) at (1,-3,-2) ;
\coordinate (2b1b3) at (2,-3,-1) ;
\coordinate (2b13) at (2,-3,1) ;

\coordinate (321) at (1,2,3) ;
\coordinate (3b21) at (1,-2,3) ;
\coordinate (b321) at (-1,2,3) ;
\coordinate (b3b21) at (-1,-2,3) ;
\coordinate (231) at (2,1,3) ;
\coordinate (2b31) at (2,-1,3) ;
\coordinate (b231) at (-2,1,3) ;
\coordinate (b2b31) at (-2,-1,3) ;

\coordinate (31b2) at (1,2,-3) ;
\coordinate (3b2b1) at (1,-2,-3) ;
\coordinate (b32b1) at (-1,2,-3) ;
\coordinate (b3b2b1) at (-1,-2,-3) ;
\coordinate (23b1) at (2,1,-3) ;
\coordinate (2b3b1) at (2,-1,-3) ;
\coordinate (b23b1) at (-2,1,-3) ;
\coordinate (b2b3b1) at (-2,-1,-3) ;

\draw (2,-3,1)--(3,-2,1)--(3,-2,-1)--(2,-3,-1)--cycle;

\draw[fill] (3,1,2) circle (1.2pt); 	
\draw[fill] (3,-1,2) circle (1.2pt); 	
\draw[fill] (3,-2,1) circle (1.2pt); 	
\draw[fill] (3,-2,-1) circle (1.2pt); 
\draw[fill] (3,-1,-2) circle (1.2pt); 
\draw[fill] (3,1,-2) circle (1.2pt); 	
\draw[fill] (3,2,-1) circle (1.2pt); 	
\draw[fill] (3,2,1) circle (1.2pt); 	

\draw (3,1,2)--(3,-1,2)--(3,-2,1)--(3,-2,-1)--(3,-1,-2)--(3,1,-2)--(3,2,-1)--(3,2,1)--cycle;

\draw[fill] (-3,1,2) circle (1pt); 	
\draw[fill] (-3,-1,2) circle (1pt); 	
\draw[fill] (-3,-2,1) circle (1pt); 	
\draw[fill] (-3,-2,-1) circle (1pt); 	
\draw[fill] (-3,-1,-2) circle (1pt); 	
\draw[fill] (-3,1,-2) circle (1pt); 	
\draw[fill] (-3,2,-1) circle (1pt); 	
\draw[fill] (-3,2,1) circle (1pt); 	

\draw[dotted] (-3,1,2)--(-3,-1,2)--(-3,-2,1)--(-3,-2,-1)--(-3,-1,-2)--(-3,1,-2)--(-3,2,-1)--(-3,2,1)--cycle;

\draw[fill] (1,3,2) circle (1.2pt);	
\draw[fill] (-1,3,2) circle (1.2pt); 	
\draw[fill] (-2,3,1) circle (1.2pt); 	
\draw[fill] (-2,3,-1) circle (1.2pt); 
\draw[fill] (-1,3,-2) circle (1.2pt);	
\draw[fill] (1,3,-2) circle (1.2pt);	
\draw[fill] (2,3,-1) circle (1.2pt); 	
\draw[fill] (2,3,1) circle (1.2pt); 	

\draw (1,3,2)--(-1,3,2)--(-2,3,1)--(-2,3,-1)--(-1,3,-2)--(1,3,-2)--(2,3,-1)--(2,3,1)--cycle;

\draw[fill] (1,-3,2) circle (1pt); 	
\draw[fill] (-1,-3,2) circle (1pt); 	
\draw[fill] (-2,-3,1) circle (1pt); 	
\draw[fill] (-2,-3,-1) circle (1pt); 	
\draw[fill] (-1,-3,-2) circle (1pt); 	
\draw[fill] (1,-3,-2) circle (1pt); 	
\draw[fill] (2,-3,-1) circle (1pt); 	
\draw[fill] (2,-3,1) circle (1pt); 	

\draw[dotted] (1,-3,2)--(-1,-3,2)--(-2,-3,1)--(-2,-3,-1)--(-1,-3,-2)--(1,-3,-2)--(2,-3,-1)--(2,-3,1)--cycle;

\draw[fill] (1,2,3) circle (1.2pt);	
\draw[fill] (1,-2,3) circle (1.2pt);  
\draw[fill] (-1,2,3) circle (1.2pt); 	
\draw[fill] (-1,-2,3) circle (1.2pt); 
\draw[fill] (2,1,3) circle (1.2pt); 	
\draw[fill] (2,-1,3) circle (1.2pt); 	
\draw[fill] (-2,1,3) circle (1.2pt); 	
\draw[fill] (-2,-1,3) circle (1.2pt);	

\draw (1,2,3)--(-1,2,3)--(-2,1,3)--(-2,-1,3)--(-1,-2,3)--(1,-2,3)--(2,-1,3)--(2,1,3)--cycle;

\draw[fill] (1,2,-3) circle (1pt); 	
\draw[fill] (1,-2,-3) circle (1pt); 	
\draw[fill] (-1,2,-3) circle (1pt); 	
\draw[fill] (-1,-2,-3) circle (1pt); 	
\draw[fill] (2,1,-3) circle (1pt);  	
\draw[fill] (2,-1,-3) circle (1pt); 	
\draw[fill] (-2,1,-3) circle (1pt); 	
\draw[fill] (-2,-1,-3) circle (1pt);  	

\draw[dotted] (1,2,-3)--(-1,2,-3)--(-2,1,-3)--(-2,-1,-3)--(-1,-2,-3)--(1,-2,-3)--(2,-1,-3)--(2,1,-3)--cycle;

\draw (1,2,3)--(2,1,3)--(3,1,2)--(3,2,1)--(2,3,1)--(1,3,2)--cycle;

\draw (-1,2,3)--(-2,1,3)--(-3,1,2)--(-3,2,1)--(-2,3,1)--(-1,3,2)--cycle;

\draw[dotted] (-1,2,-3)--(-2,1,-3)--(-3,1,-2)--(-3,2,-1)--(-2,3,-1)--(-1,3,-2)--cycle;

\draw (1,-2,3)--(2,-1,3)--(3,-1,2)--(3,-2,1)--(2,-3,1)--(1,-3,2)--cycle;

\draw[dotted] (1,-2,-3)--(2,-1,-3)--(3,-1,-2)--(3,-2,-1)--(2,-3,-1)--(1,-3,-2)--cycle;

\draw[dotted] (-1,-2,-3)--(-2,-1,-3)--(-3,-1,-2)--(-3,-2,-1)--(-2,-3,-1)--(-1,-3,-2)--cycle;

\draw (1,2,-3)--(2,1,-3)--(3,1,-2)--(3,2,-1)--(2,3,-1)--(1,3,-2)--cycle;

\draw[dotted] (-1,-2,3)--(-2,-1,3)--(-3,-1,2)--(-3,-2,1)--(-2,-3,1)--(-1,-3,2)--cycle;

\draw (2,-3,1)--(2,-3,-1)--(3,-2,-1);
\draw (3,-1,-2)--(2,-1,-3)--(2,1,-3);
\draw (1,2,-3)--(-1,2,-3)--(-1,3,-2);

\draw[fill=yellow, opacity=0.5] ($0.5*(b3b21)+0.5*(b2b31)$)--($0.5*(231)+0.5*(321)$)--($0.5*(123)+0.5*(213)$)--($0.5*(12b3)+0.5*(21b3)$)--($0.5*(23b1)+0.5*(31b2)$)--(23b1)--(2b3b1)--(1b3b2)--(1b2b3)--(2b1b3)--(2b13)--(3b12)--(3b21)--(b3b21)--cycle;

\draw[thick, blue] ($0.5*(b3b21)+0.5*(b2b31)$)--($0.5*(231)+0.5*(321)$)--($0.5*(123)+0.5*(213)$)--($0.5*(12b3)+0.5*(21b3)$)--($0.5*(23b1)+0.5*(31b2)$);
\draw[thick, dashed, blue] ($0.5*(b3b21)+0.5*(b2b31)$)--($0.5*(b2b13)+0.5*(b1b23)$)--($0.5*(b2b1b3)+0.5*(b1b2b3)$)--($0.5*(b3b2b1)+0.5*(b2b3b1)$)--($0.5*(23b1)+0.5*(31b2)$);
\draw[thick, blue] ($0.5*(23b1)+0.5*(31b2)$)--(23b1)--(2b3b1)--(1b3b2)--(1b2b3)--(2b1b3)--(2b13)--(3b12)--(3b21)--(b3b21)--($0.5*(b3b21)+0.5*(b2b31)$);
\draw[blue, dashed, thick] (1,-3,2)--(-1,-3,2)--(-2,-3,1)--(-2,-3,-1)--(-1,-3,-2)--(1,-3,-2)--(2,-3,-1)--(2,-3,1)--cycle;

\draw[blue, thick, dashed] (b2b13)--($1/2*(b1b23)+1/2*(b2b13)$);
\draw[blue, thick, dashed] (b2b1b3)--($1/2*(b1b2b3)+1/2*(b2b1b3)$);
\draw[blue, thick, dashed] (b3b2b1)--($1/2*(b2b3b1)+1/2*(b3b2b1)$);
\draw[blue, thick, dashed] (3b1b2)--(3b2b1)--(b3b2b1)--(b3b1b2);
\draw[blue, thick, dashed] (b3b12)--(b3b21);
\draw[blue, thick, dashed] (2b3b1)--(3b2b1);

\draw[blue, thick] (3,1,2)--(3,-1,2)--(3,-2,1)--(3,-2,-1)--(3,-1,-2)--(3,1,-2)--(3,2,-1)--(3,2,1)--cycle;
\draw[blue, thick] (3b21)--(2b31)--(231)--($0.5*(231)+0.5*(321)$);
\draw[blue, thick] (2b31)--(1b32);
\draw[blue, thick] (132)--(231);
\draw[blue, thick] (13b2)--(23b1);
\draw[blue, thick] (123)--($1/2*(123)+1/2*(213)$);
\draw[blue, thick] (12b3)--($1/2*(12b3)+1/2*(21b3)$);
\draw[blue, thick] (23b1)--($1/2*(23b1)+1/2*(31b2)$);
\node at (0,0,-5) {$P_{B_3}(\{1\})$};

\end{scope}

\begin{scope}[scale=0.4, rotate around x=-90, rotate around z=-105, xshift=300]

\coordinate (132) at (3,1,2) ;
\coordinate (1b32) at (3,-1,2) ;
\coordinate (1b23) at (3,-2,1) ;
\coordinate (1b2b3) at (3,-2,-1) ;
\coordinate (1b3b2) at (3,-1,-2) ;
\coordinate (13b2) at (3,1,-2) ;
\coordinate (12b3) at (3,2,-1) ;
\coordinate (123) at (3,2,1) ;

\coordinate (b132) at (-3,1,2) ;
\coordinate (b1b32) at (-3,-1,2) ;
\coordinate (b1b23) at (-3,-2,1) ;
\coordinate (b1b2b3) at (-3,-2,-1) ;
\coordinate (b1b1b3) at (-3,-1,-2) ;
\coordinate (b13b2) at (-3,1,-2) ;
\coordinate (b12b3) at (-3,2,-1) ;
\coordinate (b123) at (-3,2,1) ;

\coordinate (312) at (1,3,2) ;
\coordinate (b312) at (-1,3,2) ;
\coordinate (b213) at (-2,3,1) ;
\coordinate (b21b3) at (-2,3,-1) ;
\coordinate (b31b2) at (-1,3,-2) ;
\coordinate (31b2) at (1,3,-2) ;
\coordinate (21b3) at (2,3,-1) ;
\coordinate (213) at (2,3,1) ;

\coordinate (3b12) at (1,-3,2) ;
\coordinate (b3b12) at (-1,-3,2) ;
\coordinate (b2b13) at (-2,-3,1) ;
\coordinate (b2b1b3) at (-2,-3,-1) ;
\coordinate (b3b1b2) at (-1,-3,-2) ;
\coordinate (3b1b2) at (1,-3,-2) ;
\coordinate (2b1b3) at (2,-3,-1) ;
\coordinate (2b13) at (2,-3,1) ;

\coordinate (321) at (1,2,3) ;
\coordinate (3b21) at (1,-2,3) ;
\coordinate (b321) at (-1,2,3) ;
\coordinate (b3b21) at (-1,-2,3) ;
\coordinate (231) at (2,1,3) ;
\coordinate (2b31) at (2,-1,3) ;
\coordinate (b231) at (-2,1,3) ;
\coordinate (b2b31) at (-2,-1,3) ;

\coordinate (31b2) at (1,2,-3) ;
\coordinate (3b2b1) at (1,-2,-3) ;
\coordinate (b32b1) at (-1,2,-3) ;
\coordinate (b3b2b1) at (-1,-2,-3) ;
\coordinate (23b1) at (2,1,-3) ;
\coordinate (2b3b1) at (2,-1,-3) ;
\coordinate (b23b1) at (-2,1,-3) ;
\coordinate (b2b3b1) at (-2,-1,-3) ;

\draw (2,-3,1)--(3,-2,1)--(3,-2,-1)--(2,-3,-1)--cycle;

\draw[fill] (3,1,2) circle (1.2pt); 	
\draw[fill] (3,-1,2) circle (1.2pt); 	
\draw[fill] (3,-2,1) circle (1.2pt); 	
\draw[fill] (3,-2,-1) circle (1.2pt); 
\draw[fill] (3,-1,-2) circle (1.2pt); 
\draw[fill] (3,1,-2) circle (1.2pt); 	
\draw[fill] (3,2,-1) circle (1.2pt); 	
\draw[fill] (3,2,1) circle (1.2pt); 	

\draw (3,1,2)--(3,-1,2)--(3,-2,1)--(3,-2,-1)--(3,-1,-2)--(3,1,-2)--(3,2,-1)--(3,2,1)--cycle;

\draw[fill] (-3,1,2) circle (1pt); 	
\draw[fill] (-3,-1,2) circle (1pt); 	
\draw[fill] (-3,-2,1) circle (1pt); 	
\draw[fill] (-3,-2,-1) circle (1pt); 	
\draw[fill] (-3,-1,-2) circle (1pt); 	
\draw[fill] (-3,1,-2) circle (1pt); 	
\draw[fill] (-3,2,-1) circle (1pt); 	
\draw[fill] (-3,2,1) circle (1pt); 	

\draw[dotted] (-3,1,2)--(-3,-1,2)--(-3,-2,1)--(-3,-2,-1)--(-3,-1,-2)--(-3,1,-2)--(-3,2,-1)--(-3,2,1)--cycle;

\draw[fill] (1,3,2) circle (1.2pt);	
\draw[fill] (-1,3,2) circle (1.2pt); 	
\draw[fill] (-2,3,1) circle (1.2pt); 	
\draw[fill] (-2,3,-1) circle (1.2pt); 
\draw[fill] (-1,3,-2) circle (1.2pt);	
\draw[fill] (1,3,-2) circle (1.2pt);	
\draw[fill] (2,3,-1) circle (1.2pt); 	
\draw[fill] (2,3,1) circle (1.2pt); 	

\draw (1,3,2)--(-1,3,2)--(-2,3,1)--(-2,3,-1)--(-1,3,-2)--(1,3,-2)--(2,3,-1)--(2,3,1)--cycle;

\draw[fill] (1,-3,2) circle (1pt); 	
\draw[fill] (-1,-3,2) circle (1pt); 	
\draw[fill] (-2,-3,1) circle (1pt); 	
\draw[fill] (-2,-3,-1) circle (1pt); 	
\draw[fill] (-1,-3,-2) circle (1pt); 	
\draw[fill] (1,-3,-2) circle (1pt); 	
\draw[fill] (2,-3,-1) circle (1pt); 	
\draw[fill] (2,-3,1) circle (1pt); 	

\draw[dotted] (1,-3,2)--(-1,-3,2)--(-2,-3,1)--(-2,-3,-1)--(-1,-3,-2)--(1,-3,-2)--(2,-3,-1)--(2,-3,1)--cycle;

\draw[fill] (1,2,3) circle (1.2pt);	
\draw[fill] (1,-2,3) circle (1.2pt);  
\draw[fill] (-1,2,3) circle (1.2pt); 	
\draw[fill] (-1,-2,3) circle (1.2pt); 
\draw[fill] (2,1,3) circle (1.2pt); 	
\draw[fill] (2,-1,3) circle (1.2pt); 	
\draw[fill] (-2,1,3) circle (1.2pt); 	
\draw[fill] (-2,-1,3) circle (1.2pt);	

\draw (1,2,3)--(-1,2,3)--(-2,1,3)--(-2,-1,3)--(-1,-2,3)--(1,-2,3)--(2,-1,3)--(2,1,3)--cycle;

\draw[fill] (1,2,-3) circle (1pt); 	
\draw[fill] (1,-2,-3) circle (1pt); 	
\draw[fill] (-1,2,-3) circle (1pt); 	
\draw[fill] (-1,-2,-3) circle (1pt); 	
\draw[fill] (2,1,-3) circle (1pt);  	
\draw[fill] (2,-1,-3) circle (1pt); 	
\draw[fill] (-2,1,-3) circle (1pt); 	
\draw[fill] (-2,-1,-3) circle (1pt);  	

\draw[dotted] (1,2,-3)--(-1,2,-3)--(-2,1,-3)--(-2,-1,-3)--(-1,-2,-3)--(1,-2,-3)--(2,-1,-3)--(2,1,-3)--cycle;

\draw (1,2,3)--(2,1,3)--(3,1,2)--(3,2,1)--(2,3,1)--(1,3,2)--cycle;

\draw (-1,2,3)--(-2,1,3)--(-3,1,2)--(-3,2,1)--(-2,3,1)--(-1,3,2)--cycle;

\draw[dotted] (-1,2,-3)--(-2,1,-3)--(-3,1,-2)--(-3,2,-1)--(-2,3,-1)--(-1,3,-2)--cycle;

\draw (1,-2,3)--(2,-1,3)--(3,-1,2)--(3,-2,1)--(2,-3,1)--(1,-3,2)--cycle;

\draw[dotted] (1,-2,-3)--(2,-1,-3)--(3,-1,-2)--(3,-2,-1)--(2,-3,-1)--(1,-3,-2)--cycle;

\draw[dotted] (-1,-2,-3)--(-2,-1,-3)--(-3,-1,-2)--(-3,-2,-1)--(-2,-3,-1)--(-1,-3,-2)--cycle;

\draw (1,2,-3)--(2,1,-3)--(3,1,-2)--(3,2,-1)--(2,3,-1)--(1,3,-2)--cycle;

\draw[dotted] (-1,-2,3)--(-2,-1,3)--(-3,-1,2)--(-3,-2,1)--(-2,-3,1)--(-1,-3,2)--cycle;

\draw (2,-3,1)--(2,-3,-1)--(3,-2,-1);
\draw (3,-1,-2)--(2,-1,-3)--(2,1,-3);
\draw (1,2,-3)--(-1,2,-3)--(-1,3,-2);

\draw[fill=yellow, opacity=0.5] ($1/2*(123)+1/2*(321)$)--($1/2*(123)+1/2*(213)$)--($1/2*(12b3)+1/2*(21b3)$)--($1/2*(23b1)+1/2*(31b2)$)--(23b1)--(2b3b1)--(1b3b2)--($1/2*(1b3b2)+1/2*(1b2b3)$)--($1/2*(123)+1/2*(132)$)--cycle;

\draw[thick, blue] ($1/2*(123)+1/2*(321)$)--($1/2*(123)+1/2*(213)$)--($1/2*(12b3)+1/2*(21b3)$)--($1/2*(23b1)+1/2*(31b2)$)--(23b1)--(2b3b1)--(1b3b2)--($1/2*(1b3b2)+1/2*(1b2b3)$)--($1/2*(123)+1/2*(132)$)--cycle;

\draw[thick, blue, dashed]  ($1/2*(123)+1/2*(321)$)--($1/2*(b1b2b3)+1/2*(b3b2b1)$)-- ($1/2*(b3b2b1)+1/2*(b2b3b1)$)-- ($1/2*(23b1)+1/2*(31b2)$);
\draw[thick, blue, dashed] ($1/2*(1b2b3)+1/2*(1b3b2)$)--($1/2*(3b1b2)+1/2*(3b2b1)$)--($1/2*(b3b1b2)+1/2*(b3b2b1)$)--($1/2*(b1b2b3)+1/2*(b3b2b1)$);
\draw[thick, blue, dashed] ($1/2*(3b1b2)+1/2*(3b2b1)$)--(3b2b1)--(2b3b1);
\draw[thick, blue, dashed] ($1/2*(b3b1b2)+1/2*(b3b2b1)$)--(b3b2b1)--($1/2*(b3b2b1)+1/2*(b2b3b1)$);
\draw[thick, blue, dashed] (3b2b1)--(b3b2b1);

\draw[thick, blue] (1b3b2)--(13b2)--(23b1);
\draw[thick, blue] (13b2)--(12b3)--(123)--($1/2*(123)+1/2*(132)$);
\draw[thick, blue] (12b3)--($1/2*(12b3)+1/2*(21b3)$);
\draw[thick, blue] (123)--($1/2*(123)+1/2*(213)$);
\node at (0,0,-5) {$P_{B_3}(\{1,2\})$};
\end{scope}

\begin{scope}[scale=0.4, rotate around x=-90, rotate around z=-105, xshift=600]

\coordinate (132) at (3,1,2) ;
\coordinate (1b32) at (3,-1,2) ;
\coordinate (1b23) at (3,-2,1) ;
\coordinate (1b2b3) at (3,-2,-1) ;
\coordinate (1b3b2) at (3,-1,-2) ;
\coordinate (13b2) at (3,1,-2) ;
\coordinate (12b3) at (3,2,-1) ;
\coordinate (123) at (3,2,1) ;

\coordinate (b132) at (-3,1,2) ;
\coordinate (b1b32) at (-3,-1,2) ;
\coordinate (b1b23) at (-3,-2,1) ;
\coordinate (b1b2b3) at (-3,-2,-1) ;
\coordinate (b1b1b3) at (-3,-1,-2) ;
\coordinate (b13b2) at (-3,1,-2) ;
\coordinate (b12b3) at (-3,2,-1) ;
\coordinate (b123) at (-3,2,1) ;

\coordinate (312) at (1,3,2) ;
\coordinate (b312) at (-1,3,2) ;
\coordinate (b213) at (-2,3,1) ;
\coordinate (b21b3) at (-2,3,-1) ;
\coordinate (b31b2) at (-1,3,-2) ;
\coordinate (31b2) at (1,3,-2) ;
\coordinate (21b3) at (2,3,-1) ;
\coordinate (213) at (2,3,1) ;

\coordinate (3b12) at (1,-3,2) ;
\coordinate (b3b12) at (-1,-3,2) ;
\coordinate (b2b13) at (-2,-3,1) ;
\coordinate (b2b1b3) at (-2,-3,-1) ;
\coordinate (b3b1b2) at (-1,-3,-2) ;
\coordinate (3b1b2) at (1,-3,-2) ;
\coordinate (2b1b3) at (2,-3,-1) ;
\coordinate (2b13) at (2,-3,1) ;

\coordinate (321) at (1,2,3) ;
\coordinate (3b21) at (1,-2,3) ;
\coordinate (b321) at (-1,2,3) ;
\coordinate (b3b21) at (-1,-2,3) ;
\coordinate (231) at (2,1,3) ;
\coordinate (2b31) at (2,-1,3) ;
\coordinate (b231) at (-2,1,3) ;
\coordinate (b2b31) at (-2,-1,3) ;

\coordinate (31b2) at (1,2,-3) ;
\coordinate (3b2b1) at (1,-2,-3) ;
\coordinate (b32b1) at (-1,2,-3) ;
\coordinate (b3b2b1) at (-1,-2,-3) ;
\coordinate (23b1) at (2,1,-3) ;
\coordinate (2b3b1) at (2,-1,-3) ;
\coordinate (b23b1) at (-2,1,-3) ;
\coordinate (b2b3b1) at (-2,-1,-3) ;

\draw (2,-3,1)--(3,-2,1)--(3,-2,-1)--(2,-3,-1)--cycle;

\draw[fill] (3,1,2) circle (1.2pt); 	
\draw[fill] (3,-1,2) circle (1.2pt); 	
\draw[fill] (3,-2,1) circle (1.2pt); 	
\draw[fill] (3,-2,-1) circle (1.2pt); 
\draw[fill] (3,-1,-2) circle (1.2pt); 
\draw[fill] (3,1,-2) circle (1.2pt); 	
\draw[fill] (3,2,-1) circle (1.2pt); 	
\draw[fill] (3,2,1) circle (1.2pt); 	

\draw (3,1,2)--(3,-1,2)--(3,-2,1)--(3,-2,-1)--(3,-1,-2)--(3,1,-2)--(3,2,-1)--(3,2,1)--cycle;

\draw[fill] (-3,1,2) circle (1pt); 	
\draw[fill] (-3,-1,2) circle (1pt); 	
\draw[fill] (-3,-2,1) circle (1pt); 	
\draw[fill] (-3,-2,-1) circle (1pt); 	
\draw[fill] (-3,-1,-2) circle (1pt); 	
\draw[fill] (-3,1,-2) circle (1pt); 	
\draw[fill] (-3,2,-1) circle (1pt); 	
\draw[fill] (-3,2,1) circle (1pt); 	

\draw[dotted] (-3,1,2)--(-3,-1,2)--(-3,-2,1)--(-3,-2,-1)--(-3,-1,-2)--(-3,1,-2)--(-3,2,-1)--(-3,2,1)--cycle;

\draw[fill] (1,3,2) circle (1.2pt);	
\draw[fill] (-1,3,2) circle (1.2pt); 	
\draw[fill] (-2,3,1) circle (1.2pt); 	
\draw[fill] (-2,3,-1) circle (1.2pt); 
\draw[fill] (-1,3,-2) circle (1.2pt);	
\draw[fill] (1,3,-2) circle (1.2pt);	
\draw[fill] (2,3,-1) circle (1.2pt); 	
\draw[fill] (2,3,1) circle (1.2pt); 	

\draw (1,3,2)--(-1,3,2)--(-2,3,1)--(-2,3,-1)--(-1,3,-2)--(1,3,-2)--(2,3,-1)--(2,3,1)--cycle;

\draw[fill] (1,-3,2) circle (1pt); 	
\draw[fill] (-1,-3,2) circle (1pt); 	
\draw[fill] (-2,-3,1) circle (1pt); 	
\draw[fill] (-2,-3,-1) circle (1pt); 	
\draw[fill] (-1,-3,-2) circle (1pt); 	
\draw[fill] (1,-3,-2) circle (1pt); 	
\draw[fill] (2,-3,-1) circle (1pt); 	
\draw[fill] (2,-3,1) circle (1pt); 	

\draw[dotted] (1,-3,2)--(-1,-3,2)--(-2,-3,1)--(-2,-3,-1)--(-1,-3,-2)--(1,-3,-2)--(2,-3,-1)--(2,-3,1)--cycle;

\draw[fill] (1,2,3) circle (1.2pt);	
\draw[fill] (1,-2,3) circle (1.2pt);  
\draw[fill] (-1,2,3) circle (1.2pt); 	
\draw[fill] (-1,-2,3) circle (1.2pt); 
\draw[fill] (2,1,3) circle (1.2pt); 	
\draw[fill] (2,-1,3) circle (1.2pt); 	
\draw[fill] (-2,1,3) circle (1.2pt); 	
\draw[fill] (-2,-1,3) circle (1.2pt);	

\draw (1,2,3)--(-1,2,3)--(-2,1,3)--(-2,-1,3)--(-1,-2,3)--(1,-2,3)--(2,-1,3)--(2,1,3)--cycle;

\draw[fill] (1,2,-3) circle (1pt); 	
\draw[fill] (1,-2,-3) circle (1pt); 	
\draw[fill] (-1,2,-3) circle (1pt); 	
\draw[fill] (-1,-2,-3) circle (1pt); 	
\draw[fill] (2,1,-3) circle (1pt);  	
\draw[fill] (2,-1,-3) circle (1pt); 	
\draw[fill] (-2,1,-3) circle (1pt); 	
\draw[fill] (-2,-1,-3) circle (1pt);  	

\draw[dotted] (1,2,-3)--(-1,2,-3)--(-2,1,-3)--(-2,-1,-3)--(-1,-2,-3)--(1,-2,-3)--(2,-1,-3)--(2,1,-3)--cycle;

\draw (1,2,3)--(2,1,3)--(3,1,2)--(3,2,1)--(2,3,1)--(1,3,2)--cycle;

\draw (-1,2,3)--(-2,1,3)--(-3,1,2)--(-3,2,1)--(-2,3,1)--(-1,3,2)--cycle;

\draw[dotted] (-1,2,-3)--(-2,1,-3)--(-3,1,-2)--(-3,2,-1)--(-2,3,-1)--(-1,3,-2)--cycle;

\draw (1,-2,3)--(2,-1,3)--(3,-1,2)--(3,-2,1)--(2,-3,1)--(1,-3,2)--cycle;

\draw[dotted] (1,-2,-3)--(2,-1,-3)--(3,-1,-2)--(3,-2,-1)--(2,-3,-1)--(1,-3,-2)--cycle;

\draw[dotted] (-1,-2,-3)--(-2,-1,-3)--(-3,-1,-2)--(-3,-2,-1)--(-2,-3,-1)--(-1,-3,-2)--cycle;

\draw (1,2,-3)--(2,1,-3)--(3,1,-2)--(3,2,-1)--(2,3,-1)--(1,3,-2)--cycle;

\draw[dotted] (-1,-2,3)--(-2,-1,3)--(-3,-1,2)--(-3,-2,1)--(-2,-3,1)--(-1,-3,2)--cycle;

\draw (2,-3,1)--(2,-3,-1)--(3,-2,-1);
\draw (3,-1,-2)--(2,-1,-3)--(2,1,-3);
\draw (1,2,-3)--(-1,2,-3)--(-1,3,-2);

\draw[fill=yellow, opacity=0.5] ($1/2*(123)+1/2*(321)$)--($1/2*(123)+1/2*(213)$)--($1/2*(123)+1/2*(21b3)$)--($1/2*(123)+1/2*(12b3)$)--($1/2*(123)+1/2*(1b2b3)$)--($1/2*(123)+1/2*(132)$)--cycle;
\draw[thick, blue] ($1/2*(123)+1/2*(321)$)--($1/2*(123)+1/2*(213)$)--($1/2*(123)+1/2*(21b3)$)--($1/2*(123)+1/2*(12b3)$)--($1/2*(123)+1/2*(1b2b3)$)--($1/2*(123)+1/2*(132)$)--cycle;

\draw[thick, dashed, blue] (0,0,0)--($1/2*(123)+1/2*(321)$);
\draw[thick, dashed, blue] (0,0,0)--($1/2*(123)+1/2*(1b2b3)$);
\draw[thick, dashed, blue] (0,0,0)--($1/2*(123)+1/2*(21b3)$);

\draw[thick, blue] (123)--($1/2*(123)+1/2*(132)$);
\draw[thick, blue] (123)--($1/2*(123)+1/2*(213)$);
\draw[thick, blue] (123)--($1/2*(123)+1/2*(12b3)$);

\node at (0,0,-5) {$P_{B_3}(\{1,2,3\})$};

\end{scope}

\end{tikzpicture}
\caption{Examples of partitioned weight polytopes in type $B_3$.}
\label{fig_part_weight_poly_B}
\end{figure}

The next  proposition is the type $B_n$ analogue of Proposition \ref{prop_A_position_F(I)}, and can be proved similarly using the definition of $F(I)$ given in \eqref{eq_B_facet_convex_hull} together with the property that $x_i+x_{\bar i}=0$ in $E$ for $i=1, \dots, n$.

\begin{proposition}\label{prop_B_position_F(I)}
For each $k\in [n]$ and $I\in \M_{B_n}$, the following claims hold.
\begin{enumerate}
\item The facet $F(\indI)$ intersects nontrivially with $\Hk$ if and only if $\indI$ is $s_k$-invariant. In this case, $F(\indI)$ is bisected by the hyperplane $\Hk$. 
\item The facet $F(\indI)$ is contained in the interior part of $\Hk^\le$ if and only if  
\begin{enumerate}
\item $k\in I$ but  $k+1\notin I$ or  $\overline{k+1}\in I$ but $\bar{k} \notin I$, when $1\leq k < n$; 
\item $n\in I$ but $\bar n \notin I$, when $k=n$.
\end{enumerate}
\end{enumerate}
\end{proposition}

Motivated by this proposition, we make the following definition.  

\begin{definition}\label{def_B_K-lower}
For a subset $K\subset [n]$, we define $\M_{B_n}(K)$ to be the set of elements $I\in \M_{B_n}$ such that $I\cap N$ is a lower subset of $N$ with respect to the order 
\[
1<2<\cdots<n<\bar{n}<\cdots<\bar{2}<\bar{1}
\]
for every disjoint component $N$ appearing in the $W_K$-orbit decomposition \eqref{eq:B_orbit_decomposition}.
\end{definition}

\begin{example}
Continuing with Example~\ref{ex_B_3_orbit_decomp}-(3), $\M_{B_3}(\{2,3\})$ consists of 
\[
\{1\}, \{\bar 1\}, \{2\}, \{1,2\}, \{2, \bar1\}, \{2,3\}, \{1,2,3\}, \{2,3, \bar 1\}. 
\]
\end{example}

With the definition of $\M_{B_n}(K)$ above, Corollary~\ref{cor_A_facet_characterization_of_Part_perm}, Propositions~\ref{prop_A_flagness_part_w_poly} and~\ref{prop_A_intersection_facets_PK} hold for type $B_n$.  
Moreover, Proposition~\ref{prop_A_cohomologyXK} also holds for type $B_n$ with 
\[
\alpha^{\v}_k =\begin{cases}
e_k-e_{k+1} & 1\leq k < n,\\
{2}e_n & k=n.
\end{cases}
\]

\subsection{Proof of Theorem \ref{main} for type $B_n$}\label{subsec_B_proof_of_main_thm}
 
In analogy with what was done in type $A_{n-1}$, one can write
\begin{align*}
e_\indI - e_{v(\indI)}=\sum_{k \in K} c^{\indI,v}_k \alpha^{\v}_k 
\end{align*}
for some non-negative integers $c^{\indI,v}_k$, whenever $\indI \in \M_{B_n}(K)$ and $v \in W_K$. In order to state a sufficient condition for $c_k^{\indI,v}$ to be zero, we introduce the following notation: 
\begin{equation}\label{notation_khat}
\widehat{[k]} \colonequals 
\left\{\bar{n}, \overline{n-1},\ldots, \bar{k} \right\}, \quad \text{ for }1 \leq k \leq n.
\end{equation}
The following lemma is the type $B_n$ analogue of Lemma~\ref{lem_A_coeff_nonzero} and admits a similar proof. 

\begin{lemma} \label{lemm:B_coeff_zero}
Let $\indI \in \M_{B_n}(K)$, $v \in W_K$ and consider the $W_K$-orbit decomposition \eqref{eq:B_orbit_decomposition}.  Then $c^{\indI,v}_k=0$ $(k\in K)$ in the following cases:
\begin{enumerate}
\item When $k\in N_i$ ($k\not=n$ by definition of $N_i$ since $k\in K$),   
\[
\begin{split}
&v(\indI) \cap N_i \subset [k] \cap N_i \ {\rm or} \ v(\indI) \cap N_i \supset [k] \cap N_i, \text{ and}\\ 
&v(\indI) \cap \overline{N_i} \subset \widehat{[k+1]}  \cap \overline{N_i} \ {\rm or} \ v(\indI) \cap \overline{N_i} \supset \widehat{[k+1]} \cap \overline{N_i} .
\end{split}
\]
\item When $k\in N'$, 
\[
v(\indI) \cap N' \subset [k] \cap N' \ {\rm or} \ v(\indI) \cap N' \supset [k] \cap N'.
\]
\end{enumerate}
\end{lemma}

With these modifications, the argument developed for type $A_{n-1}$ after Lemma~\ref{lem_A_coeff_nonzero} works for type $B_n$.

\section{Type $D$}\label{sec_typeD}

The proof of Theorem \ref{main} for type $A_{n-1}$ also works for type $D_n$ with suitable modifications.  In this section, we point out necessary modifications.  The situation is similar to that in type $B_n$ but more complicated.  We maintain the notations $\nbar$ and $\bar i$ etc. from Appendix~\ref{sec_typeB}. 

\subsection{Weight polytope of type $D_n$}\label{subsec_D_weight_poly}

In the same vector space $E$ defined in \eqref{eq_B_vector_sp} with the same basis $\{t_1, \dots, t_n\}\subset E$, we take the simple roots 
\begin{equation*}
\alpha_i = \begin{cases}
t_i-t_{i+1} &\text{for }1 \leq i \leq n-1;\\
t_{n-1}+t_n &\text{for } i=n.
\end{cases}
\end{equation*}
to define the type $D_n$ root system $\Phi_{D_n}$. 
The Weyl group $W_{D_n}$ is the group of \emph{even-signed permutations}, namely
\[
W_{D_n}=\{u \in W_{B_n} \mid |\{u(1),\ldots, u(n)\} \cap \overline{[n]}|  \text{ is even}\}.
\]
This group is generated by $s_1, \dots, s_n$ where 
\begin{align*}
s_i&=(i, i+1)(\overline{i},\overline{i+1}) \quad \text{ for }1 \leq i \leq n-1;\\
s_n&=(n-1, \overline{n})(\overline{n-1}, n).
\end{align*}

As with type $B_n$, we take a point $(a_1,\dots,a_n, a_{\bar n}, \dots, a_{\bar 1})\in E$ with $a_1<\cdots <a_n<0$ and define the weight polytope of type $D_n$ by   
\begin{equation*}
P_{D_n}\colonequals \text{\rm conv}\{ (a_{u(1)},\dots,a_{u(n)},a_{u(\bar n)},\dots,a_{u(\bar 1)})\in E \mid u\in W_{D_n} \}. 
\end{equation*}
This polytope is $n$-dimensional and sits in the $n$-dimensional vector space $E$. See Figure \ref{fig_weight_Poly_dim2_3TypeD} for the description of $P_{D_2}$ and $P_{D_3}$, where vertices are labeled by the first half of the one-line notation as in Figure \ref{fig_weight_Poly_dim2_3TypeB}. 

\begin{figure}
\begin{tikzpicture}


\begin{scope}[scale=0.6, yshift=-260]
\draw[thick, fill=yellow!50] (-1,-2)--(2,1)--(1,2)--(-2,-1)--cycle;

\node[below] at (-1,-2) {\tiny$21$};
\node[left] at (-2,-1) {\tiny$12$};
\node[above] at (1,2) {\tiny$\bar2\bar1$};
\node[right] at (2,1) {\tiny$\bar1\bar2$};
\node at (0,-3.7) {$P_{D_2}\subset \RR^4$};

\end{scope}

\begin{scope}[scale=0.5, rotate around x=-90, rotate around z=-105, xshift=350, yshift=-300]
\node at (0,0,-5) {$P_{D_3}\subset \RR^6$};

\coordinate (132) at (3,1,2);
\coordinate (13b2) at (3,-1,2);
\coordinate (12b3) at (3,-2,1);
\coordinate (12b3b) at (3,-2,-1);
\coordinate (13b2b) at (3,-1,-2);
\coordinate (132b) at (3,1,-2);
\coordinate (12b3) at (3,2,-1);
\coordinate (123) at (3,2,1);

\draw[fill] (3,1,2) circle (1.2pt); 	\node[right] at (3,1,2) {\tiny$132$};
\draw[fill] (3,-2,-1) circle (1.2pt); \node[left] at (3,-2,-1) {\tiny$1 \bar2 \bar3$};
\draw[fill] (3,-1,-2) circle (1.2pt); \node[left] at (3,-1,-2) {\tiny$1 \bar3 \bar2$};
\draw[fill] (3,2,1) circle (1.2pt); 	\node[left] at (3,2,1) {\tiny$123$};

\draw[thick] (132)--(123)--(13b2b)--(12b3b)--cycle;

\coordinate (1b32) at (-3,1,2);
\coordinate (1b3b2) at (-3,-1,2);
\coordinate (1b2b3) at (-3,-2,1);
\coordinate (1b2b3b) at (-3,-2,-1);
\coordinate (1b3b2b) at (-3,-1,-2);
\coordinate (1b32b) at (-3,1,-2);
\coordinate (1b23b) at (-3,2,-1);
\coordinate (1b23) at (-3,2,1);

\draw[fill] (-3,-1,2) circle (1pt); 	\node[right] at (-3,-1,2) {\tiny$\bar1 \bar3 2$};
\draw[fill] (-3,-2,1) circle (1pt); 	\node[above, gray] at (-3,-2,1) {\tiny$\bar1 \bar2 3$};
\draw[fill] (-3,1,-2) circle (1pt); 	\node[above, gray] at (-3,1,-2) {\tiny$\bar1 3 \bar2$};
\draw[fill] (-3,2,-1) circle (1pt); 	\node[right] at (-3,2,-1) {\tiny$\bar1 2 \bar3$};

\draw[thick] (1b3b2)--(1b23b);
\draw[dotted] (1b3b2)--(1b2b3)--(1b32b)--(1b23b);

\coordinate (312) at (1,3,2);
\coordinate (3b12) at (-1,3,2);
\coordinate (2b13) at (-2,3,1);
\coordinate (2b13b) at (-2,3,-1);
\coordinate (3b12b) at (-1,3,-2);
\coordinate (312b) at (1,3,-2);
\coordinate (213b) at (2,3,-1);
\coordinate (213) at (2,3,1);

\draw[fill] (1,3,2) circle (1.2pt);	\node[left] at (1,3,2) {\tiny$312$};
\draw[fill] (-2,3,-1) circle (1.2pt); \node[right] at (-2,3,-1) {\tiny$\bar2 1 \bar3 $};
\draw[fill] (-1,3,-2) circle (1.2pt);	\node[right] at (-1,3,-2) {\tiny$\bar3 1 \bar2$};
\draw[fill] (2,3,1) circle (1.2pt); 	\node[right] at (2,3,1) {\tiny$213$};

\draw[thick] (213)--(312)--(2b13b)--(3b12b)--cycle;

\coordinate (31b2) at (1,-3,2);
\coordinate (3b1b2) at (-1,-3,2);
\coordinate (2b1b3) at (-2,-3,1);
\coordinate (2b1b3b) at (-2,-3,-1);
\coordinate (3b1b2b) at (-1,-3,-2);
\coordinate (31b2b) at (1,-3,-2);
\coordinate (21b3b) at (2,-3,-1);
\coordinate (21b3) at (2,-3,1);

\draw[fill] (-1,-3,2) circle (1pt); 	\node[left] at (-1,-3,2) {\tiny$\bar3 \bar1 2$};
\draw[fill] (-2,-3,1) circle (1pt); 	\node[below, gray] at (-2,-3,1) {\tiny$\bar2 \bar1 3$};
\draw[fill] (1,-3,-2) circle (1pt); 	\node[right, gray] at (1,-3,-2) {\tiny$3\bar1 \bar2$};
\draw[fill] (2,-3,-1) circle (1pt); 	\node[left] at (2,-3,-1) {\tiny$2 \bar1 \bar3$};

\draw[thick](21b3b)--(3b1b2);
\draw[dotted] (21b3b)--(31b2b)--(2b1b3)--(3b1b2);

\coordinate (321) at (1,2,3);
\coordinate (32b1) at (-1,-2,3);
\coordinate (3b21) at (-1,2,3);
\coordinate (3b2b1) at (-1,-2,3);
\coordinate (231) at (2,1,3);
\coordinate (23b1) at (2,-1,3);
\coordinate (2b31) at (-2,1,3);
\coordinate (2b3b1) at (-2,-1,3);

\draw[fill] (1,2,3) circle (1.2pt);	\node[below] at (1.2,2,3) {\tiny$321$};
\draw[fill] (-1,-2,3) circle (1.2pt); \node[above] at (-1,-2,3) {\tiny$\bar3 \bar2 1$};
\draw[fill] (2,1,3) circle (1.2pt); 	\node[below] at (2-0.4,1,3-0.2) {\tiny$231$};
\draw[fill] (-2,-1,3) circle (1.2pt);	\node[above] at (-2,-1,3) {\tiny$\bar2 \bar3 1$};

\draw[thick] (231)--(321)--(2b3b1)--(3b2b1)--cycle;

\coordinate (321b) at (1,2,-3);
\coordinate (32b1b) at (1,-2,-3);
\coordinate (3b21b) at (-1,2,-3);
\coordinate (3b2b1b) at (-1,-2,-3);
\coordinate (231b) at (2,1,-3);
\coordinate (23b1b) at (2,-1,-3);
\coordinate (2b31b) at (-2,1,-3);
\coordinate (2b3b1b) at (-2,-1,-3);

\draw[fill] (1,-2,-3) circle (1pt); 	\node[above right,gray] at (1,-2,-3) {\tiny$3\bar2\bar1$};
\draw[fill] (-1,2,-3) circle (1pt); 	\node[below] at (-1,2,-3) {\tiny$\bar3 2\bar1$};
\draw[fill] (2,-1,-3) circle (1pt); 	\node[below] at (2,-1,-3) {\tiny$2 \bar3 \bar1$};
\draw[fill] (-2,1,-3) circle (1pt); 	\node[above, gray] at (-2,1,-3) {\tiny$\bar2 3 \bar1$};

\draw[thick] (23b1b)--(3b21b);
\draw[dotted] (23b1b)--(32b1b)--(2b31b)--(3b21b);

\draw[thick] (1,2,3)--(2,1,3)--(3,1,2)--(3,2,1)--(2,3,1)--(1,3,2)--cycle;







\draw[dotted] (-1,-2,3)--(-2,-1,3)--(-3,-1,2)--(-3,-2,1)--(-2,-3,1)--(-1,-3,2)--cycle;


\draw[thick] (21b3b)--(12b3b)--(132)--(231)--(3b2b1)--(3b1b2)--cycle;
\draw[thick] (321)--(312)--(2b13b)--(1b23b)--(1b3b2)--(2b3b1)--cycle;

\draw[thick] (3b21b)--(3b12b);
\draw[dotted] (2b31b)--(1b32b);
\draw[thick] (13b2b)--(23b1b);
\draw[dotted] (31b2b)--(32b1b);
\end{scope}

\end{tikzpicture}
\caption{Weight polytopes of $P_{D_n}$ dimension $2$ and $3$.}
\label{fig_weight_Poly_dim2_3TypeD}
\end{figure}
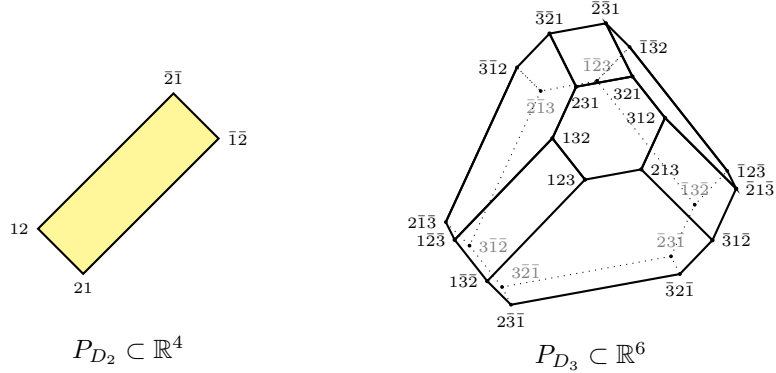

\begin{definition} \label{def:MDn} 
$\M_{D_n}\colonequals \{I \in \M_{B_n} \mid |I|\neq n-1\}$.
\end{definition}

Every facet of $P_{D_n}$ is of the form
\[
F(\indI)=\left\{(x_1, \dots, x_{n}, x_{\bar n},\dots,x_{\bar 1}) \in P_{D_n} \mid \sum_{i\in I}x_i =\sum_{i\in \mathcal{I}}a_i\right\}
\]
for some $I\in \M_{D_n}$, where 
\[
\mathcal{I}=\begin{cases} 
[n-1]\sqcup\{\bar{n}\} \quad &\text{if $|I|=n$ and $|I\cap \overline{[n]}|$ is odd},\\
[|I|] \quad &\text{otherwise}.
\end{cases}
\]
Every such $F(\indI)$ is a facet.  An (inward) normal vector to $F(I)$ in $E$, regarded as an element of the dual space $E^*$, is given by $e_I=\sum_{i\in I}e_i$ as in the type $B_n$ case.

\begin{figure}
\begin{tikzpicture}[scale=0.8]

\draw[dotted] (-1,-2)--(1,-2)--(2,-1)--(2,1)--(1,2)--(-1,2)--(-2,1)--(-2,-1)--cycle;
\draw (-1,-2)--(-2,-1)--(1,2)--(2,1)--cycle;
\node at (-1.5,-1.5) {\tiny$F(12)$};
\node at (-0.5,0.5) {\tiny$F(1\bar2)$};
\node at (1.5,1.5) {\tiny$F(\bar1\bar2)$};
\node at (0.5,-0.5) {\tiny$F(\bar12)$};


\begin{scope}[scale=0.6, rotate around x=-90, rotate around z=-105, xshift=400]
\coordinate (132) at (3,1,2);
\coordinate (13b2) at (3,-1,2);
\coordinate (12b3) at (3,-2,1);
\coordinate (12b3b) at (3,-2,-1);
\coordinate (13b2b) at (3,-1,-2);
\coordinate (132b) at (3,1,-2);
\coordinate (12b3) at (3,2,-1);
\coordinate (123) at (3,2,1);

\coordinate (1b32) at (-3,1,2);
\coordinate (1b3b2) at (-3,-1,2);
\coordinate (1b2b3) at (-3,-2,1);
\coordinate (1b2b3b) at (-3,-2,-1);
\coordinate (1b3b2b) at (-3,-1,-2);
\coordinate (1b32b) at (-3,1,-2);
\coordinate (1b23b) at (-3,2,-1);
\coordinate (1b23) at (-3,2,1);

\coordinate (312) at (1,3,2);
\coordinate (3b12) at (-1,3,2);
\coordinate (2b13) at (-2,3,1);
\coordinate (2b13b) at (-2,3,-1);
\coordinate (3b12b) at (-1,3,-2);
\coordinate (312b) at (1,3,-2);
\coordinate (213b) at (2,3,-1);
\coordinate (213) at (2,3,1);

\coordinate (31b2) at (1,-3,2);
\coordinate (3b1b2) at (-1,-3,2);
\coordinate (2b1b3) at (-2,-3,1);
\coordinate (2b1b3b) at (-2,-3,-1);
\coordinate (3b1b2b) at (-1,-3,-2);
\coordinate (31b2b) at (1,-3,-2);
\coordinate (21b3b) at (2,-3,-1);
\coordinate (21b3) at (2,-3,1);

\coordinate (321) at (1,2,3);
\coordinate (32b1) at (-1,-2,3);
\coordinate (3b21) at (-1,2,3);
\coordinate (3b2b1) at (-1,-2,3);
\coordinate (231) at (2,1,3);
\coordinate (23b1) at (2,-1,3);
\coordinate (2b31) at (-2,1,3);
\coordinate (2b3b1) at (-2,-1,3);

\coordinate (321b) at (1,2,-3);
\coordinate (32b1b) at (1,-2,-3);
\coordinate (3b21b) at (-1,2,-3);
\coordinate (3b2b1b) at (-1,-2,-3);
\coordinate (231b) at (2,1,-3);
\coordinate (23b1b) at (2,-1,-3);
\coordinate (2b31b) at (-2,1,-3);
\coordinate (2b3b1b) at (-2,-1,-3);

\draw[fill] (1,2,-3) circle (1pt); 	
\draw[fill] (1,-2,-3) circle (1pt); 	
\draw[fill] (-1,2,-3) circle (1pt);	
\draw[fill] (-1,-2,-3) circle (1pt);	
\draw[fill] (2,1,-3) circle (1pt);  	
\draw[fill] (2,-1,-3) circle (1pt); 	
\draw[fill] (-2,1,-3) circle (1pt); 	
\draw[fill] (-2,-1,-3) circle (1pt);  

\draw[dotted] (1,2,-3)--(-1,2,-3)--(-2,1,-3)--(-2,-1,-3)--(-1,-2,-3)--(1,-2,-3)--(2,-1,-3)--(2,1,-3)--cycle;

\draw[fill] (1,2,3) circle (1.2pt);	
\draw[fill] (1,-2,3) circle (1.2pt); 	
\draw[fill] (-1,2,3) circle (1.2pt); 	
\draw[fill] (-1,-2,3) circle (1.2pt); 
\draw[fill] (2,1,3) circle (1.2pt); 	
\draw[fill] (2,-1,3) circle (1.2pt); 	
\draw[fill] (-2,1,3) circle (1.2pt); 	
\draw[fill] (-2,-1,3) circle (1.2pt);	

\draw[dotted] (1,2,3)--(-1,2,3)--(-2,1,3)--(-2,-1,3)--(-1,-2,3)--(1,-2,3)--(2,-1,3)--(2,1,3)--cycle;
\node at (0,0,3) {\tiny$F(3)$};

\draw[fill] (-3,1,2) circle (1pt); 	
\draw[fill] (-3,-1,2) circle (1pt); 	
\draw[fill] (-3,-2,1) circle (1pt); 	
\draw[fill] (-3,-2,-1) circle (1pt); 	
\draw[fill] (-3,-1,-2) circle (1pt); 	
\draw[fill] (-3,1,-2) circle (1pt); 	
\draw[fill] (-3,2,-1) circle (1pt); 	
\draw[fill] (-3,2,1) circle (1pt); 	

\draw[dotted] (-3,1,2)--(-3,-1,2)--(-3,-2,1)--(-3,-2,-1)--(-3,-1,-2)--(-3,1,-2)--(-3,2,-1)--(-3,2,1)--cycle;

\draw[fill] (3,1,2) circle (1.2pt); 	
\draw[fill] (3,-1,2) circle (1.2pt); 	
\draw[fill] (3,-2,1) circle (1.2pt); 	
\draw[fill] (3,-2,-1) circle (1.2pt); 
\draw[fill] (3,-1,-2) circle (1.2pt); 
\draw[fill] (3,1,-2) circle (1.2pt); 	
\draw[fill] (3,2,-1) circle (1.2pt); 	
\draw[fill] (3,2,1) circle (1.2pt); 	

\draw[dotted] (3,1,2)--(3,-1,2)--(3,-2,1)--(3,-2,-1)--(3,-1,-2)--(3,1,-2)--(3,2,-1)--(3,2,1)--cycle;
\node at (3,0,0) {\tiny$F(1)$};

\draw[fill] (1,-3,2) circle (1pt); 	
\draw[fill] (-1,-3,2) circle (1pt); 	
\draw[fill] (-2,-3,1) circle (1pt); 	
\draw[fill] (-2,-3,-1) circle (1pt); 	
\draw[fill] (-1,-3,-2) circle (1pt); 	
\draw[fill] (1,-3,-2) circle (1pt); 	
\draw[fill] (2,-3,-1) circle (1pt); 	
\draw[fill] (2,-3,1) circle (1pt); 	

\draw[dotted] (1,-3,2)--(-1,-3,2)--(-2,-3,1)--(-2,-3,-1)--(-1,-3,-2)--(1,-3,-2)--(2,-3,-1)--(2,-3,1)--cycle;

\draw[fill] (1,3,2) circle (1.2pt);	
\draw[fill] (-1,3,2) circle (1.2pt); 	
\draw[fill] (-2,3,1) circle (1.2pt); 	
\draw[fill] (-2,3,-1) circle (1.2pt); 
\draw[fill] (-1,3,-2) circle (1.2pt);	
\draw[fill] (1,3,-2) circle (1.2pt);	
\draw[fill] (2,3,-1) circle (1.2pt); 	
\draw[fill] (2,3,1) circle (1.2pt); 	

\draw[dotted] (1,3,2)--(-1,3,2)--(-2,3,1)--(-2,3,-1)--(-1,3,-2)--(1,3,-2)--(2,3,-1)--(2,3,1)--cycle;
\node at (0,3,0) {\tiny$F(2)$};


\draw[thick] (1,2,3)--(2,1,3)--(3,1,2)--(3,2,1)--(2,3,1)--(1,3,2)--cycle;
\node at (2,2,2) {\tiny$F(123)$};
\draw[dotted] (-1,2,3)--(-2,1,3)--(-3,1,2)--(-3,2,1)--(-2,3,1)--(-1,3,2)--cycle;
\draw[dashed, thick] (-1,2,-3)--(-2,1,-3)--(-3,1,-2)--(-3,2,-1)--(-2,3,-1)--(-1,3,-2)--cycle;
\draw[dotted] (1,-2,3)--(2,-1,3)--(3,-1,2)--(3,-2,1)--(2,-3,1)--(1,-3,2)--cycle;
\draw[dashed, thick] (1,-2,-3)--(2,-1,-3)--(3,-1,-2)--(3,-2,-1)--(2,-3,-1)--(1,-3,-2)--cycle;
\draw[dotted] (-1,-2,-3)--(-2,-1,-3)--(-3,-1,-2)--(-3,-2,-1)--(-2,-3,-1)--(-1,-3,-2)--cycle;
\draw[dotted] (1,2,-3)--(2,1,-3)--(3,1,-2)--(3,2,-1)--(2,3,-1)--(1,3,-2)--cycle;
\draw[dashed, thick] (-1,-2,3)--(-2,-1,3)--(-3,-1,2)--(-3,-2,1)--(-2,-3,1)--(-1,-3,2)--cycle;

\draw[dotted] (2,-3,1)--(2,-3,-1)--(3,-2,-1);
\draw[dotted] (3,-1,-2)--(2,-1,-3)--(2,1,-3);
\draw[dotted] (1,2,-3)--(-1,2,-3)--(-1,3,-2);

\draw[thick] (231)--(321)--(2b3b1)--(3b2b1)--cycle;
\draw[thick] (132)--(123)--(13b2b)--(12b3b)--cycle;
\draw[thick] (231)--(3b2b1)--(3b1b2)--(21b3b)--(12b3b)--(132)--cycle;
\node at (2.2,-1,2 ){\tiny$F(1\bar23)$};
\draw[thick] (123)--(213)--(3b12b)--(3b21b)--(23b1b)--(13b2b)--cycle;
\node at (8/6,8/6,-8/6) {\tiny$F(12\bar3)$};

\draw[thick] (213)--(312)--(2b13b)--(3b12b)--cycle;
\draw[thick] (321)--(312)--(2b13b)--(1b23b)--(1b3b2)--(2b3b1)--cycle;
\draw[thick, dashed] (2b1b3)--(31b2b); 
\draw[thick, dashed] (32b1b)--(2b31b);
\draw[thick, dashed] (1b2b3)--(1b32b);
\end{scope}

\end{tikzpicture}
\caption{Labeling on facets of $P_{D_n}$ for $n=2,3$.}
\label{fig_labeling_on_PDn}
\end{figure}
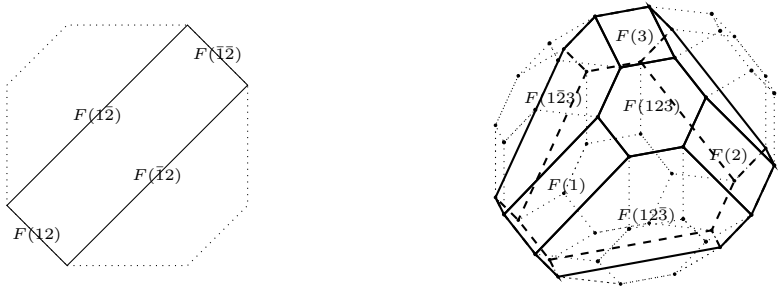

The following lemma corresponds to Lemma~\ref{lem_A_facets_intersecting}, and gives a necessary and sufficient condition for two facets to intersect nontrivially. 

\begin{lemma}\label{lem_D_facets_intersecting}
Let $F(I), F(J)$ be two facets of $P_{D_n}$ corresponding to $I,J\in \M_{D_n}$, respectively. 
\begin{enumerate}
\item Suppose that $|\indI| \neq n$ or $|\indJ| \neq n$. Then, $\FA$ and $\FB$ intersect nontrivially if and only if $\indI\subset \indJ$ or $\indJ\subset \indI$. 
\item Suppose that $|\indI|=|\indJ|=n$. Then, $\FA$ and $\FB$ intersect nontrivially if and only if $\indI=\indJ$ or $\indI=(\indJ \setminus \{j \}) \cup \{\bar{j} \}$ for some $j \in \indJ$. 
\end{enumerate}
Equivalently, $\FA$ and $\FB$ do not intersect if and only if $\indI\not\subset \indJ$, $\indJ\not\subset \indI$ and $|\indI \cap \indJ| \leq n-2$, the last condition following from Definition~\ref{def:MDn}. 
\end{lemma}

With this understood, Proposition~\ref{prop:A_cohom_of_perm}  holds for type $D_n$ if we add the condition $|I\cap J|\le n-2$ to the condition $I\not\subset J$ and $J\not\subset I$ in (2) of Proposition~\ref{prop:A_cohom_of_perm}. The modification arises from Lemma~\ref{lem_D_facets_intersecting}.

Let $K$ be a subset of $[n]$ and $W_K$ the subgroup of $W_{D_n}$ generated by $\{s_k:k\in K\}$.  Then, the $W_K$-orbit decomposition of $[n]\sqcup \overline{[n]}$ is of the following form:
\begin{equation} \label{eq:D_orbit_decomposition}
[n]\sqcup \overline{[n]}=\bigsqcup_{i=1}^{m}(N_i\sqcup \overline{N_i})\sqcup (N'\sqcup\overline{N'}) \sqcup N'',
\end{equation}
where $m$ is some non-negative integer, $N_i\subset [n]$ for each $i$, $\{n-1,\bar{n}\}\subset N'$ if $n\in K$ but $n-1\notin K$, $N'=\emptyset$ otherwise, $\{n,\bar{n}\}\subset N''=\overline{N''}$ if $\{n-1,n\}\subset K$, and  $N''=\emptyset$ otherwise. 
We note that either $N'$ or $N''$ is empty for arbitrary subset $K$ of $[n]$.

\begin{example}\label{ex_D_3_orbit_decomp}
Take $n=3$.  
\begin{enumerate}
\item If $K=\{1\}$, then $[3]\sqcup \overline{[3]}=\{1,2\} \sqcup \{\bar2, \bar1\}\sqcup  \{3\}\sqcup \{\bar 3\}$.
\item If $K=\{1,2\}$, then $[3]\sqcup \overline{[3]}=\{1,2,3\} \sqcup \{\bar3, \bar2, \bar1\}$.
\item If $K=\{1,3\}$, then $[3]\sqcup \overline{[3]}=\{1,2, \bar3\} \sqcup \{3, \bar2, \bar1\}$.
\item If $K=\{2,3\}$, then $[3]\sqcup \overline{[3]}=\{1\}\sqcup \{\bar 1\} \sqcup \{2,3,\bar3, \bar2\}$. 
\end{enumerate}
\end{example}

Using the $W_K$-orbit decompositions \eqref{eq:D_orbit_decomposition}, one can check that Lemma \ref{lem_invariants} and Proposition \ref{prop_A_generated_by_deg2} hold for type $D_n$.

\subsection{Partitioned weight polytope of type $D_n$}\label{subsec_D_part_weight_poly}

If $\Phi$ is of type $D_n$, then for every $k \in [n]$, the half-space $\Hk^\leq$ satisfies 
\begin{align*} \label{eq_Hk}
\Hk^{\leq}&=\{ x \in E \mid x_k \leq x_{k+1}\}, \quad  {\rm for} \ 1 \leq k \leq n-1;\\
\Hn^{\leq}&=\{x \in E \mid x_{n-1}\leq x_{\bar n}\},
\end{align*}
and the partitioned weight polytope $P_{D_n}$ associated with a subset $K\subset [n]$ is defined by 
\begin{equation*}
P_{D_n}(K) \colonequals P_{D_n} \cap \bigcap_{k \in K} \Hk^{\leq}.
\end{equation*}
See Figure \ref{fig_part_weight_poly_D} for several examples of partitioned weight polytopes in type $D_3$

\begin{figure}
\begin{tikzpicture}[scale=0.9]

\begin{scope}[scale=0.4, rotate around x=-90, rotate around z=-105, yshift=-350]

\coordinate (132) at (3,1,2);
\coordinate (13b2) at (3,-1,2);
\coordinate (12b3) at (3,-2,1);
\coordinate (12b3b) at (3,-2,-1);
\coordinate (13b2b) at (3,-1,-2);
\coordinate (132b) at (3,1,-2);
\coordinate (12b3) at (3,2,-1);
\coordinate (123) at (3,2,1);

\draw[fill] (3,1,2) circle (1.2pt); 	
\draw[fill] (3,-2,-1) circle (1.2pt); 
\draw[fill] (3,-1,-2) circle (1.2pt); 
\draw[fill] (3,2,1) circle (1.2pt); 	

\draw (132)--(123)--(13b2b)--(12b3b)--cycle;

\coordinate (1b32) at (-3,1,2);
\coordinate (1b3b2) at (-3,-1,2);
\coordinate (1b2b3) at (-3,-2,1);
\coordinate (1b2b3b) at (-3,-2,-1);
\coordinate (1b3b2b) at (-3,-1,-2);
\coordinate (1b32b) at (-3,1,-2);
\coordinate (1b23b) at (-3,2,-1);
\coordinate (1b23) at (-3,2,1);

\draw[fill] (-3,-1,2) circle (1pt); 	
\draw[fill] (-3,-2,1) circle (1pt); 	
\draw[fill] (-3,1,-2) circle (1pt); 	
\draw[fill] (-3,2,-1) circle (1pt); 	

\draw(1b3b2)--(1b23b);
\draw[dotted] (1b3b2)--(1b2b3)--(1b32b)--(1b23b);

\coordinate (312) at (1,3,2);
\coordinate (3b12) at (-1,3,2);
\coordinate (2b13) at (-2,3,1);
\coordinate (2b13b) at (-2,3,-1);
\coordinate (3b12b) at (-1,3,-2);
\coordinate (312b) at (1,3,-2);
\coordinate (213b) at (2,3,-1);
\coordinate (213) at (2,3,1);

\draw[fill] (1,3,2) circle (1.2pt);	
\draw[fill] (-2,3,-1) circle (1.2pt); 
\draw[fill] (-1,3,-2) circle (1.2pt);	
\draw[fill] (2,3,1) circle (1.2pt); 	

\draw (213)--(312)--(2b13b)--(3b12b)--cycle;

\coordinate (31b2) at (1,-3,2);
\coordinate (3b1b2) at (-1,-3,2);
\coordinate (2b1b3) at (-2,-3,1);
\coordinate (2b1b3b) at (-2,-3,-1);
\coordinate (3b1b2b) at (-1,-3,-2);
\coordinate (31b2b) at (1,-3,-2);
\coordinate (21b3b) at (2,-3,-1);
\coordinate (21b3) at (2,-3,1);

\draw[fill] (-1,-3,2) circle (1pt); 	
\draw[fill] (-2,-3,1) circle (1pt); 	
\draw[fill] (1,-3,-2) circle (1pt); 	
\draw[fill] (2,-3,-1) circle (1pt); 	

\draw(21b3b)--(3b1b2);
\draw[dotted] (21b3b)--(31b2b)--(2b1b3)--(3b1b2);

\coordinate (321) at (1,2,3);
\coordinate (32b1) at (-1,-2,3);
\coordinate (3b21) at (-1,2,3);
\coordinate (3b2b1) at (-1,-2,3);
\coordinate (231) at (2,1,3);
\coordinate (23b1) at (2,-1,3);
\coordinate (2b31) at (-2,1,3);
\coordinate (2b3b1) at (-2,-1,3);

\draw[fill] (1,2,3) circle (1.2pt);	
\draw[fill] (-1,-2,3) circle (1.2pt); 
\draw[fill] (2,1,3) circle (1.2pt); 	
\draw[fill] (-2,-1,3) circle (1.2pt);	
\draw (231)--(321)--(2b3b1)--(3b2b1)--cycle;

\coordinate (321b) at (1,2,-3);
\coordinate (32b1b) at (1,-2,-3);
\coordinate (3b21b) at (-1,2,-3);
\coordinate (3b2b1b) at (-1,-2,-3);
\coordinate (231b) at (2,1,-3);
\coordinate (23b1b) at (2,-1,-3);
\coordinate (2b31b) at (-2,1,-3);
\coordinate (2b3b1b) at (-2,-1,-3);

\draw[fill] (1,-2,-3) circle (1pt); 	
\draw[fill] (-1,2,-3) circle (1pt); 	
\draw[fill] (2,-1,-3) circle (1pt); 	
\draw[fill] (-2,1,-3) circle (1pt); 	

\draw (23b1b)--(3b21b);
\draw[dotted] (23b1b)--(32b1b)--(2b31b)--(3b21b);

\draw (1,2,3)--(2,1,3)--(3,1,2)--(3,2,1)--(2,3,1)--(1,3,2)--cycle;

\draw[dotted] (-1,-2,3)--(-2,-1,3)--(-3,-1,2)--(-3,-2,1)--(-2,-3,1)--(-1,-3,2)--cycle;

\draw (21b3b)--(12b3b)--(132)--(231)--(3b2b1)--(3b1b2)--cycle;
\draw (321)--(312)--(2b13b)--(1b23b)--(1b3b2)--(2b3b1)--cycle;

\draw (3b21b)--(3b12b);
\draw[dotted] (2b31b)--(1b32b);
\draw (13b2b)--(23b1b);
\draw[dotted] (31b2b)--(32b1b);

\draw[fill=yellow, opacity=0.5]  ($1/2*(3b2b1)+1/2*(2b3b1)$)--($1/2*(231)+1/2*(321)$)--($1/2*(123)+1/2*(213)$)--($1/2*(23b1b)+1/2*(3b21b)$)--(23b1b)--(13b2b)--(12b3b)--(21b3b)--(3b1b2)--(3b2b1)--cycle;
\draw[thick, blue] ($1/2*(3b2b1)+1/2*(2b3b1)$)--($1/2*(231)+1/2*(321)$)--($1/2*(123)+1/2*(213)$)--($1/2*(23b1b)+1/2*(3b21b)$)--(23b1b)--(13b2b)--(12b3b)--(21b3b)--(3b1b2)--(3b2b1)--cycle;

\draw[ thick, blue, dotted] ($1/2*(3b2b1)+1/2*(2b3b1)$)--($1/2*(2b1b3)+1/2*(1b2b3)$)--($1/2*(32b1b)+1/2*(2b31b)$)--($1/2*(23b1b)+1/2*(3b21b)$);

\draw[ thick, blue, dotted] (3b1b2)--(2b1b3)--($1/2*(2b1b3)+1/2*(1b2b3)$);
\draw[ thick, blue, dotted] (2b1b3)--(31b2b)--(21b3b);
\draw[ thick, blue, dotted] (31b2b)--(32b1b)--(23b1b);
\draw[ thick, blue, dotted] (32b1b)--($1/2*(32b1b)+1/2*(2b31b)$);

\draw[thick, blue] (3b2b1)--(231)--(132)--(12b3b);
\draw[thick, blue] (132)--(123)--(13b2b);
\draw[thick, blue] (123)--($1/2*(123)+1/2*(213)$);
\draw[thick, blue] (231)--($1/2*(231)+1/2*(321)$);

\node at (0,0,-5) {$P_{D_3}(\{1\})$};
\end{scope}


\begin{scope}[scale=0.4, rotate around x=-90, rotate around z=-105, yshift=-350, xshift=250]

\coordinate (132) at (3,1,2);
\coordinate (13b2) at (3,-1,2);
\coordinate (12b3) at (3,-2,1);
\coordinate (12b3b) at (3,-2,-1);
\coordinate (13b2b) at (3,-1,-2);
\coordinate (132b) at (3,1,-2);
\coordinate (12b3) at (3,2,-1);
\coordinate (123) at (3,2,1);

\draw[fill] (3,1,2) circle (1.2pt); 	
\draw[fill] (3,-2,-1) circle (1.2pt); 
\draw[fill] (3,-1,-2) circle (1.2pt); 
\draw[fill] (3,2,1) circle (1.2pt); 	

\draw (132)--(123)--(13b2b)--(12b3b)--cycle;

\coordinate (1b32) at (-3,1,2);
\coordinate (1b3b2) at (-3,-1,2);
\coordinate (1b2b3) at (-3,-2,1);
\coordinate (1b2b3b) at (-3,-2,-1);
\coordinate (1b3b2b) at (-3,-1,-2);
\coordinate (1b32b) at (-3,1,-2);
\coordinate (1b23b) at (-3,2,-1);
\coordinate (1b23) at (-3,2,1);

\draw[fill] (-3,-1,2) circle (1pt); 	
\draw[fill] (-3,-2,1) circle (1pt); 	
\draw[fill] (-3,1,-2) circle (1pt); 	
\draw[fill] (-3,2,-1) circle (1pt); 	

\draw(1b3b2)--(1b23b);
\draw[dotted] (1b3b2)--(1b2b3)--(1b32b)--(1b23b);

\coordinate (312) at (1,3,2);
\coordinate (3b12) at (-1,3,2);
\coordinate (2b13) at (-2,3,1);
\coordinate (2b13b) at (-2,3,-1);
\coordinate (3b12b) at (-1,3,-2);
\coordinate (312b) at (1,3,-2);
\coordinate (213b) at (2,3,-1);
\coordinate (213) at (2,3,1);

\draw[fill] (1,3,2) circle (1.2pt);	
\draw[fill] (-2,3,-1) circle (1.2pt); 
\draw[fill] (-1,3,-2) circle (1.2pt);	
\draw[fill] (2,3,1) circle (1.2pt); 	

\draw (213)--(312)--(2b13b)--(3b12b)--cycle;

\coordinate (31b2) at (1,-3,2);
\coordinate (3b1b2) at (-1,-3,2);
\coordinate (2b1b3) at (-2,-3,1);
\coordinate (2b1b3b) at (-2,-3,-1);
\coordinate (3b1b2b) at (-1,-3,-2);
\coordinate (31b2b) at (1,-3,-2);
\coordinate (21b3b) at (2,-3,-1);
\coordinate (21b3) at (2,-3,1);

\draw[fill] (-1,-3,2) circle (1pt); 	
\draw[fill] (-2,-3,1) circle (1pt); 	
\draw[fill] (1,-3,-2) circle (1pt); 	
\draw[fill] (2,-3,-1) circle (1pt); 	

\draw(21b3b)--(3b1b2);
\draw[dotted] (21b3b)--(31b2b)--(2b1b3)--(3b1b2);

\coordinate (321) at (1,2,3);
\coordinate (32b1) at (-1,-2,3);
\coordinate (3b21) at (-1,2,3);
\coordinate (3b2b1) at (-1,-2,3);
\coordinate (231) at (2,1,3);
\coordinate (23b1) at (2,-1,3);
\coordinate (2b31) at (-2,1,3);
\coordinate (2b3b1) at (-2,-1,3);

\draw[fill] (1,2,3) circle (1.2pt);	
\draw[fill] (-1,-2,3) circle (1.2pt); 
\draw[fill] (2,1,3) circle (1.2pt); 	
\draw[fill] (-2,-1,3) circle (1.2pt);	
\draw (231)--(321)--(2b3b1)--(3b2b1)--cycle;

\coordinate (321b) at (1,2,-3);
\coordinate (32b1b) at (1,-2,-3);
\coordinate (3b21b) at (-1,2,-3);
\coordinate (3b2b1b) at (-1,-2,-3);
\coordinate (231b) at (2,1,-3);
\coordinate (23b1b) at (2,-1,-3);
\coordinate (2b31b) at (-2,1,-3);
\coordinate (2b3b1b) at (-2,-1,-3);

\draw[fill] (1,-2,-3) circle (1pt); 	
\draw[fill] (-1,2,-3) circle (1pt); 	
\draw[fill] (2,-1,-3) circle (1pt); 	
\draw[fill] (-2,1,-3) circle (1pt); 	

\draw (23b1b)--(3b21b);
\draw[dotted] (23b1b)--(32b1b)--(2b31b)--(3b21b);

\draw (1,2,3)--(2,1,3)--(3,1,2)--(3,2,1)--(2,3,1)--(1,3,2)--cycle;

\draw[dotted] (-1,-2,3)--(-2,-1,3)--(-3,-1,2)--(-3,-2,1)--(-2,-3,1)--(-1,-3,2)--cycle;

\draw (21b3b)--(12b3b)--(132)--(231)--(3b2b1)--(3b1b2)--cycle;
\draw (321)--(312)--(2b13b)--(1b23b)--(1b3b2)--(2b3b1)--cycle;

\draw (3b21b)--(3b12b);
\draw[dotted] (2b31b)--(1b32b);
\draw (13b2b)--(23b1b);
\draw[dotted] (31b2b)--(32b1b);

\draw[fill=yellow, opacity=0.5] 
($1/2*(123)+1/2*(213)$)--
($1/2*(231)+1/2*(321)$)--
($1/2*(3b2b1)+1/2*(2b3b1)$)--
(3b2b1)--
($1/2*(3b1b2)+1/2*(3b2b1)$)--
($1/2*(12b3b)+1/2*(132)$)--
($1/2*(13b2b)+1/2*(123)$)--
($1/3*(123)+1/3*(23b1b)+1/3*(3b12b)$)--cycle;

\draw[thick, blue] ($1/2*(123)+1/2*(213)$)--($1/2*(231)+1/2*(321)$)--($1/2*(3b2b1)+1/2*(2b3b1)$)--(3b2b1)--($1/2*(3b1b2)+1/2*(3b2b1)$)--($1/2*(12b3b)+1/2*(132)$)--($1/2*(13b2b)+1/2*(123)$)--($1/3*(123)+1/3*(23b1b)+1/3*(3b12b)$)--cycle;

\draw[thick, blue, dotted] ($1/2*(3b1b2)+1/2*(3b2b1)$)--($1/2*(3b1b2)+1/2*(1b3b2)$)--($1/2*(3b2b1)+1/2*(2b3b1)$);
\draw[thick, blue, dotted] ($1/2*(3b1b2)+1/2*(1b3b2)$)--($1/3*(123)+1/3*(23b1b)+1/3*(3b12b)$);

\draw[thick, blue] (3b2b1)--(231)--($1/2*(231)+1/2*(321)$);
\draw[thick, blue] (231)--(132)--(123)--($1/2*(123)+1/2*(213)$);
\draw[thick, blue] (123)--($1/2*(13b2b)+1/2*(123)$);
\draw[thick, blue] (132)--($1/2*(12b3b)+1/2*(132)$);

\node at (0,0,-5) {$P_{D_3}(\{1,3\})$};
\end{scope}

\begin{scope}[scale=0.4, rotate around x=-90, rotate around z=-105, yshift=-350, xshift=500]

\coordinate (132) at (3,1,2);
\coordinate (13b2) at (3,-1,2);
\coordinate (12b3) at (3,-2,1);
\coordinate (12b3b) at (3,-2,-1);
\coordinate (13b2b) at (3,-1,-2);
\coordinate (132b) at (3,1,-2);
\coordinate (12b3) at (3,2,-1);
\coordinate (123) at (3,2,1);

\draw[fill] (3,1,2) circle (1.2pt); 	
\draw[fill] (3,-2,-1) circle (1.2pt); 
\draw[fill] (3,-1,-2) circle (1.2pt); 
\draw[fill] (3,2,1) circle (1.2pt); 	

\draw (132)--(123)--(13b2b)--(12b3b)--cycle;

\coordinate (1b32) at (-3,1,2);
\coordinate (1b3b2) at (-3,-1,2);
\coordinate (1b2b3) at (-3,-2,1);
\coordinate (1b2b3b) at (-3,-2,-1);
\coordinate (1b3b2b) at (-3,-1,-2);
\coordinate (1b32b) at (-3,1,-2);
\coordinate (1b23b) at (-3,2,-1);
\coordinate (1b23) at (-3,2,1);

\draw[fill] (-3,-1,2) circle (1pt); 	
\draw[fill] (-3,-2,1) circle (1pt); 	
\draw[fill] (-3,1,-2) circle (1pt); 	
\draw[fill] (-3,2,-1) circle (1pt); 	

\draw(1b3b2)--(1b23b);
\draw[dotted] (1b3b2)--(1b2b3)--(1b32b)--(1b23b);

\coordinate (312) at (1,3,2);
\coordinate (3b12) at (-1,3,2);
\coordinate (2b13) at (-2,3,1);
\coordinate (2b13b) at (-2,3,-1);
\coordinate (3b12b) at (-1,3,-2);
\coordinate (312b) at (1,3,-2);
\coordinate (213b) at (2,3,-1);
\coordinate (213) at (2,3,1);

\draw[fill] (1,3,2) circle (1.2pt);	
\draw[fill] (-2,3,-1) circle (1.2pt); 
\draw[fill] (-1,3,-2) circle (1.2pt);	
\draw[fill] (2,3,1) circle (1.2pt); 	

\draw (213)--(312)--(2b13b)--(3b12b)--cycle;

\coordinate (31b2) at (1,-3,2);
\coordinate (3b1b2) at (-1,-3,2);
\coordinate (2b1b3) at (-2,-3,1);
\coordinate (2b1b3b) at (-2,-3,-1);
\coordinate (3b1b2b) at (-1,-3,-2);
\coordinate (31b2b) at (1,-3,-2);
\coordinate (21b3b) at (2,-3,-1);
\coordinate (21b3) at (2,-3,1);

\draw[fill] (-1,-3,2) circle (1pt); 	
\draw[fill] (-2,-3,1) circle (1pt); 	
\draw[fill] (1,-3,-2) circle (1pt); 	
\draw[fill] (2,-3,-1) circle (1pt); 	

\draw(21b3b)--(3b1b2);
\draw[dotted] (21b3b)--(31b2b)--(2b1b3)--(3b1b2);

\coordinate (321) at (1,2,3);
\coordinate (32b1) at (-1,-2,3);
\coordinate (3b21) at (-1,2,3);
\coordinate (3b2b1) at (-1,-2,3);
\coordinate (231) at (2,1,3);
\coordinate (23b1) at (2,-1,3);
\coordinate (2b31) at (-2,1,3);
\coordinate (2b3b1) at (-2,-1,3);

\draw[fill] (1,2,3) circle (1.2pt);	
\draw[fill] (-1,-2,3) circle (1.2pt); 
\draw[fill] (2,1,3) circle (1.2pt); 	
\draw[fill] (-2,-1,3) circle (1.2pt);	
\draw (231)--(321)--(2b3b1)--(3b2b1)--cycle;

\coordinate (321b) at (1,2,-3);
\coordinate (32b1b) at (1,-2,-3);
\coordinate (3b21b) at (-1,2,-3);
\coordinate (3b2b1b) at (-1,-2,-3);
\coordinate (231b) at (2,1,-3);
\coordinate (23b1b) at (2,-1,-3);
\coordinate (2b31b) at (-2,1,-3);
\coordinate (2b3b1b) at (-2,-1,-3);

\draw[fill] (1,-2,-3) circle (1pt); 	
\draw[fill] (-1,2,-3) circle (1pt); 	
\draw[fill] (2,-1,-3) circle (1pt); 	
\draw[fill] (-2,1,-3) circle (1pt); 	

\draw (23b1b)--(3b21b);
\draw[dotted] (23b1b)--(32b1b)--(2b31b)--(3b21b);

\draw (1,2,3)--(2,1,3)--(3,1,2)--(3,2,1)--(2,3,1)--(1,3,2)--cycle;

\draw[dotted] (-1,-2,3)--(-2,-1,3)--(-3,-1,2)--(-3,-2,1)--(-2,-3,1)--(-1,-3,2)--cycle;

\draw (21b3b)--(12b3b)--(132)--(231)--(3b2b1)--(3b1b2)--cycle;
\draw (321)--(312)--(2b13b)--(1b23b)--(1b3b2)--(2b3b1)--cycle;

\draw (3b21b)--(3b12b);
\draw[dotted] (2b31b)--(1b32b);
\draw (13b2b)--(23b1b);
\draw[dotted] (31b2b)--(32b1b);

\draw[fill=yellow, opacity=0.5] ($1/2*(123)+1/2*(132)$)--($1/2*(123)+1/2*(12b3b)$)--($1/2*(123)+1/2*(13b2b)$)--($1/2*(3b21b)+1/2*(3b12b)$)--(3b12b)--(2b13b)--(1b23b)--($1/2*(1b23b)+1/2*(1b3b2)$)--($1/2*(312)+1/2*(321)$)--cycle;

\draw[blue, thick] ($1/2*(123)+1/2*(132)$)--($1/2*(123)+1/2*(12b3b)$)--($1/2*(123)+1/2*(13b2b)$)--($1/2*(3b21b)+1/2*(3b12b)$)--(3b12b)--(2b13b)--(1b23b)--($1/2*(1b23b)+1/2*(1b3b2)$)--($1/2*(312)+1/2*(321)$)--cycle;

\draw[blue, thick] ($1/2*(123)+1/2*(132)$)--(123)--(213)--(312)--($1/2*(312)+1/2*(321)$);
\draw[blue, thick] (123)--($1/2*(123)+1/2*(13b2b)$);
\draw[blue, thick] (213)--(3b12b);
\draw[blue, thick] (312)--(2b13b);

\draw[blue, thick, dotted] ($1/2*(1b3b2)+1/2*(1b23b)$)--($1/2*(1b32b)+1/2*(1b3b2)$)--($1/2*(123)+1/2*(12b3b)$);
\draw[blue, thick, dotted] ($1/2*(1b32b)+1/2*(1b3b2)$)--($1/2*(1b32b)+1/2*(1b23b)$)--(1b23b);
\draw[blue, thick, dotted] ($1/2*(3b21b)+1/2*(3b12b)$)--($1/2*(1b32b)+1/2*(1b23b)$);

%
%
%

\node at (0,0,-5) {$P_{D_3}(\{2,3\})$};
\end{scope}


\begin{scope}[scale=0.4, rotate around x=-90, rotate around z=-105, yshift=-350, xshift=750]

\coordinate (132) at (3,1,2);
\coordinate (13b2) at (3,-1,2);
\coordinate (12b3) at (3,-2,1);
\coordinate (12b3b) at (3,-2,-1);
\coordinate (13b2b) at (3,-1,-2);
\coordinate (132b) at (3,1,-2);
\coordinate (12b3) at (3,2,-1);
\coordinate (123) at (3,2,1);

\draw[fill] (3,1,2) circle (1.2pt); 	
\draw[fill] (3,-2,-1) circle (1.2pt); 
\draw[fill] (3,-1,-2) circle (1.2pt); 
\draw[fill] (3,2,1) circle (1.2pt); 	

\draw (132)--(123)--(13b2b)--(12b3b)--cycle;

\coordinate (1b32) at (-3,1,2);
\coordinate (1b3b2) at (-3,-1,2);
\coordinate (1b2b3) at (-3,-2,1);
\coordinate (1b2b3b) at (-3,-2,-1);
\coordinate (1b3b2b) at (-3,-1,-2);
\coordinate (1b32b) at (-3,1,-2);
\coordinate (1b23b) at (-3,2,-1);
\coordinate (1b23) at (-3,2,1);

\draw[fill] (-3,-1,2) circle (1pt); 	
\draw[fill] (-3,-2,1) circle (1pt); 	
\draw[fill] (-3,1,-2) circle (1pt); 	
\draw[fill] (-3,2,-1) circle (1pt); 	

\draw(1b3b2)--(1b23b);
\draw[dotted] (1b3b2)--(1b2b3)--(1b32b)--(1b23b);

\coordinate (312) at (1,3,2);
\coordinate (3b12) at (-1,3,2);
\coordinate (2b13) at (-2,3,1);
\coordinate (2b13b) at (-2,3,-1);
\coordinate (3b12b) at (-1,3,-2);
\coordinate (312b) at (1,3,-2);
\coordinate (213b) at (2,3,-1);
\coordinate (213) at (2,3,1);

\draw[fill] (1,3,2) circle (1.2pt);	
\draw[fill] (-2,3,-1) circle (1.2pt); 
\draw[fill] (-1,3,-2) circle (1.2pt);	
\draw[fill] (2,3,1) circle (1.2pt); 	

\draw (213)--(312)--(2b13b)--(3b12b)--cycle;

\coordinate (31b2) at (1,-3,2);
\coordinate (3b1b2) at (-1,-3,2);
\coordinate (2b1b3) at (-2,-3,1);
\coordinate (2b1b3b) at (-2,-3,-1);
\coordinate (3b1b2b) at (-1,-3,-2);
\coordinate (31b2b) at (1,-3,-2);
\coordinate (21b3b) at (2,-3,-1);
\coordinate (21b3) at (2,-3,1);

\draw[fill] (-1,-3,2) circle (1pt); 	
\draw[fill] (-2,-3,1) circle (1pt); 	
\draw[fill] (1,-3,-2) circle (1pt); 	
\draw[fill] (2,-3,-1) circle (1pt); 	

\draw(21b3b)--(3b1b2);
\draw[dotted] (21b3b)--(31b2b)--(2b1b3)--(3b1b2);

\coordinate (321) at (1,2,3);
\coordinate (32b1) at (-1,-2,3);
\coordinate (3b21) at (-1,2,3);
\coordinate (3b2b1) at (-1,-2,3);
\coordinate (231) at (2,1,3);
\coordinate (23b1) at (2,-1,3);
\coordinate (2b31) at (-2,1,3);
\coordinate (2b3b1) at (-2,-1,3);

\draw[fill] (1,2,3) circle (1.2pt);	
\draw[fill] (-1,-2,3) circle (1.2pt); 
\draw[fill] (2,1,3) circle (1.2pt); 	
\draw[fill] (-2,-1,3) circle (1.2pt);	
\draw (231)--(321)--(2b3b1)--(3b2b1)--cycle;

\coordinate (321b) at (1,2,-3);
\coordinate (32b1b) at (1,-2,-3);
\coordinate (3b21b) at (-1,2,-3);
\coordinate (3b2b1b) at (-1,-2,-3);
\coordinate (231b) at (2,1,-3);
\coordinate (23b1b) at (2,-1,-3);
\coordinate (2b31b) at (-2,1,-3);
\coordinate (2b3b1b) at (-2,-1,-3);

\draw[fill] (1,-2,-3) circle (1pt); 	
\draw[fill] (-1,2,-3) circle (1pt); 	
\draw[fill] (2,-1,-3) circle (1pt); 	
\draw[fill] (-2,1,-3) circle (1pt); 	

\draw (23b1b)--(3b21b);
\draw[dotted] (23b1b)--(32b1b)--(2b31b)--(3b21b);

\draw (1,2,3)--(2,1,3)--(3,1,2)--(3,2,1)--(2,3,1)--(1,3,2)--cycle;

\draw[dotted] (-1,-2,3)--(-2,-1,3)--(-3,-1,2)--(-3,-2,1)--(-2,-3,1)--(-1,-3,2)--cycle;

\draw (21b3b)--(12b3b)--(132)--(231)--(3b2b1)--(3b1b2)--cycle;
\draw (321)--(312)--(2b13b)--(1b23b)--(1b3b2)--(2b3b1)--cycle;

\draw (3b21b)--(3b12b);
\draw[dotted] (2b31b)--(1b32b);
\draw (13b2b)--(23b1b);
\draw[dotted] (31b2b)--(32b1b);

\draw[fill=yellow, opacity=0.5] ($1/2*(123)+1/2*(321)$)--($1/2*(123)+1/2*(213)$)--($1/3*(123)+1/3*(3b12b)+1/3*(23b1b)$)--($1/2*(123)+1/2*(13b2b)$)--($1/2*(12b3b)+1/2*(123)$)--($1/2*(123)+1/2*(132)$)--cycle;
\draw[thick, blue] ($1/2*(123)+1/2*(321)$)--($1/2*(123)+1/2*(213)$)--($1/3*(123)+1/3*(3b12b)+1/3*(23b1b)$)--($1/2*(123)+1/2*(13b2b)$)--($1/2*(12b3b)+1/2*(123)$)--($1/2*(123)+1/2*(132)$)--cycle;

\draw[dotted, blue, thick] ($1/2*(123)+1/2*(321)$)--($1/3*(2b1b3)+1/3*(32b1b)+1/3*(1b32b)$);
\draw[dotted, blue, thick] ($1/3*(2b1b3)+1/3*(32b1b)+1/3*(1b32b)$)--($1/2*(12b3b)+1/2*(123)$);
\draw[dotted, blue, thick] ($1/3*(2b1b3)+1/3*(32b1b)+1/3*(1b32b)$)--($1/3*(123)+1/3*(3b12b)+1/3*(23b1b)$);  

\draw[thick, blue] (123)--($1/2*(123)+1/2*(132)$);
\draw[thick, blue] (123)--($1/2*(123)+1/2*(213)$);
\draw[thick, blue] (123)--($1/2*(123)+1/2*(13b2b)$);

\node at (0,0,-5) {$P_{D_3}(\{1,2,3\})$};
\end{scope}

\end{tikzpicture}
\caption{Examples of partitioned weight polytopes in type $D_3$.}
\label{fig_part_weight_poly_D}
\end{figure}
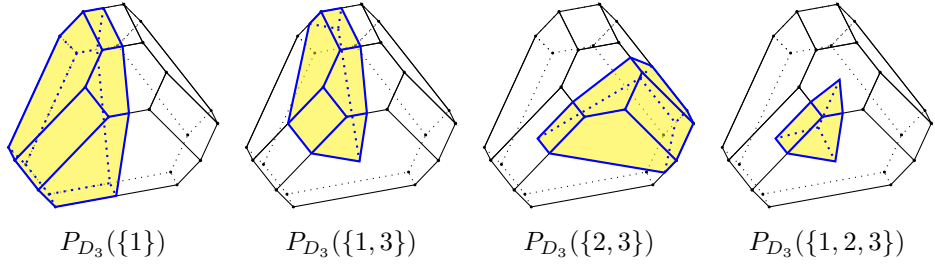

The next  proposition is the type $D_n$ analogue of Proposition \ref{prop_A_position_F(I)}.  

\begin{proposition}\label{prop_D_position_F(I)}
For each $k\in [n]$ and $I\in \mathfrak{M}_{D_n}$, the following claims hold. 
\begin{enumerate}
\item \label{claim_D_1} The facet $F(\indI)$ intersects $\Hk$ nontrivially if and only if $\indI$ is $s_k$-invariant. In this case, $F(\indI)$ is bisected by the hyperplane $\Hk$. 
\item \label{claim_D_2} The facet $F(\indI)$ is contained in the half-space $\Hk^<$ if and only if $\indI$ satisfies
\begin{enumerate}
\item $k\in I$ but  $k+1\notin I$ or  $\overline{k+1}\in I$ but $\bar{k} \notin I$, for $1\leq k < n$;
\item $n\in I$ but $\overline{n-1} \notin I$ or $n-1\in I$ but $\bar n \notin I$, for $k=n$. 
\end{enumerate}
\end{enumerate}
\end{proposition}

Motivated by this proposition, we make the following definition. 

\begin{definition}\label{def_D_K-lower}
For a subset $K\subset [n]$, we define $\M_{D_n}(K)$ to be the set of elements $I\in \M_{D_n}$ such that $I\cap N$ is a lower subset of $N$ with respect to the partial order  
\begin{equation*}
\begin{tikzcd}[column sep=-0.5ex, row sep=-2.0ex]
&&&&&&&&n&&&&&&&&\\
&&&&&&& \reflectbox{\rotatebox[origin=c]{150}{$\prec$}}
 &&  \reflectbox{\rotatebox[origin=c]{210}{$\prec$}}  &&&&&&&\\
\ \ \ 1&\prec&\cdots&&&\prec&n-1&&& &\overline{n-1}&\prec&&& \cdots& \prec &\bar 1.\\
&&&&&&&  \reflectbox{\rotatebox[origin=c]{30}{$\succ$}}  &&  \reflectbox{\rotatebox[origin=c]{330}{$\succ$}}  &&&&&&&\\
&&&&&&&&\bar n&&&&&&&&
\end{tikzcd} 
\end{equation*}
for every disjoint component $N$ appearing in the $W_K$-orbit decomposition \eqref{eq:D_orbit_decomposition}. 
\end{definition}

\begin{example}
For the cases in Example \ref{ex_D_3_orbit_decomp}, we have  
\begin{enumerate}
\item $\M_{D_3}(\{1\})=\left\{ \{1\}, \{\bar 2\}, \{3\}, \{\bar 3\}, \{1,2,3\}, \{1,2,\bar3\}, \{3,\bar2, \bar1\}, \{\bar3, \bar2, \bar1\}\right\}.$
\item $\M_{D_3}(\{1,2\})=\{\{1\}, \{\bar3\}, \{1,2,3\}, \{1,2,\bar3\}, \{\bar3, \bar2, \bar1\}, \{1, \bar3, \bar2\} \}. $
\item $\M_{D_3}(\{1,3\})=\{\{1\}, \{3\}, \{1,2,\bar3\}, \{1,2,3\}, \{3, \bar2, \bar1\}, \{1, 3, \bar2\} \}. $
\item $\M_{D_3}(\{2,3\})=\{ \{1\}, \{\bar 1\}, \{2\}, \{1,2,3\}, \{1,2,\bar3\}, \{2,3, \bar1\}, \{2, \bar 3, \bar 1\} \}$.
\end{enumerate}
\end{example}

With the definition of $\M_{D_n}(K)$ above, Corollary~\ref{cor_A_facet_characterization_of_Part_perm} and  Proposition~\ref{prop_A_flagness_part_w_poly} hold for type $D_n$. Proposition~\ref{prop_A_intersection_facets_PK} also holds for type $D_n$ if we add the condition $|I\cap J|\le n-2$ to the condition $I\not\subset J$ and $J\not\subset I$ in (1) of Proposition~\ref{prop_A_intersection_facets_PK}.  Similarly, Proposition~\ref{prop_A_cohomologyXK} also holds for type $D_n$ with 
\[
\alpha^{\v}_k =\begin{cases}
e_k-e_{k+1} & 1\leq k < n,\\
e_{n-1}+e_n & k=n
\end{cases}
\]
if we add the condition $|I\cap J|\le n-2$ to the condition $I\not\subset J$ and $J\not\subset I$ in (2) of Proposition~\ref{prop_A_cohomologyXK}.

\subsection{Proof of Theorem \ref{main} for type $D$}\label{subsec_D_proof_of_main_thm}

In analogy with our argument in type $A_{n-1}$,  we can write
\begin{align*} \label{eq_D_coefficientsbeta}
e_\indI - e_{v(\indI)}=\sum_{k \in K} c^{\indI,v}_k \alpha^{\v}_k 
\end{align*}
with each $c_k^{I,v}$ a non-negative integer, for arbitrary $\indI \in \M_{D_n}(K)$ and $v \in W_K$. 
The following lemma is the type $D_n$ analogue of Lemma~\ref{lem_A_coeff_nonzero}. 
We maintain the notation $\widehat{[k]}$ from \eqref{notation_khat}. 

\begin{lemma} \label{lemma:Lemma4.4TypeD}
Let $I\in \M_{D_n}(K)$, $v\in W_K$ and consider the $W_K$-decomposition \eqref{eq:D_orbit_decomposition}.  Then $c_k^{I,v}=0$ $(k\in K)$ in the following cases: 
\begin{enumerate}
\item When $k\in N_i$ ($k\not=n$ by definition of $N_i$ since $k\in K$),   
\[
\begin{split}
&v(\indI) \cap N_i \subset [k] \cap N_i \ {\rm or} \ v(\indI) \cap N_i \supset [k] \cap N_i, \text{ and}\\ 
&v(\indI) \cap \overline{N_i} \subset \widehat{[k+1]}  \cap \overline{N_i} \ {\rm or} \ v(\indI) \cap \overline{N_i} \supset \widehat{[k+1]} \cap \overline{N_i} .
\end{split}
\]
\item When $k\in N'$ ($k\not=n$ by definition of $N'$), 
\begin{align*} 
&v(\indI) \cap N' \subset [k] \cap N' \ {\rm or} \ v(\indI) \cap N' \supset [k] \cap N', \text{ and} \\
&\begin{array}{ll}
v(\indI) \cap \overline{N'} \subset \big((\widehat{[k+1]} \setminus \{\bar{n} \}) \cup \{n \}\big) \cap \overline{N'} \ {\rm or} \\
v(\indI) \cap \overline{N'} \supset 
\big((\widehat{[k+1]} \setminus \{\bar{n} \}) \cup \{n \})\big) \cap \overline{N'}.
\end{array} 
\end{align*}
\item When $k\in \overline{N'}$ ($k=n$ by definition of $\overline{N'}$), 
\begin{align*} 
&v(\indI) \cap N' \subset [n-1] \cap N' \ {\rm or} \ v(\indI) \cap N' \supset [n-1] \cap N',\text{ and} \\
&v(\indI) \cap \overline{N'} \subset \{n\} \cap \overline{N'} \ {\rm or} \ v(\indI) \cap \overline{N'} \supset \{n\} \cap \overline{N'}. 
\end{align*}
\item  When $k \in N''$, 
\begin{enumerate}
\item[(i)] if $k \leq n-2$, 
\begin{equation*}
v(\indI) \cap N'' \subset [k] \cap N'' \ {\rm or} \ v(\indI) \cap N'' \supset [k] \cap N'',
\end{equation*}
\item[(ii)] if $k=n-1$, 
\begin{align*}
\begin{split}
&v(\indI) \cap N'' \subset ([n-1] \cup \{\bar{n} \}) \cap N'' \ {\rm or} \ v(\indI) \cap N'' \supset [n-1] \cap N'' \ {\rm or} \ \\
&v(\indI) \cap N'' = \big(([n-1] \setminus \{i \}) \cup \{\bar{i} \}\cup \{\bar{n} \}\big) \cap N'' \ {\rm for \ some} \ i\in [n-1]\cap N'',
\end{split}
\end{align*}
\item[(iii)] if $k=n$, 
\begin{align*}  
\begin{split}
&v(\indI) \cap N'' \subset [n] \cap N'' \ {\rm or} \ v(\indI) \cap N'' \supset [n-1] \cap N'' \ {\rm or} \ \\ 
&v(\indI) \cap N'' = \big(([n] \setminus \{i \}) \cup \{\bar{i} \}\big) \cap N'' \ {\rm for \ some } \ i\in [n-1]\cap N''.
\end{split}
\end{align*}
\end{enumerate}
\end{enumerate}
\end{lemma}

With these modifications, the argument developed for type $A_{n-1}$ after Lemma~\ref{lem_A_coeff_nonzero} works for type $D_n$.

\end{document}